\documentclass[leqno, 10pt, a4paper]{amsart}
\usepackage{amsmath}
\usepackage{amssymb, latexsym, mathrsfs, a4wide}

 \newtheorem{definition}{Definition}[section]
 \newtheorem{theorem}[definition]{Theorem}
 \newtheorem{lemma}[definition]{Lemma}
 \newtheorem{proposition}[definition]{Proposition}
 \newtheorem{corollary}[definition]{Corollary}

 \newtheorem*{theorem*}{Theorem}
\newtheorem*{proposition*}{Proposition}
\newtheorem*{lemma*}{Lemma}

\theoremstyle{definition}

 \theoremstyle{remark}
 \newtheorem{example}[definition]{Example}
 \newtheorem{remark}[definition]{Remark}

   \newtheorem{claim}{Claim}
   \newtheorem*{claim*}{Claim}
\newcommand{\op}[1]{\operatorname{#1}}

\newcommand{\acou}[2]{\ensuremath{\left\langle #1 , #2 \right\rangle}}

\newcommand{\sacou}[2]{\ensuremath{{}_{\sigma}\!\left\langle #1 , #2 \right\rangle}}
\newcommand{\sacous}[2]{\ensuremath{{}_{\sigma}\! \left\langle #1 , #2 \right\rangle_{\sigma}}} 
\newcommand{\acous}[2]{\ensuremath{\left\langle #1 , #2 \right\rangle_{\sigma}}}
\newcommand{\acoup}[2]{\ensuremath{\left(#1,#2\right)}}
\newcommand{\acoupd}[2]{\ensuremath{\left(#1,#2\right)_{D,\sigma}}}
\newcommand{\sacoup}[2]{\ensuremath{{}_{\sigma}\!\left( #1 , #2 \right)}} 
\newcommand{\sacoupa}[2]{\ensuremath{{}_{\sigma_{1}}\!\left( #1 , #2 \right)}} 
\newcommand{\sacoupb}[2]{\ensuremath{{}_{\sigma_{2}}\!\left( #1 , #2 \right)}} 
\newcommand{\sacoups}[2]{\ensuremath{{}_{\sigma}\!\left( #1 , #2 \right)_{\sigma}}} 
\newcommand{\acoups}[2]{\ensuremath{\left(#1 , #2 \right)_{\sigma}}}
\newcommand{\acoupsa}[2]{\ensuremath{\left(#1 , #2 \right)_{\sigma_{1}}}}
\newcommand{\acoupsb}[2]{\ensuremath{\left(#1 , #2 \right)_{\sigma_{2}}}}

\def\XXint#1#2#3{{\setbox0=\hbox{$#1{#2#3}{\int}$}
\vcenter{\hbox{$#2#3$}}\kern-.5\wd0}}

\newcommand{\ind}{\op{ind}}
\newcommand{\Ch}{\op{Ch}}

\newcommand{\Ad}{\op{Ad}}

\newcommand{\GL}{\op{GL}}

\newcommand{\C}{\ensuremath{\mathbb{C}}} 
 
\newcommand{\N}{\ensuremath{\mathbb{N}}} 
\newcommand{\Q}{\ensuremath{\mathbb{Q}}} 
\newcommand{\R}{\ensuremath{\mathbb{R}}} 
 
\newcommand{\Z}{\ensuremath{\mathbb{Z}}}

\newcommand{\fs}{\ensuremath{\mathfrak{s}}}

\newcommand{\Ca}[1]{\ensuremath{\mathcal{#1}}}
\newcommand{\cA}{\Ca{A}}
\newcommand{\cB}{\Ca{B}}

\newcommand{\cE}{\Ca{E}}
\newcommand{\cF}{\ensuremath{\mathcal{F}}}

\newcommand{\cH}{\ensuremath{\mathcal{H}}}

\newcommand{\cL}{\ensuremath{\mathcal{L}}}

\newcommand{\cS}{\ensuremath{\mathcal{S}}}

\newcommand{\sD}{\ensuremath{{/\!\!\!\!D}}}%{/{\!\!\!\!D}}
\newcommand{\sS}{\ensuremath{{/\!\!\!\!\!\;S}}}

\newcommand{\Gammab}{\Gamma\backslash}

\newcommand{\im}{\op{Im}}
\newcommand{\coker}{\op{coker}} 

\newcommand{\Hom}{\op{Hom}}

\newcommand{\End}{\ensuremath{\op{End}}}

\newcommand{\dom}{\op{dom}}

\newcommand{\opp}{\textup{o}}

\numberwithin{equation}{section}

\begin{document}

\title{Noncommutative Geometry and Conformal Geometry. III. Vafa-Witten Inequality and Poincar\'e Duality}
 \author{Rapha\"el Ponge}
 \address{Department of Mathematical Sciences, Seoul National University, Seoul, South Korea}
 \email{ponge.snu@gmail.com}
 \author{Hang Wang}
% \address{Mathematical Sciences Center, Tsinghua University, Beijing, China}
% \email{hwang@math.tsinghua.edu.cn}
 \address{School of Mathematical Sciences, University of Adelaide, Adelaide, Australia}
 \email{hang.wang01@adelaide.edu.au}

  \thanks{R.P.\ was partially supported by International Faculty Research Fund of Seoul National University and 
  Basic Research Grant 2013R1A1A2008802 of National Research Foundation of Korea.}
 
\begin{abstract}
 This paper is the third part of a series of papers whose aim is to use the
 framework of \emph{twisted spectral triples} to study conformal geometry from a noncommutative geometric viewpoint. In this paper we 
reformulate the inequality of Vafa-Witten~\cite{VW:CMP84} in the setting of twisted spectral triples. This involves a notion of Poincar\'e duality for twisted spectral triples. Our main results have various consequences. In particular, we obtain a version in conformal geometry of the original inequality of Vafa-Witten, in the sense of an explicit control of the Vafa-Witten bound under conformal changes of metrics. This result has several noncommutative manifestations 
for conformal deformations of ordinary spectral triples, spectral triples associated with conformal weights on noncommutative tori, and spectral triples associated with duals of torsion-free discrete cocompact subgroups  satisfying the Baum-Connes conjecture.
\end{abstract}

\maketitle 

\section{Introduction}
This paper is the third part of a series of papers initiated in~\cite{PW:NCGCG1, PW:NCGCG2}. The goal of this series is to use the recent framework of twisted spectral triples introduced by Connes-Moscovici~\cite{CM:TGNTQF} to study conformal geometry from a noncommutative geometric viewpoint. In this paper we 
reformulate the inequality of Vafa-Witten~\cite{VW:CMP84} in the setting of twisted spectral triples. This has various 
geometric applications, including a version of Vafa-Witten's inequality in conformal geometry. 

Given a compact spin Riemannian manifold $M$, the inequality of Vafa-Witten~\cite{VW:CMP84} provides us with a uniform bound $C>0$ 
such that, for any Hermitian vector bundle $E$ over $M$ and any Hermitian connection $\nabla^{E}$ on $E$, 
we have 
\begin{equation}
    \left|\lambda_{1}\left( \sD_{\nabla^{E}}\right) \right| \leq C,
    \label{eq:Intro.Vafa-Witten-inequality}
\end{equation}where $\lambda_{1}\left( \sD_{\nabla^{E}}\right)$ is the eigenvalue of the coupled Dirac 
operator $\sD_{\nabla^{E}}$ with the smallest absolute value. It is a remarkable fact that the Vafa-Witten bound $C$  is totally 
independent of the bundle and connection data. It should also be mentioned that this inequality does not hold for the 
connection Laplacian $\left( \nabla^{E}\right)^{*}\nabla^{E}$ (see~\cite{At:EDO}). 

The arguments of Vafa-Witten combine the max-min principle with clever manipulations on the index theorems of 
Atiyah-Singer~\cite{AS:IEO3} and Atiyah-Patodi-Singer~\cite{APS:SARG3}. An important step is the construction of an auxiliary Hermitian vector bundles $\cF$ such that 
the equation $\sD_{\nabla^{\cE}\otimes \nabla^{\cF}}u=0$ has nontrivial solutions.
While Vafa and Witten constructed these vector bundles by pulling back the Bott element from spheres, Moscovici~\cite{Mo:EIPDNCG} observed 
that this aspect of Vafa-Witten's argument was actually a manifestation of Poincar\'e duality. 
Elaborating on this observation, he extended Vafa-Witten's inequality to 
the framework of Connes' noncommutative geometry. More precisely, he proved 
the inequality for noncommutative spaces (a.k.a.\ spectral triples) that satisfy some version of Poincar\'e 
duality. As a result, Vafa-Witten's inequality holds in fairly great generality. In particular, it holds on Lipschitz 
manifolds, duals of (torsion free) discrete cocompact subgroups of $Sp(n,1)$, 
$SL(3,\R)$, $SL(3,\C)$ and  rank $1$ real Lie groups (including $SO(n,1)$ and $SU(n,1)$), spectral triples over 
noncommutative tori~\cite{Co:NCG}, quantum complex projective lines~\cite{DL:GQPS}, Podle\'s quantum 
spheres~\cite{DS:DOSPSQS, Po:QS, Wa:NCSGSPSIC}, and spectral triples describing the standard model of particle 
physics~\cite{Co:NCGR, Co:GCMFNCG, CCM:GSTNM}.

The aim of this paper is to define Poincar\'e duality and establish Vafa-Witten inequalities for twisted spectral 
triples in the sense of~\cite{CM:TGNTQF}, that is, in the setting of type III 
noncommutative geometry. The axioms satisfied by a twisted spectral triple $\left(\cA,\cH,D\right)_{\sigma}$ are almost identical to the 
usual axioms for an ordinary spectral triple up to the ``twist'' of replacing the boundedness of commutators $[D,a]$, $a \in \cA$, by that 
of twisted commutators,
\begin{equation*}
    [D,a]_{\sigma}:=Da-\sigma(a)D, \qquad a\in \cA,
\end{equation*}where $\sigma$ is a given automorphism of the algebra $\cA$. Examples of twisted spectral triples 
include conformal 
deformations of spectral triples, crossed-products of spin manifolds with arbitrary groups of conformal 
diffeomorphisms, twistings by scaling automorphisms, and 
spectral triples over noncommutative tori associated with conformal weights
(see~\cite{CM:TGNTQF, CT:GBTNC2T, Mo:LIFTST} and Section~\ref{sec:TwistedST}).

As explained in this paper, the conformal deformations of spectral triples and the construction of twisted spectral triples 
over noncommutative tori associated with conformal weights both fit into the framework of 
\emph{pseudo-inner twistings} of ordinary spectral triples as defined in Section~\ref{sec:Pseudo-inner twistings}. The 
class of pseudo-inner twisted spectral triples provides us with the main examples of twisted spectral triples for which 
the results of this 
paper apply. For instance, up to unitary equivalence, a conformal change of metric 
in a Dirac spectral triple amounts to a pseudo-inner twisting by the square root of the conformal factor  (see 
Proposition~\ref{prop:ConformalChangeDiracST}). 

Similar to ordinary spectral triples, the datum of a twisted spectral triple $\left(\cA,\cH,D\right)_{\sigma}$ gives rise 
to an additive index map $\ind_{D,\sigma}:K_{0}(\cA)\rightarrow\frac{1}{2} \Z$ (see \cite{CM:TGNTQF, PW:Index} and Section~\ref{sec:IndexMapTST}). 
Using this index map there is no difficulty to define Poincar\'e duality for twisted 
spectral triples 
(cf.\ Definition~\ref{def:TSTRationalPoincareDual}). Such a duality occurs on pseudo-inner twistings of ordinary spectral triples satisfying Poincar\'e duality in the sense of ordinary 
spectral triples 
(Proposition~\ref{Prop:TSTPD}). In particular, we see that some twisted spectral triples naturally appear as Poincar\'e 
duals of ordinary spectral triples. 

The main result of this paper is a version of Vafa-Witten inequality for twisted spectral triples 
satisfying Poincar\'e duality in the sense of twisted spectral triples (Theorem~\ref{thm:EigenvalueInequality}). This 
version of Vafa-Witten inequality holds for pseudo-inner twistings of ordinary spectral triples satisfying Poincar\'e 
duality. Furthermore, in this case we are able to give an explicit control of the Vafa-Witten bound in terms of the pseudo-inner twisting 
(Theorem~\ref{eq:ConformalVersionInequality}).  

These results have various consequences. A first of these is the extension of Moscovici's inequality to ordinary spectral triples that are not necessarily Poincar\'e 
duals of ordinary spectral triples, but are in Poincar\'e duality with \emph{twisted} spectral triples (Theorem~\ref{thm:VWSTdualTST}).  

For Dirac operators coupled with Hermitian connections on spin manifolds, the Vafa-Witten bound in~(\ref{eq:Intro.Vafa-Witten-inequality}) depends on the metric in a 
somewhat elusive way.  We refer to~\cite{An:OFVWBTDT, Ba:UBFEDOSM, DM:VWBCPS, Go:VWECSS, He:EVWBS} for various attempts to understand this dependence on the metric. Bearing this in mind, it is natural to look at the behavior of the Vafa-Witten bound with respect to 
conformal changes of metrics. As a consequence of our results, we obtain a conformal version the original Vafa-Witten inequality 
for \emph{coupled} Dirac operators on spin manifolds, where the Vafa-Witten bound is simply controlled by the maximum value of the conformal  factor 
(Theorem~\ref{thm:CoupledDiracOpWV}). %This enables us to define a new invariant attached to conformal structures,  

The aforementioned conformal version of Vafa-Witten's inequality has a noncommutative version.
More precisely, as an immediate consequence of the inequality of pseudo-inner twisted spectral triples, we obtain an inequality for conformal deformations of spectral triples with an explicit control of the Vafa-Witten bound in terms of the conformal factors (Theorem~\ref{thm:ConformalPerST-WVIneq1}). This result can be seen as a conformal version of Moscovici's inequality for ordinary spectral triples. 

Another consequences are versions of Vafa-Witten's inequality for spectral triples over noncommutative tori associated
 with conformal weights. We establish inequalities with an explicit control of the Vafa-Witten bound in terms of the Weyl factor of the conformal weight 
 (Theorem~\ref{thm:VWNCtorus1} and Theorem~\ref{thm:VWNCtorus2}). We also illustrate our results by  a noncompact example related to duals of torsion-free discrete cocompact subgroups of Lie groups satisfying the Baum-Connes conjecture and corresponding to conformal deformations by group elements (Theorem~\ref{prop:Vafa-Witten.duals-discrete-groups}).
 
The \emph{global strategy} of the proof of the Vafa-Witten inequality for twisted spectral triples is similar to that for ordinary 
spectral triples, but the \emph{local tactics} has a few twists. The most serious of these twists concerns the notion of eigenvalues. The 
Vafa-Witten inequality on spin manifolds and ordinary  spectral triples is stated for coupled Dirac operators 
$D_{\nabla^{\cE}}$  associated with Hermitian connections $\nabla^{\cE}$. These operators are selfadjoint operators acting 
between the same Hilbert space, and so eigenvalues of these operators have a clear meaning. However, for a twisted spectral triple $\left(\cA,\cH,D\right)_{\sigma}$  
and a 
noncommutative vector bundle $\cE$ (i.e., a finitely generated projective module over $\cA$), 
the coupled Dirac operators $D_{\nabla^{\cE}}$ (as defined in~\cite{PW:Index}) act between the Hilbert spaces $\cH(\cE)=\cE\otimes_{\cA}\cH$ and 
$\cH(\cE^{\sigma})=\cE^{\sigma}\otimes_{\cA}\cH$, where $\cE^{\sigma}$ is a ``$\sigma$-translation" of $\cE$ (see Definition~\ref{def:sigmaTranslate} for the precise meaning).  
These Hilbert spaces $\cH(\cE)$ and $\cH(\cE^{\sigma})$ 
do not agree in general, and so we cannot define eigenvalues of the operators $D_{\nabla^{\cE}}$ in the usual way. 

The above issue is dealt with by introducing the notion of a \emph{$\sigma$-Hermitian structure} on a noncommutative vector 
bundle $\cE$. Such a structure is given by the data of a Hermitian 
metric on $\cE$ and an identification $\fs:\cE^{\sigma}\rightarrow \cE$ satisfying some suitable compatibility condition 
(see Definition~\ref{def:SigmaHermitianStructure} for the precise conditions). This gives rise to a duality 
between $\cH(\cE)$ and $\cH(\cE^{\sigma})$. Thanks to this duality we define notions of \emph{$\fs$-adjoint 
operators}, \emph{$\fs$-selfadjointness} and \emph{$\fs$-eigenvalues}. They substitute for the usual notions of adjoint 
operators, selfadjointness and eigenvalues. Furthermore, we establish a max-min principle for the $\fs$-eigenvalues of an 
$\fs$-selfadjoint Fredholm operator (Proposition~\ref{prop:max-min}). 

In~\cite{PW:Index} Dirac operators coupled with $\sigma$-connections were defined and used to give a geometric description of the index map of a twisted spectral triple (see also Section~\ref{sec:IndexMapSigmaConnections}). In order to have coupled Dirac operators that are $\fs$-selfadjoint we use 
\emph{$\sigma$-Hermitian $\sigma$-connections}. These are $\sigma$-connections which are compatible with a given
$\sigma$-Hermitian structure (cf.\ Definition~\ref{def:SigmaHermitianConnection}). Dirac operators 
coupled with $\sigma$-Hermitian $\sigma$-connections are $\fs$-selfadjoint (see Proposition~\ref{prop:HermitianConnSA}).
The Vafa-Witten inequality is stated for the $\fs$-eigenvalues of those operators. 
 
This paper is organized as follows. In Section~\ref{sec:TwistedST}, we review the main definitions and examples regarding twisted 
spectral triples and introduce pseudo-inner twistings of ordinary spectral triples. In Section~\ref{sec:IndexMapTST}, we review the 
construction of the index map of a twisted spectral triple. In Section~\ref{sec:IndexMapSigmaConnections}, we recall the interpretation given in~\cite{PW:Index}
of this index map in terms of $\sigma$-connections. In Section~\ref{sec:SigmaHermitianStructures}, we introduce $\sigma$-Hermitian structures and 
define the $\fs$-spectrum of an operator. In Section~\ref{sec:sigmaHermit-sigmaConnection}, we introduce $\sigma$-Hermitian $\sigma$-connections and 
show that they produce $\fs$-selfadjoint coupled Dirac operators. 
 In Section~\ref{sec:spectral-triples}, we review basic facts about Poincar\'e duality for ordinary spectral triples. 
 In Section~\ref{sec:Poincare Duality for Twisted Spectral Triples}, we define Poincar\'e duality  for twisted 
spectral triples and look at various examples of such a duality. 
In section~\ref{sec:PDSigmaHermitianConnections}, we give a geometric description of Poincar\'e duality in terms of $\sigma$-Hermitian $\sigma$-connections.
In Section~\ref{sec:VWInequalities}, we establish Vafa-Witten inequalities for twisted spectral triples. In 
Section~\ref{sec:InequalityApp}, we derive their various geometric consequences. 

 Strictly speaking, in this paper we only deal with the Vafa-Witten inequality for \emph{even} twisted spectral triples. 
 The inequality also holds for \emph{odd} twisted spectral triples. The treatment involves spectral 
 flow manipulations and is postponed to a forthcoming paper~\cite{PW:NCGCG4} dealing with odd twisted spectral triples.  

\subsection*{Acknowledgements}
An important part of the research of this paper was carried out during visits of the two authors to the Mathematical 
Sciences Institute of the Australian National University and the Shanghai Center of Mathematical Sciences of Fudan 
University, and visits  of the first named author to Kyoto University, Mathematical Sciences Center of Tsinghua University, and University of California at Berkeley.
The authors would like to thank these institutions for their hospitality.

\section{Twisted Spectral Triples. Examples}
\label{sec:TwistedST}
In this section, we review the main definitions  and various examples regarding twisted spectral triples. Further examples are 
presented in~\cite{CM:TGNTQF} and~\cite{Mo:LIFTST}.

\subsection{Twisted spectral triples} 
Let us first recall the definition of an ordinary spectral triple. 

\begin{definition}\label{SpectralTriple}
A spectral triple $(\cA, \cH, D)$ is given by
\begin{enumerate}
\item A $\Z_2$-graded Hilbert space $\mathcal{H}=\mathcal{H}^+\oplus \mathcal{H}^-$.
\item An involutive unital algebra $\mathcal{A}$ represented by even bounded operators on $\cH$. %preserving its $\Z_{2}$-grading.
\item An  odd selfadjoint  unbounded operator $D$ on $\mathcal{H}$ such that %for all $a\in\mathcal{A},$
\begin{enumerate}
    \item The resolvent $(D+i)^{-1}$ is compact.
    \item $a (\dom D) \subset \dom D$ and $[D, a]$ is bounded for all $a \in \cA$. 
\end{enumerate}
\end{enumerate}   
\end{definition}
\begin{remark}
    A linear operator $T$ of $\cH$ is called even (resp., odd) when it maps $\cH^{\pm}\cap \dom T$ to $\cH^{\pm}$  
    (resp., $\cH^{\mp}$).
\end{remark}

\begin{example}
\label{DiracSpectralTriple}
    The prototype of a spectral triple is given by a Dirac spectral triple, \[ 
    \left(C^{\infty}(M),L^{2}_{g}(M,\sS),\sD_{g}\right),\]where 
    $(M,g)$ is a closed Riemannian spin manifold of even dimension 
    and $\sD_g$ is the Dirac operator acting on the $L^2$-sections of the spinor bundle $\sS=\sS^{+}\oplus \sS^{+}$.
\end{example}

Twisted spectral triples (a.k.a.\ modular spectral triples or $\sigma$-spectral triples) were introduced by Connes-Moscovici~\cite{CM:TGNTQF}.  
The definition of a twisted spectral triple is  almost identical to that of an ordinary spectral triple, 
except for some ``twist'' given by the conditions (3) and (4)(b) below. 

\begin{definition}\label{TwistedSpectralTriple}
A twisted spectral triple $(\cA, \cH, D)_{\sigma}$ consists of the following data:
\begin{enumerate}
\item A $\Z_2$-graded Hilbert space $\mathcal{H}=\mathcal{H}^{+}\oplus \mathcal{H}^{-}$.
\item An involutive unital algebra $\mathcal{A}$ represented by even bounded operators on $\cH$.

\item An automorphism $\sigma:\cA\rightarrow \cA$ such that $\sigma(a)^{*}=\sigma^{-1}(a^{*})$ for all $a\in 
\cA$. 
\item An odd selfadjoint unbounded operator $D$ on $\mathcal{H}$ such that %for all $a\in\mathcal{A},$
\begin{enumerate}
    \item The resolvent $(D+i)^{-1}$ is compact.
    \item $a (\dom D) \subset \dom D$ and $[D, a]_{\sigma}:=Da-\sigma(a)D$ is bounded for all $a \in \cA$. 
\end{enumerate}
\end{enumerate}   
\end{definition}

\begin{remark}\label{rmk:TST.sigma-involution}
    The condition that $\sigma(a)^{*}=\sigma^{-1}(a^{*})$ for all $a\in 
\cA$ ensures us that $a\rightarrow \sigma(a)^{*}$ is an involutive antilinear automorphism of $\cA$. 
\end{remark}

\begin{remark}
    Throughout the paper we will further assume that the algebra $\cA$ is closed under holomorphic functional calculus. This 
    implies that the $K$-theory groups of $\cA$ agree with that of its norm closure in $\cL(\cH)$. 
\end{remark}

\begin{remark}
    The boundedness of commutators naturally appear in the setting of quantum groups, but in the attempts of 
    constructing twisted spectral over quantum groups the compactness of the resolvent of $D$ seems to fail 
    (see~\cite{DA:QGTST, KS:TSTQSU2, KW:TSTCDC}). We also refer to~\cite{KW:TSTCDC} for relationships between 
    twisted spectral triples and Woronowicz's covariant differential calculi. 
\end{remark}

\subsection{Pseudo-inner twistings}
\label{sec:Pseudo-inner twistings}
As pointed out in~\cite{CM:TGNTQF}, an important class of examples of twisted spectral triples arises from  \emph{conformal deformations} (i.e., inner twistings) of ordinary spectral triples defined as follows. 

Let $(\cA,\cH,D)$ be an ordinary spectral and let $k$ be a positive invertible element of $\cA$. 
We note that $k$ acts as an even operator on $\cH$ and its action preserves the 
domain of $D$. Consider the operator, 
\begin{equation}
    D_{k}:=kDk, \qquad \dom D_{k}=\dom D.
    \label{eq:TwistedST.conformal-deformationD}
\end{equation}
As it turns out, $(\cA,\cH,D_{k})$ is not a spectral triple in general, but it can be turned into a 
\emph{twisted} spectral triple. 

\begin{proposition}(See \cite{CM:TGNTQF}.)
\label{Prop:ConformalPerturbation}Consider the automorphism $\sigma:\cA\rightarrow \cA$ defined by
    \begin{equation}
        \sigma(a)= k^{2}a k^{-2} \qquad \forall a \in \cA.
        \label{eq:TwistedST.sigmah}
    \end{equation}
Then $(\mathcal{A}, \mathcal{H}, D_{k})_{\sigma}$ is a twisted spectral triple. 
\end{proposition}

We shall now present a generalization of the above example. 
As $D$ is an odd operator with respect to the orthogonal splitting 
$\cH=\cH^{+}\oplus \cH^{-}$, it takes the form, 
\begin{equation}
    D= 
    \begin{pmatrix}
        0 & D^{-} \\
        D^{+} & 0
    \end{pmatrix}, \qquad D^{\pm}:\dom D\cap \cH^{\pm}\rightarrow \cH^{\mp}.
    \label{eq:TST.decompositionD}
\end{equation}
Let $\omega \in  \cL(\cH)$ be an even positive invertible operator preserving the domain of $D$. In particular, with respect to the splitting 
$\cH=\cH^{+}\oplus \cH^{-}$ the operator $\omega$ takes the form, 
\begin{equation}
    \omega=
    \begin{pmatrix}
        \omega^{+} & 0 \\
        0  & \omega^{-}
    \end{pmatrix}, \qquad \omega^{\pm}\in \cL(\cH^{\pm}).
    \label{eq:TST.decomposition-w}
\end{equation}
We further assume there is a pair  of positive invertible elements $k^{+}$ and $k^{-}$ of $\cA$ such that
\begin{gather}
    k^{+}k^{-}=k^{-}k^{+},
    \label{eq:TST.commutativity-kpm}\\
    \sigma^{\pm}(a):= \omega^{\pm}a(\omega^{\pm})^{-1}=k^{\pm}a(k^{\pm})^{-1} \qquad \forall a \in \cA.
    \label{eq:TST.pseudo-inner-sigmapm}
\end{gather}
Thus $(\sigma^{+},\sigma^{-})$ is a commuting pair of inner automorphisms of $\cA$. 

\begin{definition}\label{def:TST.pseudo-inner-twisting}
  An  even positive invertible element $\omega$ which preserves the domain of $D$ and 
  satisfies~(\ref{eq:TST.commutativity-kpm})--(\ref{eq:TST.pseudo-inner-sigmapm}) is called a 
  pseudo-inner twisting operator. 
\end{definition}

Given a pseudo-inner twisting $\omega$ as above, set $k=k^{+}k^{-}$ and let $\sigma:\cA\rightarrow \cA$ be the 
automorphism given by 
\begin{equation}
    \sigma(a)=\sigma^{+}\circ \sigma^{-}(a)=kak^{-1}, \qquad a \in \cA. 
    \label{eq:TST.pseudo-inner-sigma}
\end{equation}We also note that, for all $a\in \cA$,
\begin{equation*}
    \sigma(a)^{*}=(k ak^{-1})^{*}=k^{-1}a^{*}k=\sigma^{-1}(a^{*}). 
\end{equation*}
By assumption the domain of $D$ is preserved by $\omega$. Define 
    \begin{equation}
    \label{eq:DOmega}
        D_{\omega}:= \omega D \omega=
        \begin{pmatrix}
           0 & \omega^{+}D^{-} \omega^{-}\\
            \omega^{-}D^{+} \omega^{+}& 0
        \end{pmatrix}, \qquad \dom D_{\omega}:=\dom D. 
    \end{equation}

\begin{proposition}\label{Prop:ConformalPerturbationw}
The triple $\left(\cA, \cH, D_{\omega}\right)_{\sigma}$ is a twisted spectral triple. 
\end{proposition}
\begin{proof}
    The only conditions of Definition~\ref{TwistedSpectralTriple} that need to be checked are (4)(a) and (4)(b). We have 
    $(\omega D\omega+i)\omega^{-1}D^{-1}\omega^{-1}=1+i\omega^{-1}D^{-1}\omega^{-1}-\omega P_{0} \omega^{-1}$, where $D^{-1}$ the partial inverse of $D$ and $P_{0}$ is the orthogonal 
    projection onto 
    $\ker D$. Thus, 
    \begin{equation*}
        (\omega D\omega+i)^{-1}=\omega^{-1}D^{-1}\omega^{-1}-(i\omega^{-1}D^{-1}\omega^{-1}-\omega P_{0} \omega^{-1})(\omega D\omega+i)^{-1}. 
    \end{equation*}As all of the summands of the r.h.s.\ are compact operators, we deduce that 
    $(\omega D\omega+i)^{-1}$ too is a compact operator. 
    
    Let $a \in \cA$. Note that $\cA$ preserves $\dom D_{\omega}=\dom D$ and 
    $\sigma(a)=\sigma^{-}(\sigma^{+}(a))=\omega^{-}\sigma^{+}(a)(\omega^{-})^{-1}$. Therefore $[D_{\omega}^{+},a]_{\sigma}$ is equal to
      \begin{align}
      (\omega^{-}D^{+}\omega^{+})a -\sigma(a)(\omega^{-}D^{+}\omega^{+}) &= 
        \omega^{-}D^{+}\sigma^{+}(a)\omega^{+} -\omega^{-}\sigma^{+}(a)(\omega^{-})^{-1}(\omega^{-}D^{+}\omega^{+}) 
        \nonumber \\
       & =\omega^{-}[D^{+},\sigma^{+}(a)]\omega^{+}\in \cL(\cH^{+},\cH^{-}).
        \label{eq:Commutator1}
   \end{align}
   Similarly, 
   \begin{equation}
       [D_{\omega}^{-},a]_{\sigma}=\omega^{+}[D^{-},\sigma^{-}(a)]\omega^{-}\in \cL(\cH^{-},\cH^{+}).
               \label{eq:Commutator2}
   \end{equation}
   This shows that, for all $a \in \cA$, the twisted commutator 
   $[D_{\omega},a]_{\sigma}$ is bounded. The proof is complete. 
\end{proof}

 \begin{example}
     An inner twisting by a positive invertible element $k\in \cA$ is a  pseudo-inner twisting, where $\omega$ is given by the 
     action of $k$ on $\cH$. In this case $k^{\pm}=k$, and so $k^{+}k^{-}=k^{2}$. 
  \end{example}

\begin{example}(See {\cite[Sect.~1.1]{CM:MCNC2T}}.)
A simple example of a non-inner pseudo-inner twisting is obtained as follows. In the decomposition~(\ref{eq:TST.decomposition-w}) take $\omega^{+}=k$ and $\omega^{-}=1$, 
where $k$ is a positive invertible element of $\cA$. Then 
   $\omega$ is a pseudo-inner twisting. In this case $\sigma^{+}(a)=kak^{-1}$ and $\sigma^{-}(a)=a$. Moreover,
\begin{equation*}
              D_{\omega}= 
            \begin{pmatrix}
               0 & k D^{-} \\
                D^{+} k& 0
            \end{pmatrix} \qquad \text{and} \qquad \sigma(a)=kak^{-1} \quad \forall a \in \cA.
\end{equation*}
\end{example}

\begin{example}
 \label{ex:PseudoInnerTwistDiracSpectralTriple}
    Let $(C^{\infty}(M),L^{2}_{g}(M,\sS),\sD_{g})$ be a Dirac spectral triple as in Example~\ref{DiracSpectralTriple}, and consider a positive invertible even 
    section $\omega \in C^{\infty}(M,\End \sS)$. This gives rise to a pseudo-inner twisting with $k^{\pm}=1$, so that 
    we obtain an ordinary spectral triple $(C^{\infty}(M),L^{2}_{g}(M,\sS),\omega\, \sD_{g}\omega)$. We note this result 
    continues to hold if we only require $\omega$ to be an invertible Lipschitz section of $\End \sS$. 
 \end{example}

We briefly explain some relationship between the above example and conformal geometry.  Consider a conformal change of 
metric $\hat{g}:= k^{-2}g$, where $k\in C^{\infty}(M)$ and $k>0$. Let us take $\omega$ to be the multiplication operator by $\sqrt{k}$. 
We then obtain the ordinary spectral triple $(C^{\infty}(M),L^{2}_{g}(M,\sS), \sqrt{k}\sD_{g}\sqrt{k})$.
    
The inner product of $L^{2}_{g}(M,\sS)$ is given by 
    \begin{equation*}
        \acou{\xi}{\eta}_{g}:= \int_{M} \acoup{\xi(x)}{\eta(x)}\sqrt{g(x)}d^{n}x, \qquad \xi,\eta \in L^{2}_{g}(M,\sS),
    \end{equation*}where $\acoup{\cdot}{\cdot}$ is the Hermitian metric of $\sS$ (and $n= \dim M$). Consider the linear 
    isomorphism $U:L^{2}_{g}(M,\sS)\rightarrow 
    L^{2}_{\hat{g}}(M,\sS)$ given by
    \begin{equation*}
        U\xi= k^{\frac{n}{2}}\xi \qquad \forall \xi \in L^{2}_{\hat{g}}(M,\sS). 
    \end{equation*}We note that $U$ is a unitary operator since, for all $\xi\in L^{2}_{g}(M,\sS)$, we have
    \begin{equation*}
      \acou{U\xi}{U\xi}_{\hat{g}}= \int_{M}\acoup{k(x)^{\frac{n}{2}}\xi(x)}{k(x)^{\frac{n}{2}}\xi(x)} 
      \sqrt{k(x)^{-2}g(x)}d^{n}x=\acou{\xi}{\xi}_{g}. 
    \end{equation*}Moreover, the conformal invariance of the Dirac operator (see, e.g., \cite{Hi:HS}) means that 
    \[\sD_{\hat{g}}=k^{\frac{n+1}{2}}\sD_{g}k^{\frac{-n+1}{2}}.\] Thus,
    \begin{equation*}
        U^{*}\sD_{\hat{g}}U=k^{-\frac{n}{2}}\left( 
        k^{\frac{n+1}{2}}\sD_{g}k^{\frac{-n+1}{2}}\right)k^{\frac{n}{2}}=\sqrt{k}\; \sD_{g}\!\sqrt{k}.
    \end{equation*}Therefore, we obtain the following result. 
    
\begin{proposition}\label{prop:ConformalChangeDiracST}
        The Dirac spectral triple $(C^{\infty}(M),L^{2}_{\hat{g}}(M,\sS),\sD_{\hat{g}})$ is unitarily equivalent to the pseudo-inner twisted spectral triple 
        $(C^{\infty}(M),L^{2}_{g}(M,\sS), \sqrt{k}\; \sD_{g}\!\sqrt{k})$. 
\end{proposition}
 
\begin{remark}
\label{rem:kLip}
 Whereas the definition of $(C^{\infty}(M),L^{2}_{\hat{g}}(M,\sS),\sD_{\hat{g}})$ requires $k$ to be 
    smooth, in the definition of  $(C^{\infty}(M),L^{2}_{g}(M,\sS), \sqrt{k}\; 
    \sD_{g}\!\sqrt{k})$ it is enough to assume that $k$ is a positive Lipschitz function. 
\end{remark}

\begin{remark}
 We refer to~\cite{CM:TGNTQF} for the construction of a twisted spectral triple of the crossed-product of the Dirac spectral triple with conformal diffeomorphisms.   
 \end{remark}
 
\subsection{Twisted spectral triples over noncommutative tori} 
\label{subsec:TSTNCTorus} 
As shown by Connes-Tretkoff~\cite{CT:GBTNC2T} (see also~\cite{CM:MCNC2T}) the datum of a conformal weight on the noncommutative torus naturally gives rise to a 
twisted spectral triple. As we shall now explain, this construction fits nicely into the framework of pseudo-inner twisted spectral triples.

The  noncommutative torus $\cA_{\theta}$, $\theta \in \R$,  is the algebra, 
\begin{equation*}
    \cA_{\theta}=\left\{ \sum a_{m,n}U^{m}V^{n}; \ (a_{m,n})\in \cS(\Z^{2})\right\},
\end{equation*}where $U$ and $V$ are unitaries of $L^{2}(S^{1})$ such that $VU=e^{2i\pi \theta}UV$ and  $ \cS(\Z^{2})$ 
is the space of rapid decay sequences $(a_{m,n})_{m,n\in \Z}$ with complex entries.  We denote by $\varphi_{0}: \cA_{\theta}\rightarrow \C$ the unique normalized trace of 
$\cA_{\theta}$, i.e., 
\begin{equation*}
    \varphi_{0}\left( \sum a_{m,n}U^{m}V^{n}\right)=a_{00}. 
\end{equation*}
Let $\cH^{0}$ the Hilbert space obtained as the completion of $\cA_{\theta}$ with respect to the inner product, 
\begin{equation*}
    \acou{a}{b}=\varphi_{0}(b^{*}a), \qquad a,b\in \cA_{\theta}.
\end{equation*}

The holomorphic structures on $\cA_{\theta}$ are parametrized by numbers $\tau\in \C$, $\im \tau>0$. Fixing such a 
number, consider the derivation of $\cA_{\theta}$ given by  
\begin{equation*}
    \delta:=\delta_{1}+\overline{\tau}\delta_{2},
\end{equation*}where $\delta_{j}:\cA_{\theta} \rightarrow \cA_{\theta}$, $j=1,2$, are the canonical derivations of 
$\cA_{\theta}$ such that
\begin{equation*}
    \delta_{1}(U)=U, \qquad \delta_{2}(V)=V, \qquad \delta_{1}(V)=\delta_{2}(U)=0.
\end{equation*}The derivation $\delta$ plays the role of the operator $\frac{1}{i}\left(\frac{\partial}{\partial x}-i\frac{\partial}{\partial y}\right)$ on the 
ordinary torus. 

Let $\cA_{\theta}^{1,0}$ be the subspace of $\cA_{\theta}$ spanned by the ``(1,0)-forms'' $a\delta b$, where 
$a$ and $b$ range over  $\cA_{\theta}$. We denote by $\cH^{1,0}$ the Hilbert space obtained as the completion of $\cA_{\theta}^{1,0}$ with 
respect to the inner product, 
\begin{equation*}
    \acou{a_{1}\delta b_{1}}{a_{2}\delta b_{2}}=\varphi_{0}\left( a_{2}^{*}a_{1}(\delta b_{1})(\delta
    b_{2})^{*}\right), \qquad a_{j},b_{j}\in \cA_{\theta}.
\end{equation*}In addition, let $\partial$ be the closure of the operator $\delta$ seen as an operator from $\cA_{\theta}$ to 
$\cH^{1,0}$. On the Hilbert space $\cH:=\cH^{0}\oplus \cH^{1,0}$ consider the following selfadjoint unbounded operator,
\begin{equation}\label{eq:Derivation}
    D=
    \begin{pmatrix}
        0 & \partial^{*} \\
        \partial & 0
    \end{pmatrix}, \qquad \dom D=\dom \partial \oplus \dom \partial^{*}.
\end{equation}Then $(\cA_{\theta},\cH,D)$ is an ordinary spectral triple (see~\cite[VI.4.{$\beta$}]{Co:NCG}). 

Let us denote by $\cA_{\theta}^{\opp}$ the opposite algebra of $\cA_{\theta}$, i.e., the same vector space with the 
opposite product $(a,b)\rightarrow ba$. The right regular representation of $\cA_{\theta}$  extends to a representation 
$a\rightarrow a^{\opp}$ of $\cA_{\theta}^{\opp}$ in $\cH$ and it can be shown that $(\cA_{\theta}^{\opp},\cH,D)$ is an ordinary spectral triple 
(see~\cite[VI.4.{$\beta$}]{Co:NCG}). 

Let $k$ be a positive invertible element of $\cA_{\theta}$, and consider the weight $\varphi:\cA_{\theta}\rightarrow \C$ defined by
\begin{equation}
\label{eq:ConformalWeight}
    \varphi(a):= \varphi_{0}\left( a k^{-2}\right) \qquad \forall a \in \cA_{\theta}.
\end{equation}
In the terminology of~\cite{CM:MCNC2T} $\varphi$ is called a \emph{conformal weight} with the \emph{Weyl factor} $k$. 

Let $\cH_{\varphi}^{0}$ be the Hilbert space obtained as the completion of $\cA_{\theta}$ with 
respect to the inner product, 
\begin{equation*}
    \acou{a}{b}_{\varphi}:= \varphi(b^{*}a)=\varphi_{0}\left( b^{*}a k^{-2}\right), \qquad a,b \in \cA_{\theta}.
\end{equation*}
Let $\partial_{\varphi}$ be the closed extension with respect to the above inner product of the operator $\delta$ seen 
as an operator from $\cA_{\theta}$ to $\cH^{1,0}$. On the Hilbert space $\cH_{\varphi}:=\cH_{\varphi}^{0}\oplus 
\cH^{1,0}$ consider the selfadjoint unbounded operator,
\begin{equation}
\label{eq:DiracOperatorNCTorusT}
   D_{\varphi}:= 
   \begin{pmatrix}
       0 & \partial_{\varphi}^{*} \\
       \partial_{\varphi} & 0
   \end{pmatrix}, \qquad \dom D_{\varphi}=\dom \partial_{\varphi} \oplus \dom \partial_{\varphi}^{*}.
\end{equation}

The left regular representation of $\cA_{\theta}$ extends to a unitary representation of $\cA_{\theta}$ in 
$\cH_{\varphi}^{0}$. The right regular representation too extends to a representation $a\rightarrow a^{\opp}$ of $\cA_{\theta}^{\opp}$, but 
this representation is not unitary. Indeed, for all $(a,\xi,\eta)\in \cA_{\theta}^{3}$, 
\begin{equation*}
    \acou{a^{\opp}\xi}{ \eta }_{\varphi}=  \varphi_{0}\left(  \eta ^{*} \xi ak^{-2}\right)=\varphi_{0}\left( 
    ( \eta k^{-2}a^{*}k^{2})^{*} \xi k^{-2}\right)=\acou{ \xi }{(k^{-2}a^{*}k^{2})^{\opp} \eta }_{\varphi},
\end{equation*}which shows that the adjoint $a^{\opp}$ with respect to $\acou{\cdot}{\cdot}_{\varphi}$ is 
$\left(k^{-2}a^{*}k^{2}\right)^{\opp}$. A unitary representation of $\cA_{\theta}^{\opp}$ in 
$\cH_{\varphi}^{0}$ is given by
\begin{equation}
\label{eq:aphi0}
  \cA^{\opp}_{\theta}\ni  a \longrightarrow  a_{\varphi}^{\opp}:=\left(k^{-1}ak\right)^{\opp}\in 
  \cL\left(\cH_{\varphi}^{0}\right).
\end{equation}We also note that $a\rightarrow k^{-1}a^{*}k$ is the Tomita antilinear involution of the GNS representation 
associated with $\varphi$.  

In what follows we denote by $R_{k}$ the right-multiplication operator by $k$ on $\cA_{\theta}$.

\begin{lemma}(See \cite{CM:MCNC2T, CT:GBTNC2T}.)
\label{lem:UnitaryOpH0HphiNCTorus}
The operator $R_{k}$ uniquely extends to a 
   unitary operator $W_{0}:\cH^{0}\rightarrow \cH_{\varphi}^0$ such that, for all $a \in \cA_{\theta}$ and $\xi \in 
   \cH^{0}$, 
   \begin{equation}
   \label{eq:W_0module}
       W_{0}(a\xi)=aW_{0}\xi \qquad \text{and} \qquad W_{0}(a^{\opp}\xi)=a^{\opp}_{\varphi}W_{0}\xi.
   \end{equation}
   That is, $W_{0}$ intertwines the representations of $\cA_{\theta}$  (resp., $\cA_{\theta}^{\opp}$)
   in $\cH$ and $\cH_{\varphi}$. 
   \end{lemma}

We represent $\cA_{\theta}^{\opp}$ in $\cH_{\varphi}=\cH_{\varphi}^{0}\oplus \cH^{1,0}$ by means of the unitary 
representation,
\begin{equation}
\label{eq:RepOppAlgNCTorus}
      \cA^{\opp}_{\theta}\ni  a \longrightarrow  a_{\varphi}^{\opp}:= 
      \begin{pmatrix}
         \left(k^{-1}ak\right)^{\opp}  & 0 \\
          0 & a^{\opp}
      \end{pmatrix}
 \in  \cL\left(\cH_{\varphi}\right),
\end{equation}
where by an abuse of notation we denote by $a_{\varphi}^{\opp}$ both the representation of $a$ as an operator on $\cH_{\varphi}$ and its 
restriction to $\cH_{\varphi}^{0}$. Let $W$ be the unitary operator from $\cH=\cH^{0}\oplus \cH^{1,0}$ to $\cH_{\varphi}=\cH^{0}_{\varphi}\oplus \cH^{1,0}$ defined by
\begin{equation}
    W= 
\begin{pmatrix}
        W_{0} & 0 \\
                 0 & 1
\end{pmatrix}.
\label{eq:IsoHandH_phi}
\end{equation}
It follows from Lemma~\ref{lem:UnitaryOpH0HphiNCTorus}  that $W$ intertwines the respective representations of $\cA_{\theta}$ and $\cA_{\theta}^{\opp}$ 
   in $\cH$ and $\cH_{\varphi}$. Moreover, if we regard $\cA_{\theta}$ as a dense subspace of $\cH^{0}$, then, for all 
   $\xi \in \cA_{\theta}$, 
\begin{equation}
\label{eq:Dphi}
    W^{-1}D_{\varphi}W\xi = \delta (R_{k} \xi) = \partial k^{\opp}\xi.
\end{equation}
Consider the bounded operator on $\cH=\cH^{0}\oplus \cH^{1,0}$ defined by 
\begin{equation}
    \omega=
    \begin{pmatrix}
        k^{\opp} & 0 \\
        0 & 1
    \end{pmatrix}.
    \label{eq:TST.Rk-inner-twisting}
\end{equation}It follows from~(\ref{eq:Dphi}) that
\begin{equation}
\label{eq:Dphi&D}
    W^{-1}D_{\varphi}W= 
    \begin{pmatrix}
        0 & k^{\opp}\partial^{*} \\
      \partial k^{\opp}   & 0
    \end{pmatrix}=\omega D \omega.
\end{equation}

We note that $\omega$ is a positive invertible even operator of $\cH_{\varphi}$. Moreover, as $k^{\opp}$ commutes with the action of 
$\cA_{\theta}$ we see that $\omega$ is a pseudo-inner twisting for the spectral triple $(\cA_{\theta},\cH,D)$ with 
trivial associated automorphisms. Thus $(\cA_{\theta},\cH,\omega D\omega)$ is an ordinary spectral triple. 
In addition, for $a \in \cA_{\theta}^{\opp}$, the operator $\omega a^{\opp}\omega^{-1}$ is equal to $a^{\opp}$ 
on $\cH^{1,0}$ and is equal to $k^{\opp}a^{\opp}\left(k^{\opp}\right)^{-1}=\left( k^{-1}ak\right)^{\opp}$ on 
$\cH_{\varphi}^{0}$. Therefore, we see that $\omega$ is a pseudo-inner twisting of the spectral triple  
$(\cA_{\theta}^{\opp},\cH,D)$ with $k^{+}=k^{-1}$ and $k^{-}=1$, so that the associated automorphism $\sigma$ is given 
by  
\begin{equation}
\label{eq:InnerAutomorphismNCTorus}
    \sigma(a)=k^{-1}ak \qquad \forall a \in \cA_{\theta}^{\opp}.
\end{equation}
It then follows from Proposition~\ref{Prop:ConformalPerturbationw} that $(\cA_{\theta}^{\opp}, \cH, \omega 
D\omega)_{\sigma}$ is a twisted spectral triple. 

As is shown in~(\ref{eq:W_0module}) and~(\ref{eq:Dphi&D}), the unitary operator $W$ intertwines the ordinary spectral triple 
$(\cA_{\theta},\cH,\omega D\omega)$ with $(\cA_{\theta},\cH_{\varphi},D_{\varphi})$.  We also note that
\begin{equation*}
  W^{-1}\sigma(a)^{\opp}W=\sigma(a)^{\opp}_{\varphi} \qquad \forall a \in \cA^{\opp}_{\theta}.  
\end{equation*}Therefore, we see that $W$ also intertwines the twisted 
spectral $(\cA_{\theta}^{\opp},\cH,\omega D\omega)_{\sigma}$ with $(\cA_{\theta}^{\opp},\cH_{\varphi},D_{\varphi})_{\sigma}$. As the 
axioms for a twisted spectral triple are preserved by unitary intertwinings we eventually arrive at the following 
result.

\begin{proposition}(See {\cite{CM:MCNC2T, CT:GBTNC2T}}.)
\label{prop:IsomConformalTSTNCTorus} Let $\varphi$ be a conformal weight on $\cA_{\theta}$ with the Weyl factor 
$k\in \cA_{\theta}$, $k>0$. In addition, let $W$ be the unitary operator~(\ref{eq:IsoHandH_phi}) and $\omega$ the pseudo-inner twisting 
operator~(\ref{eq:TST.Rk-inner-twisting}). Then
\begin{enumerate}
    \item  $(\cA_{\theta},\cH_{\varphi},D_{\varphi})$ is an ordinary spectral triple.    
    
    \item  $(\cA_{\theta}^{\opp},\cH_{\varphi},D_{\varphi})_{\sigma}$ is a twisted spectral triple, where 
    $\cA_{\theta}^{\opp}$ is represented in $\cH_{\varphi}$ by~(\ref{eq:RepOppAlgNCTorus}). 
    
    \item The unitary operator $W$ intertwines $(\cA_{\theta},\cH_{\varphi},D_{\varphi})$ (resp., 
    $(\cA_{\theta}^{\opp},\cH_{\varphi},D_{\varphi})_{\sigma}$) with 
   $(\cA_{\theta},\cH,\omega D\omega)$ (resp., $(\cA_{\theta}^{\opp},\cH,\omega D\omega)_{\sigma}$). 
\end{enumerate}
\end{proposition}

\section{The Index Map of a Twisted Spectral Triple}
\label{sec:IndexMapTST}
In this section, we recall the construction of the index map of a twisted spectral triple. 
The exposition follows closely that of~\cite{PW:Index}. We refer to~\cite{GGK:BCLO} for basic facts about 
\emph{unbounded} Fredholm operators. 

 Let us first briefly recall how the datum of an ordinary spectral $(\cA, \cH, D)$  gives rise to an additive index map $\ind_{D}:K_{0}(\cA)\rightarrow \Z$. 
 Let $e$ be an idempotent in $M_{q}(\cA)$, $q \in \Z$. We regard $e\cH^{q}$ as a closed subspace of the Hilbert space 
 and we equip it with the induced inner product. As $e$ acts as an even operator on 
 $\cH^{q}=\left(\cH^{+}\right)^{q}\oplus \left(\cH^{-}\right)^{q}$ we see that $e\cH^{q}\cap 
 (\cH^{\pm})^{q}=e(\cH^{\pm})^{q}$, and so we have an orthogonal splitting 
 $e\cH^{q}=e(\cH^{+})^{q}\oplus e(\cH^{-})^{q}$. In addition, as the action of $\cA$ preserves the 
 domain of $D$ we see that $e(\dom D)^{q}=(\dom D)^{q}\cap e\cH^{q}$. We then form the unbounded operator $D_{e}$ on $e 
 \cH^{q}$ defined by
 \begin{equation*}
     D_{e}:=e (D\otimes 1_{q}), \qquad \dom D_{e}=e(\dom D)^{q}.  
 \end{equation*}With respect to the orthogonal splitting $e\cH^{q}=e (\cH^{+})^{q}\oplus e (\cH^{-})^{q}$ the operator 
 $D_{e}$ takes the form, 
 \begin{equation*}
     D_{e}= 
     \begin{pmatrix}
         0 & D^{-}_{e} \\
         D_{e}^{+} & 0
     \end{pmatrix}, \qquad D_{e}^{\pm}=e(D^{\pm}\otimes 1_{q}),
 \end{equation*}where $D^{\pm}$ is defined as in~(\ref{eq:TST.decompositionD}). 
 
 \begin{lemma}(See {\cite[Lemma 4.3]{PW:Index}}.) The operator $D_{e}$ is closed and Fredholm, and we have
     \begin{equation}
     \label{eq:IndexD_ePM}
         \ind D_{e}^{\pm} = \dim \ker D_{e}^{\pm}- \dim \ker D_{e^{*}}^{\mp}. 
     \end{equation}
 \end{lemma}
 \begin{remark} \label{rem:DualityeHe*H}
     When $e^*=e$ the operator $D_{e}$ is selfadjoint, i.e., $(D_{e}^{\pm})^{*}=D_{e}^{\mp}$. Therefore, in this case 
     $\coker D_{e}^{\pm}\simeq \ker D_{e}^{\mp}$, and so get
     \begin{equation*}
         \ind D_{e}^{+}=\dim \ker D_{e}^{+}-\dim \ker D_{e}^{-}=-\ind D_{e}^{-}. 
     \end{equation*}
     When $e\neq e^{*}$, the operator $D_{e^{*}}$ is not the adjoint of $D_{e}$ in the usual sense, but $D_{e^{*}}$ can be identified 
     with $D_{e}$ with respect to a suitable duality (cf.~\cite[Lemma 4.2]{PW:Index}). 
%      above result is well known (see, e.g., \cite{Hi:RITCM}). We note that in general 
%      $D_{e^{*}}$ is not the adjoint of $D_{e}$ in the usual sense, unless $e^{*}=e$, in which case $D_{e}$ is 
%      selfadjoint. Nevertheless, the inner product of $\cH^{q}$ induces a nondegenerate bilinear pairing between $e\cH^{q}$ 
%      and $e^{*}\cH^{q}$. The operator $D_{e^{*}}$ can be shown to be the adjoint of $D_{e}$ with respect to this duality, which implies 
%      the isomorphism~(\ref{eq:IsoCokerDeKerDe*}) (see~\cite{PW:Index}).
 \end{remark}
 
 We define the index of $D_{e}$ by 
 \begin{align}
     \ind D_{e}&= \frac{1}{2} \left(  \ind D_{e}^{+}-  \ind D_{e}^{-}\right)   \label{eq:Index.index-De} \\
                     & = \frac{1}{2} \left(  \dim \ker D_{e}^{+} + \dim \ker D_{e^{*}}^{+} -\dim \ker D_{e}^{-}-\dim \ker 
     D_{e^{*}}^{-}\right) .\nonumber
 \end{align}
 In particular, when $e^{*}=e$ we obtain
 \begin{equation}
 \label{eq:e*=eIntegerIndex}
     \ind D_{e}=  \dim \ker D_{e}^{+}- \dim \ker D_{e}^{-}=\ind D_{e}^{+}\in \Z,
 \end{equation}
 which is the usual definition of the index of $D_{e}$ when $e$ is selfadjoint (see, e.g., \cite{Hi:RITCM}). 
 
Moreover, if $g\in \op{Gl}_{q}(\cA)$ and $f\in M_{q'}(\cA)$ is another idempotent, then it can be shown that
 \begin{equation}
     \ind D_{g^{-1}eg}=\ind D_{e} \qquad \text{and} \qquad  \ind D_{e\oplus f}=\ind D_{e}+\ind D_{f}.
     \label{eq:Index.similarity-invariance-twisted}
 \end{equation}This shows that $\ind D_{e}$ depends only on the $K$-theory class of $e$ and it depends on it additively. 
 Moreover, as $e$ is similar to the orthogonal projection onto its range (see, e.g., \cite[Prop.~4.6.2]{Bl:KTOA}), 
  we see that $\ind D_{e}$ can always be put in the form~(\ref{eq:e*=eIntegerIndex}) without changing the $K$-theory class of $e$. In 
 particular, $\ind D_{e}$ is always an integer. Therefore, we arrive at the following statement.

 \begin{proposition}
 There is a unique additive map $\ind_{D}:K_{0}(\cA)\rightarrow \Z$ such that
 \begin{equation*}
     \ind_{D}[e]=\ind D_{e} \qquad \forall e \in M_{q}(\cA), \ e^{2}=e. 
 \end{equation*}    
 \end{proposition}

 As observed by Connes-Moscovici~\cite{CM:TGNTQF}, the datum of a twisted spectral triple $(\cA, \cH, D)_{\sigma}$ also 
 gives rise to an additive index map. This index map is constructed as follows.

 Given an idempotent $e\in M_{q}(\cA)$, $q \in \N$, we denote by $D_{e,\sigma}$ the unbounded operator from 
$e\cH^{q}$ to $\sigma(e)\cH^{q}$ defined by
\begin{equation}
    D_{e,\sigma}:=\sigma(e) (D\otimes 1_{q}), \qquad \dom D_{e,\sigma}=e(\dom D)^{q}.
    \label{eq:Index.Desigma}
 \end{equation}With respect to the orthogonal splittings $e\cH^{q}=e (\cH^{+})^{q}\oplus e (\cH^{-})^{q}$ and 
 $\sigma(e)\cH^{q}=\sigma(e) (\cH^{+})^{q}\oplus \sigma(e)(\cH^{-})^{q}$ the operator 
$D_{e,\sigma}$ takes the form, 
\begin{equation*}
    D_{e,\sigma}= 
    \begin{pmatrix}
        0 & D^{-}_{e,\sigma} \\
        D_{e,\sigma}^{+} & 0
    \end{pmatrix}, \qquad D_{e,\sigma}^{\pm}=\sigma(e)(D^{\pm}\otimes 1_{q}).
\end{equation*}%where $D^{\pm}$ is defined as in~(\ref{eq:TST.decompositionD}). 

\begin{lemma}(See {\cite[Lemma 4.3]{PW:Index}}.) \label{lem:D_esigmaFredholmIndex}
The operator $D_{e,\sigma}$ is closed and Fredholm, and we have
    \begin{equation}\label{eq:IndexD_eSigmaPM}
        \ind D_{e,\sigma}^{\pm} = \dim \ker D_{e,\sigma}^{\pm} - \dim \ker D_{\sigma(e)^{*},\sigma}^{\mp}. 
    \end{equation}
\end{lemma}

\begin{remark}
   As in Remark~\ref{rem:DualityeHe*H}, with respect to a suitable duality the adjoint of $D_{e,\sigma}^{\pm}$ is naturally identified 
   with $ D_{\sigma(e)^{*},\sigma}^{\mp}$ (cf.~\cite[Lemma 4.2]{PW:Index}). This implies that $\coker D_{e,\sigma}^{\pm} \simeq \ker 
   D_{\sigma(e)^{*},\sigma}^{\mp}$, which yields the equality~(\ref{eq:IndexD_eSigmaPM}). 
 \end{remark}

We define the index of $D_{e,\sigma}$ by
\begin{equation}
    \ind D _{e,\sigma}:= \frac{1}{2} \left(\ind  D_{e,\sigma}^{+}  - \ind  D_{e,\sigma}^{-}\right). 
    \label{eq:Index-index-Desigma}
\end{equation}In particular, thanks to~(\ref{eq:IndexD_eSigmaPM}) we have
\begin{equation}
\label{eq:IndexDesigma}
  \ind D _{e,\sigma}= \frac{1}{2} \left(  \dim \ker D_{e,\sigma}^{+}+ \dim \ker D_{\sigma(e)^{*},\sigma}^{+} -
  \dim \ker D_{e,\sigma}^{-}-\dim \ker D_{\sigma(e)^{*},\sigma}^{-}\right).
\end{equation}
Moreover, if $g\in \op{Gl}_{q}(\cA)$ and $f\in M_{q'}(\cA)$ is another idempotent, then 
\begin{equation}
    \ind D_{g^{-1}eg,\sigma}=\ind D_{e,\sigma} \qquad \text{and} \qquad  \ind D_{e\oplus f,\sigma}=\ind D_{e,\sigma}+\ind D_{f,\sigma}.
    \label{eq:Index.similarity-invariance-untwisted}
\end{equation}
Therefore, we arrive at the following result.

\begin{proposition}(See \cite{CM:TGNTQF, PW:Index}.) There is a unique additive map $\ind_{D,\sigma}:K_{0}(\cA)\rightarrow \frac{1}{2} \Z$ such that
\begin{equation*}
    \ind_{D,\sigma}[e]=\ind D_{e,\sigma} \qquad \forall e \in M_{q}(\cA), \ e^{2}=e. 
\end{equation*}    
\end{proposition}

As pointed out in Remark~\ref{rmk:TST.sigma-involution}, the fact that $\sigma(a)^{*}=\sigma^{-1}(a^{*})$ for all $a\in \cA$ means that the map 
$a\rightarrow \sigma(a)^{*}$ is an  involutive antilinear  automorphism of $\cA$. An element $a \in \cA$ is selfadjoint 
with respect to this involution if and only if $\sigma(a)^{*}=a$. We shall say that such an element is 
\emph{$\sigma$-selfadjoint}. We observe that  if $\sigma(e)^{*}=e$, then~(\ref{eq:IndexDesigma}) shows that 
\begin{equation}
    \ind D_{e,\sigma}= \dim \ker D_{e,\sigma}^{+}-\dim \ker D_{e,\sigma}^{-}=\ind D_{e,\sigma}^{+}\in \Z.
    \label{eq:Index.integer-value}
\end{equation}

While an idempotent in $M_{q}(\cA)$ is always conjugate to a selfadjoint idempotent. We need a further technical assumption on 
the algebra $\cA$ to have a similar result for $\sigma$-selfadjoint idempotents. 

\begin{definition}
    The automorphism $\sigma$ is called ribbon, when it admits a square root in the sense that there is an automorphism $\tau:\cA\rightarrow  \cA$ such   that %, for all $a \in \cA$,
\begin{equation}
    \sigma(a)=\tau(\tau(a)) \quad \text{and} \quad \tau(a)^{*}=\tau^{-1}(a^{*}) \qquad \text{for all $a\in \cA$}.
    \label{eq:Index.square-root-sigma}
\end{equation}
\end{definition}

\begin{lemma}(See \cite{PW:Index}.)
\label{lm:CriteriaSigmaAdjointIntegerIndex}
    Assume that the automorphism $\sigma$ is ribbon.  Then
    \begin{enumerate}
        \item[(i)]  Any idempotent $e\in M_{q}(e)$, $q \in \N$, is similar to a $\sigma$-selfadjoint idempotent.
    
        \item[(ii)]  The index map $\ind_{D,\sigma}$ is integer-valued. 
    \end{enumerate}
\end{lemma}

As the following shows, the ribbon condition~(\ref{eq:Index.square-root-sigma}) is satisfied in all the examples of 
twisted spectral triples described in Section~\ref{sec:TwistedST}.

\begin{example}
    Assume that $\sigma(a)=kak^{-1}$ where $k$ is a positive invertible element of $\cA$. Then $\sigma$ is ribbon with 
    $\tau(a)=k^{\frac{1}{2}}ak^{-\frac{1}{2}}$. Note that $k^{\frac{1}{2}}$ is an element of $\cA$ since $\cA$ is 
    closed under holomorphic functional calculus. More generally, 
    suppose that $\sigma$ agrees with the value at $t=-i$ of the analytic extension of a strongly continuous one-parameter group of 
    isometric $*$-isomorphisms $(\sigma_{t})_{t\in \R}$. Then the condition~(\ref{eq:Index.square-root-sigma}) is 
    satisfied by taking $\tau:=\sigma_{\left|t=-i/2\right.}$. We note that by a result of Bost~\cite{Bo:POKTSDNC}, the analytic extension of a strongly continuous one-parameter group of 
    isometric isomorphisms on an involutive Banach algebra always exists on a subalgebra which is closed under 
    holomorphic functional calculus. 
\end{example}

Let us now give a closer look at the index map of pseudo-inner twisted spectral triples.  Let $(\cA, \cH, D)$ be an ordinary spectral triple and let $\omega \in \cL(\cH)$ be a pseudo-inner 
twisting as in Definition~\ref{def:TST.pseudo-inner-twisting}. Thanks to Proposition~\ref{Prop:ConformalPerturbationw} 
we know that $(\cA,\cH,D_{\omega})_{\sigma}$ is a twisted spectral 
triple, where $D_{\omega}=\omega D \omega$ and $\sigma$ is given by~(\ref{eq:TST.pseudo-inner-sigma}). 

\begin{lemma}
\label{lm:EqualityOfIndices} 
The following holds. 
\begin{enumerate}
    \item   Let $e$ be an idempotent in $M_{q}(\cA)$, $q \in \N$. Then 
 \begin{equation*}
     \ind (D_{\omega})_{e,\sigma}= \ind D_{e}.
 \end{equation*}

    \item  The index maps $\ind_{D_{\omega},\sigma}$ and $\ind_{D}$ agree. In particular, $\ind_{D,\sigma}$ is 
    integer-valued. 
\end{enumerate}
\end{lemma}
\begin{proof}
As the second part is an immediate consequence of the first part, we only need to prove the latter. Thus, let $e$ be an 
idempotent in $M_{q}(\cA)$, $q \in \N$. As $(\omega^{+}\otimes 1_{q})e=\sigma^{+}(e)(\omega^{+}\otimes 1_{q})$, we see that 
$\omega^{+}\otimes 1_{q}\in \cL\left((\cH^{+})^{q}\right)$ induces a continuous operator from $e(\cH^{+})^{q}$ to 
$\sigma^{+}(e)(\cH^{+})^{q}$. Let us denote by $W^{+}$ this operator. Note that $W^{+}$ is invertible and its inverse 
is the operator induced by $(\omega^{+})^{-1}\otimes 1_{q}$. In addition,  $W^{+}$ maps 
$e(\dom D^{+})^{q}$ to $\sigma^{+}(e)(\dom D^{+})^{q}$. 

Observe that $(\omega^{-}\otimes 1_{q})\sigma^{+}(e)=\sigma^{-}(\sigma^{+}(e))(\omega^{-}\otimes 
1_{q})=\sigma(e)(\omega^{-}\otimes 1_{q})$. Therefore, by arguing as above, 
we see that  $\omega^{-}\otimes 1_{q}\in \cL\left((\cH^{-})^{q}\right)$ induces an invertible continuous operator 
$W^{-}:\sigma^{+}(e)(\cH^{-})^{q}\rightarrow \sigma(e)(\cH^{-})^{q}$. Moreover, on $(\cH^{-})^{q}$ we have $\sigma(e) (\omega^{-}\otimes 
1_{q})=W^{-} \sigma^{+}(e)$. Then on $e(\cH^{+})^{q}$ we have
\begin{equation*}
    (D_{\omega}^{+})_{e,\sigma}=\sigma(e)\left( (\omega^{-}D^{+}\omega^{+})\otimes 
    1_{q}\right)=W^{-}\sigma^{+}(e)(D^{+}\otimes 1_{q})W^{+}=W^{-}D_{\sigma^{+}(e)}^{+}W^{+}.
\end{equation*}As $W^{+}$ and $W^{-}$ are invertible operators, using~(\ref{eq:Index.similarity-invariance-twisted}) we deduce that
\begin{equation*}
    \ind (D_{\omega}^{+})_{e,\sigma}= \ind D_{\sigma^{+}(e)}^{+}= \ind D_{k^{+}e(k^{+})^{-1}}^{+}=\ind D_{e}^{+}.
\end{equation*}
It can be similarly shown that $  \ind (D_{\omega}^{-})_{e,\sigma}=\ind D_{k^{-}e(k^{-})^{-1}}^{-}=\ind 
D_{e}^{-}$. Thus, 
\begin{equation*}
    \ind D_{e,\sigma}=\frac{1}{2}\left( \ind D_{e,\sigma}^{+}-\ind D_{e,\sigma}^{-}\right)=\frac{1}{2}\left( \ind 
    D_{e}^{+}-\ind D_{e}^{-}\right)= \ind D_{e}.
\end{equation*}The proof is complete. 
\end{proof}

\begin{remark}
    For ordinary spectral triples the index map is computed by the pairing of $K_{0}(\cA)$ with the Connes-Chern 
    character in cyclic cohomology~\cite{Co:GCMFNCG, Co:NCG}. We refer to~\cite{CM:TGNTQF} for the construction of the Connes-Chern 
    character for twisted spectral triples in (ordinary) cyclic cohomology. 
\end{remark}

\section{Index Map and $\sigma$-Connections}
\label{sec:IndexMapSigmaConnections}
In this section, we recall the description in~\cite{PW:Index} of the index map of a twisted spectral triple in terms of  
couplings by $\sigma$-connections. We refer to~\cite{Mo:EIPDNCG} for a similar description in the case of ordinary spectral triples. 
 
Throughout this section and the next two sections we let $(\cA,\cH,D)_{\sigma}$ be a twisted spectral triple. In 
addition, let $\cE$ be a finitely generated projective right module 
over $\cA$, i.e., $\cE$ is a direct summand of a finite rank free module.  In order to define $\sigma$-connections we 
need the following notion of ``$\sigma$-translation''. 

\begin{definition}(See \cite{PW:Index}.)\label{def:sigmaTranslate}
    A $\sigma$-translation of $\cE$ is given by a pair $(\cE^{\sigma},\sigma^{\cE})$, where
    \begin{itemize}
        \item[(i)] $\cE^{\sigma}$ is   finitely generated projective right module over $\cA$.  
    
        \item[(ii)]  $\sigma^{\cE}:\cE\rightarrow \cE^{\sigma}$ is a $\C$-linear isomorphism such that
                  \begin{equation}
                       \sigma^{\cE}(\xi a)=\sigma^{\cE}(\xi)\sigma(a)  \qquad \text{for all $\xi\in \cE$ and $a\in \cA$}.
	        \label{eq:sigma^E(xi a)}
                  \end{equation}    
    \end{itemize}
\end{definition}

\begin{remark}
    The condition~(\ref{eq:sigma^E(xi a)}) means that $\sigma^{\cE}$ is a right module isomorphism from $\cE$ onto $\cE^{(\sigma)}$, where $\cE^{(\sigma)}$ is 
$\cE^{\sigma}$ equipped with the action $(\xi,a)\rightarrow \xi\sigma(a)$. In particular, when $\sigma=\op{id}$ a 
$\sigma$-translation of $\cE$ is simply given by a right-module isomorphism $\sigma^{\cE}:\cE\rightarrow \cE^{\sigma}$. 
Therefore, a canonical choice of $\sigma$-translation is to take $(\cE, \op{id})$. This will always the choice we shall 
make when $\sigma=\op{id}$. 
\end{remark}

\begin{remark}\label{rmk:sigmae}.
   Suppose that $\cE=e\cA^{q}$, for some idempotent $e\in M_{q}(\cA)$, $q\geq 1$. The automorphism $\sigma$ lifts to $\cA^{q}$ by 
\begin{equation*}
    \sigma(\xi)=(\sigma(\xi_{j})) \qquad \forall \xi=(\xi_{j})\in \cA^{q}. 
\end{equation*}
Note that $\sigma$ is a $\C$-linear isomorphism of $\cA^{q}$ onto itself and maps $e\cA^{q}$ onto 
$\sigma(e)\cA^{q}$, and so it induces a $\C$-linear isomorphism $\sigma_{e}:e\cA^{q}\rightarrow \sigma(e)\cA^{q}$. 
Moreover, it can checked that $\sigma_{e}$ further satisfies~(\ref{eq:sigma^E(xi a)}) 
(see~\cite[Remark~5.3]{PW:Index}). Therefore, the pair $(\sigma(e)\cA^{q},\sigma_{e})$ is a $\sigma$-translation of $e\cA^{q}$. This 
will be our canonical choice of $\sigma$-translation when $\cE$ is of the form $e\cA^{q}$, with $e^{2}=e\in 
M_{q}(\cA)$, $q\geq 1$.
\end{remark}

\begin{remark}\label{rmk:sigma-translate}
    In general,  any  $\sigma$-translation is obtained as follows. By definition $\cE$ is a direct-summand of a free module 
    $\cE_{0}\simeq \cA^{q}$. Let $\phi:\cE_{0}\rightarrow \cA^{q}$ be a right module isomorphism. The 
image of $\cE$ by $\phi$ is a right module of the form $e\cA^{q}$ for some idempotent $e \in M_{q}(\cA)$. Set $\cE^{\sigma}:=\phi^{-1}\left( 
\sigma(e)\cA^q\right)$; this is a direct summand of $\cE_{0}$. The isomorphism $\phi$ induces isomorphisms of right modules,
\begin{equation*}
    \phi_{e}:\cE\longrightarrow e\cA^{q}\qquad \text{and}\qquad \phi_{\sigma(e)}:\cE^{\sigma}\longrightarrow \sigma(e)\cA^{q}.
\end{equation*}
Set $\sigma^{\cE}=\left( \phi_{\sigma(e)}\right)^{-1}\circ \sigma_{e}\circ \phi_{e}$, where 
$\sigma_{e}:e\cA^{q}\rightarrow \sigma(e)\cA^{q}$ is the $\sigma$-lift introduced in Remark~\ref{rmk:sigmae}. Then $\sigma^{\cE}$ is a 
$\C$-linear isomorphism from $\cE$ onto $\cE^{\sigma}$ satisfying~(\ref{eq:sigma^E(xi a)}) 
(see~\cite[Remark~5.4]{PW:Index}), and so $(\cE^{\sigma},\sigma^{\cE})$ is a $\sigma$-translation of $\cE$. 

Conversely, given a $\sigma$-translation $(\cE,\sigma^{\cE})$ and a right-module isomorphism $\phi:\cE\rightarrow 
e\cA^{q}$, the map $\phi_{\sigma}= \sigma_e \circ \phi \circ \left(\sigma^{\cE}\right)^{-1}$ is a right-module 
isomorphism from $\cE^{\sigma}$ onto $\sigma(e)\cA^{q}$ such that $\sigma^{\cE}=\left( \phi_{\sigma}\right)^{-1}\circ 
\sigma_{e}\circ \phi$ (see~\cite[Remark~5.4]{PW:Index}). We also note that when $\sigma$ is ribbon Lemma~\ref{lm:CriteriaSigmaAdjointIntegerIndex} 
enables us to choose an idempotent $e$ so that $\sigma(e)=e^{*}$. 
\end{remark}
 
Following~\cite{CM:TGNTQF} we consider the space of twisted 1-forms, 
 \begin{equation*}
    \Omega^{1}_{D,\sigma}(\cA)=\left\{\Sigma a^{i}[D, b^{i}]_{\sigma}: a^{i}, b^{i} \in\cA \right\}.
\end{equation*}Note that  $\Omega^{1}_{D,\sigma}(\cA)$ is a subspace of $\cL(\cH)$. Moreover, this is naturally an $(\cA,\cA)$-bimodule, since
\begin{equation*}
    a^{2}(a^{1}[D,b^{1}]_{\sigma})b^{2}= a^{2}a^{1}[D,b^{1}b^{2}]_{\sigma}-a^{2}a^{1}\sigma(b^{1})[D,b^{2}]_{\sigma}
   \qquad  \forall  a^{j}, b^{j} \in \cA.
\end{equation*}
 We also have a ``twisted'' differential $d_{\sigma}:\cA \rightarrow \Omega^{1}_{D,\sigma}(\cA)$ defined by 
\begin{equation}
\label{eq:TwistedDifferential}
    d_{\sigma}a:= [D,a]_{\sigma} \qquad \forall a \in \cA.
\end{equation}
This is a $\sigma$-derivation, in the sense that
\begin{equation}
\label{eq:SigmaDerivation}
    d_{\sigma}(ab)=(d_{\sigma}a)b+\sigma(a)d_{\sigma}b\qquad \forall a, b \in \cA. 
\end{equation}

Throughout the rest of the section we let $(\cE^{\sigma},\sigma^{\cE})$ be a $\sigma$-translation of $\cE$.
\begin{definition}
\label{def:SigmaConnection}
A $\sigma$-connection on $\cE$ is a $\C$-linear map $\nabla^{\cE}: \cE\rightarrow 
\cE^{\sigma}\otimes_{\cA}\Omega^1_{D,\sigma}(\cA)$ such that 
\begin{equation}
\label{eq:SigmaConnectionModuleMulti}
 \nabla^{\cE}(\xi a)=(\nabla^{\cE}\xi) a+\sigma^{\cE}(\xi)\otimes d_{\sigma}a \qquad \forall \xi\in\cE \ \forall a\in\cA.   
\end{equation}    
 \end{definition}

\begin{example}
 Suppose that $\cE=e\cA^{q}$ with $e=e^{2}\in M_{q}(\cA)$. Then a natural $\sigma$-connection on $\cE$ is the \emph{Grassmannian $\sigma$-connection} $\nabla_{0}^{\cE}$ defined by 
 \begin{equation}
 \label{eq:GrassmannianSigmaConnection}
     \nabla_{0}^{\cE}\xi= \sigma(e)(d_{\sigma}\xi_{j}) \qquad \text{for all $\xi=(\xi_{j})$ in $\cE$}.
 \end{equation}    
\end{example}

\begin{lemma}(See \cite{PW:Index}.)
\label{lem:DifferenceSigmaConnection}
    The set of $\sigma$-connections on $\cE$ is an affine space modeled on $\op{Hom}_{\cA}(\cE,\cE^{\sigma}\otimes \Omega^{1}_{D,\sigma}(\cA))$.
\end{lemma}

In what follows we denote by $\cE'$ the dual $\cA$-module $\Hom_{\cA}(\cE,\cA)$. 
 
\begin{definition}\label{def:connection.Hermitian-metric}
     A Hermitian metric on $\cE$ is a map $\acoup{\cdot}{\cdot}:\cE\times \cE \rightarrow \cA$ such that
     \begin{enumerate}
        \item $\acoup{\cdot}{\cdot}$ is $\cA$-sesquilinear, i.e., it is $\cA$-antilinear with respect to the first 
        variable and $\cA$-linear with respect to the second variable. 
        
         \item  $(\cdot, \cdot)$ is positive, i.e., $\acoup{\xi}{\xi}\geq 0$ for all $\xi \in \cE$.
     
         \item   $(\cdot, \cdot)$ is nondegenerate, i.e., $\xi \rightarrow \acoup{\xi}{\cdot}$ 
         is an $\cA$-antilinear isomorphism from $\cE$ onto $\cE'$.
     \end{enumerate}
 \end{definition}

  \begin{remark}
     Using (2) and a polarization argument it can be shown that $(\xi_{1},\xi_{2})=(\xi_{2},\xi_{1})^{*}$ for all 
     $\xi_{j}\in \cA$.
  \end{remark}

 \begin{example}\label{ex:CanonicalHermitianStructure}
  The canonical Hermitian structure on the free module $\cA^{q}$ is given by
  \begin{equation}
  \label{eq:CanonicalHermitianStructure}
      \acoup{\xi}{\eta}_{0}=\xi_{1}^{*}\eta_{1}+\cdots + \xi_{q}^{*}\eta_{q} \qquad \text{for all $\xi=(\xi_{j})$ and 
      $\eta=(\eta_{j})$ in $\cA^{q}$}.
  \end{equation}
%  If $e=e^{*}=e^{2}\in M_{q}(\cA)$, then this induces a Hermitian metric on $e\cA^{q}$ (\emph{cf.}\ 
%  Remark~\ref{rem:DualityeHe*H}). 
  \end{example}
 
\begin{lemma}(See \cite{PW:Index}.)
\label{lem:CanonicalHermitMetric-eA^q}
    Suppose that $\cE=e\cA^{q}$ with $e=e^{2}\in M_{q}(\cA)$. Then the canonical Hermitian metric of $\cA^{q}$ induces 
    a Hermitian metric on $\cE$.
\end{lemma}

\begin{remark}
      Let $\phi:\cE\rightarrow \cF$ be an isomorphism of finitely generated projective modules and assume $\cF$ carries 
      a Hermitian metric $\acoup{\cdot}{\cdot}_{\cF}$. Then using $\phi$ we can pullback the Hermitian metric of $\cF$ to  the 
      Hermitian metric on $\cE$ given by
      \begin{equation*}
          \acoup{\xi_{1}}{\xi_{2}}_{\cE}:=\acoup{\phi(\xi_{1})}{\phi(\xi_{2})}_{\cF} \qquad \forall \xi_{j} \in \cE.
      \end{equation*}
      In particular, if we take $\cF$ to be of the form $e\cA^{q}$ with $e=e^{2}\in M_{q}(\cA)$, then we can 
      pullback the canonical Hermitian metric $\acoup{\cdot}{\cdot}_{0}$ to a Hermitian metric on $\cE$. 
  \end{remark}  
  
 From now on we assume that $\cE$ and $\cE^{\sigma}$ carry Hermitian metrics.  We denote by $\cH(\cE)$ the pre-Hilbert space 
consisting of $\cE\otimes_\cA \cH$ equipped with the Hermitian inner product, 
 \begin{equation}
 \label{eq:HermitianInnerProductH(E)}
     \acou{\xi_{1}\otimes \zeta_{1}}{\xi_{2}\otimes \zeta_{2}}:= \acou{\zeta_{1}}{(\xi_{1},\xi_{2})\zeta_{2}}, \qquad 
     \xi_{j}\in \cE,  \zeta_{j} \in \cH, 
 \end{equation}where $\acoup{\cdot}{\cdot}$ is the Hermitian metric of $\cE$. 
 
 \begin{lemma}(See \cite{PW:Index}.) \label{lem:H(E)topIndepHermitianMetricE}
    $ \cH(\cE)$ is a Hilbert space whose topology is independent of the choice of the Hermitian inner product of $\cE$. 
 \end{lemma}
  \begin{remark}
     In~\cite{Mo:EIPDNCG} the Hilbert space $\cH(\cE)$ is defined as the completion of $\cE\otimes_\cA \cH$ with respect to the inner product~(\ref{eq:HermitianInnerProductH(E)}), as the above lemma shows this pre-Hilbert space is actually complete. 
  \end{remark}
 
 We note there is a natural $\Z_{2}$-grading on $\cH(\cE)$ given by
 \begin{equation}
 \label{eq:Z_2GradingH(E)}
     \cH(\cE)=\cH^{+}(\cE)\oplus \cH^{-}(\cE), \qquad \cH^{\pm}(\cE):=\cE\otimes_{\cA}\cH^{\pm}.
 \end{equation}
 
 We also form the $\Z_{2}$-graded Hilbert space 
 $\cH(\cE^{\sigma})$ as above.  In addition, we let $\nabla^{\cE}$ be a $\sigma$-connection on $\cE$. 
 Regarding $\Omega_{D,\sigma}^{1}(\cA)$ as a subalgebra of 
$\cL(\cH)$, we have a natural left action $c:\Omega_{D,\sigma}^{1}(\cA)\otimes_{\cA}\cH \rightarrow \cH$ given by 
\begin{equation*}
    c(\omega\otimes \zeta)=\omega(\zeta) \qquad \text{for all $\omega\in \Omega_{D,\sigma}^{1}(\cA)$ and $\zeta \in 
    \cH$}.
\end{equation*}
We denote by $c\left(\nabla^{\cE}\right)$ the composition $(1_{\cE^{\sigma}}\otimes c)\circ (\nabla^{\cE}\otimes 
1_{\cH}):\cE\otimes\cH \rightarrow \cE^{\sigma}\otimes\cH$.  Thus, for $\xi\in \cE$ and $\zeta \in \cH$, and upon writing
 $\nabla^{\cE}\xi=\sum \xi_{\alpha} \otimes\omega_{\alpha}$ with $\xi_{\alpha}\in \cE^{\sigma}$ and 
 $\omega_{\alpha}\in  \Omega^{1}_{D,\sigma}(\cA)$, we have
 \begin{equation}
 \label{eq:CliffordNabla}
        c\left(\nabla^{\cE}\right)(\xi\otimes\zeta)=\sum \xi_{\alpha}\otimes\omega_{\alpha}(\zeta). 
\end{equation}

In what follows we regard the domain of $D$ as a left $\cA$-module, which is possible since the action of $\cA$ on 
$\cH$ preserves $\dom D$. 

\begin{definition}(See \cite{PW:Index}.)
 \label{def:DNablaE}
   The coupled operator $D_{\nabla^{\cE}}:\cE\otimes_{\cA} \dom D \rightarrow \cH(\cE^{\sigma})$ is  defined by
\begin{equation}
\label{eq:Index.Dnabla}
    D_{\nabla^{\cE}}(\xi\otimes \zeta):=\sigma^{\cE}(\xi)\otimes D\zeta + c(\nabla^{\cE})(\xi\otimes \zeta) \qquad 
    \text{for all $\xi\in \cE$ and $\zeta \in \dom D$}.
\end{equation}   
\end{definition}

\begin{remark}
    With respect to the $\Z_2$-gradings~(\ref{eq:Z_2GradingH(E)}) for $\cH(\cE)$ and $\cH(\cE^{\sigma})$ the operator $D_{\nabla^{\cE}}$ takes the form, 
\begin{equation*}
 D_{\nabla^{\cE}}=   \begin{pmatrix}
    0    & D_{\nabla^{\cE}}^{-} \\
        D_{\nabla^{\cE}}^{+} & 0
    \end{pmatrix}, \qquad D_{\nabla^{\cE}}^{\pm}:\cE\otimes_{\cA}\dom D^{\pm}\longrightarrow \cH^{\mp}(\cE^{\sigma}).
\end{equation*}That is, $D_{\nabla^\cE}$ is an odd operator.
\end{remark}

\begin{example}[See~\cite{PW:Index}]
  Suppose that $\cE=e\cA^{q}$ with $e=e^{2}\in M_{q}(\cA)$. Then there is a canonical isomorphism $U_{e}$ from 
$\cH(\cE)$ to $e\cH^{q}$ given by
\begin{equation*}
    U_{e}(\xi\otimes \zeta)=\left(\xi_{j}\zeta\right)_{1\leq j \leq q} \qquad \text{for all $\xi=(\xi_{j})\in \cE$ and $\zeta \in \cH$},
\end{equation*}where $\cE=e\cA^{q}$ is regarded as a submodule of $\cA^{q}$. The inverse of $U_e$ is given by 
\begin{equation*}
    U^{-1}_{e}\left( (\zeta_{j})\right) = \sum e_{j}\otimes \zeta_{j} \qquad \text{for all $(\zeta_{j})\in e\cH^{q}$},
\end{equation*}where $e\cH^{q}$ is regarded as a subspace of $\cH^{q}$ and $e_{1},\ldots,e_{q}$ are the column vectors 
of $e$. Let $\nabla_{0}^{\cE}$ be the Grassmannian 
$\sigma$-connection of $\cE$. Then by~\cite[Lemma~5.18]{PW:Index} we have 
    \begin{equation}
        U_{\sigma(e)}D_{\nabla_{0}^{\cE}}U_{e}^{-1}=D_{e,\sigma},
        \label{eq:sigma-connections.DnablaE0}
    \end{equation}where $D_{e,\sigma}$ is defined as in~(\ref{eq:Index.Desigma}). 
\end{example}

We are now in a position to state the main result of this section.

\begin{proposition}(See \cite{PW:Index}.)\label{thm.IndexTwisted-connection}
Let $\cE$ be a Hermitian finitely generated projective module. Then, 
for any $\sigma$-connection $\nabla^{\cE}$ on $\cE$,  the operator $D_{\nabla^{\cE}}$ is Fredholm, and we have
    \begin{equation*}
        \ind_{D,\sigma}[\cE]= \frac{1}{2}\left( \ind D_{\nabla^{\cE}}^{+}-\ind D_{\nabla^{\cE}}^{-}\right).
    \end{equation*}
\end{proposition}

\begin{remark}
The above result provides us with a more transparent analogy with the index map defined by coupled Dirac operators on spin manifolds. 
It is also an important ingredient in the proof of the Vafa-Witten inequality for twisted spectral triples (\emph{cf.}\ Section~\ref{sec:VWInequalities}).  
\end{remark}

\section{$\sigma$-Hermitian Structures and $\fs$-Spectrum}
\label{sec:SigmaHermitianStructures}
 The coupled operators $D_{\nabla^{\cE}}$ defined in the previous section are operators from $\cH(\cE)$ to 
$\cH(\cE^{\sigma})$, but in general the Hilbert spaces $\cH(\cE)$ and $\cH(\cE^{\sigma})$ do not agree. Therefore, eigenvalues of the coupled operators $D_{\nabla^{\cE}}$ do not 
make sense in the usual way. In this section, we shall remedy this by introducing the notion of a $\sigma$-Hermitian structure on a finitely generated projective module.

Throughout this section we let $\cE$ be a finitely generated projective right-module over $\cA$ and 
$(\cE^{\sigma},\sigma^{\cE})$ a $\sigma$-translation of $\cE$. 

\begin{definition}
\label{def:SigmaHermitianStructure}
    A $\sigma$-Hermitian structure on $\cE$ is given by 
    
\begin{enumerate}
 \item[(i)]  A Hermitian metric $\acoup{\cdot}{\cdot}$ on $\cE$. 
 
 \item[(ii)]  A right-module isomorphism $\fs:\cE^{\sigma}\rightarrow \cE$ such that
    \begin{equation}
        \acoup{\xi_{1}}{\fs\left(\sigma^{\cE}(\xi_{2})\right)}= \sigma \left[  \acoup{\fs\left(\sigma^{\cE}(\xi_{1})\right)}{\xi_{2}}\right]\qquad \forall 
        \xi_{j}\in \cE.
     \label{eq:SigmaHermitianStructure} 
    \end{equation}
\end{enumerate}
\end{definition}

\begin{example} 
\label{ex:SigmaHermitianStrutureOnFreeModule}
The free module $\cA^{q}$ has a canonical $\sigma$-Hermitian structure defined as follows.  Note that 
    $\left(\cA^{q}\right)^{\sigma}=\cA^{q}$. Furthermore, for all $\xi=(\xi_{j})$ and 
    $\eta=(\eta_{j})$ in $\cA^{q}$, we have
    \begin{equation}
    \label{eq:CanonicalSigmaHermitStrucFreeModule}
       \sigma\left( \acoup{\sigma(\xi)}{\eta}_{0}\right)= \sum \sigma\left( \sigma(\xi_{j})^{*}\eta_{j}\right)= \sum 
       \sigma\left( \sigma(\xi_{j})^{*}\right)\sigma(\eta_{j})= \sum \xi_{j}^{*}\sigma(\eta_{j})= 
       \acoup{\xi}{\sigma(\eta)}_{0}.
    \end{equation}
  This is the condition~(\ref{eq:SigmaHermitianStructure}) for $\fs= 1_{\cA^{q}}$, and so the pair $\left\{ \acoup{\cdot}{\cdot}_{0}, 
  1_{\cA^{q}}\right\}$ defines a $\sigma$-Hermitian structure on $\cA^{q}$. 
\end{example}

As the following lemma shows, when $\sigma$ is inner the datum of a Hermitian structure canonically defines a $\sigma$-Hermitian structure. 

\begin{lemma}
\label{lem:innerSigmaHermitian}
Assume that $\sigma(a)=kak^{-1}$ for some positive invertible element $k\in \cA$. Let $\acoup{\cdot}{\cdot}$ be a 
    Hermitian metric on $\cE$. Define $\fs:\cE^{\sigma}\rightarrow \cE$ by
        \begin{equation}
        \label{eq:SigmaHermitianInner}
                   \fs(\xi)= \left[\left( \sigma^{\cE}\right)^{-1}(\xi)\right]k^{-1}, \qquad \xi \in \cE^{\sigma}.
        \end{equation}
Then $\fs$ is a right-module isomorphism and $\left\{\acoup\cdot\cdot, \fs\right\}$ defines a 
$\sigma$-Hermitian structure on $\cE$. 
\end{lemma}
\begin{proof}
    The map $\fs$ is a linear isomorphism. Let $\xi\in \cE^{\sigma}$ and $a\in \cA$. As $\sigma(a)=kak^{-1}$ 
    we have 
    \begin{equation*}
        \fs(\xi a)= \left[\left( \sigma^{\cE}\right)^{-1}(\xi a)\right]k^{-1} =\left[ \left( \sigma^{\cE}\right)^{-1}(\xi) \right]
        \sigma^{-1}(a)k^{-1}=\left[\left( \sigma^{\cE}\right)^{-1} (\xi)\right]
        k^{-1}a=\fs(\xi)a.
    \end{equation*}Thus $\fs$ is a right-module isomorphism from $\cE^{\sigma}$ onto $\cE$. 
    
    In addition, for $\xi_{1}$ and $\xi_{2}$ in $\cE^{\sigma}$, we have
    \begin{equation*}
        \acoup{\xi_{1}}{\fs\circ \sigma^{\cE}(\xi_{2})}= 
        \acoup{\xi_{1}}{\xi_{2}k^{-1}}=\acoup{\xi_{1}}{\xi_{2}}k^{-1}=\sigma\left[ 
        k^{-1}\acoup{\xi_{1}}{\xi_{2}}\right]= \sigma\left[ \acoup{\fs\circ \sigma^{\cE}(\xi_{1})}{\xi_{2}}\right].
    \end{equation*}This proves that the pair $\left\{\acoup\cdot\cdot, \fs\right\}$ defines a 
$\sigma$-Hermitian structure on $\cE$. The proof is complete.
\end{proof}

\begin{remark}
    When $\sigma$ is inner we shall use the $\sigma$-Hermitian structure defined by a given Hermitian metric and the 
    map $\fs$ given by~(\ref{eq:SigmaHermitianInner}). We note this implies that $\fs\circ \sigma^{\cE}$ is the right action of $k^{-1}$. 
\end{remark}

\begin{example}
\label{ex:fs=k^-1}
    Suppose that $\sigma$ is as above and $\cE=e\cA^{q}$ with $e=e^{2}\in M_{q}(\cA)$. Then, for all $\xi$ in 
    $\cE^{\sigma}=\sigma(e)\cA^{q}$, we have 
    \begin{equation*}
         \fs(\xi)= \left[\left( \sigma^{\cE}\right)^{-1}(\xi)\right]k^{-1}=  \sigma^{-1}(\xi)k^{-1}=(k^{-1}\xi k )k^{-1}=k^{-1}\xi.
    \end{equation*}That is, the map $\fs$ is the left-multiplication by $k^{-1}$.
\end{example}

More generally, when $\sigma$ is ribbon we have the following existence result. 
\begin{lemma}
\label{lem:SqrtSigmaSigmaHermitian}
Suppose that $\sigma$ is ribbon in the sense of~(\ref{eq:Index.square-root-sigma}) and let $\acoup{\cdot}{\cdot}$ be a 
    Hermitian metric on a finitely generated projective module $\cE$. Then there is an $\cA$-linear isomorphism $\fs:\cE^{\sigma}\rightarrow \cE$ such that 
    the pair $\{ \acoup{\cdot}{\cdot}, \fs\}$ defines a $\sigma$-Hermitian structure on $\cE$. 
\end{lemma}
\begin{proof}
Let us first assume that  $\cE=e\cA^{q}$ with $e=e^{2}\in M_{q}(\cA)$ such that $\sigma(e)=e^{*}$, and let $\acoup{\cdot}{\cdot}_{\cE}$ be a Hermitian metric on $\cE$. 
    The canonical Hermitian metric 
$\acoup{\cdot}{\cdot}_{0}$ of $\cA^{q}$ induces a 
nondegenerate $\cA$-sesquilinear pairing on $e^{*}\cA^{q}\times e\cA^{q}$ (see, e.g., \cite[Lemma~A.1]{PW:Index}). Therefore, there is a unique $\cA$-linear 
isomorphism $\fs$ from $e^{*}\cA^{q}=\sigma(e)\cA^{q}=\cE^{\sigma}$ onto 
$\cE$ such that
\begin{equation}
            \acoup{\fs\eta}{\xi}_{\cE}= \acoup{\eta}{\xi}_{0} \qquad \forall (\eta,\xi)\in \cE^{*}\times \cE.
     \label{eq:SigmaHermitianSigmaE=e*}
\end{equation}
Moreover, using~(\ref{eq:CanonicalSigmaHermitStrucFreeModule}) we see that, for all $\xi_{j}\in \cE$,  
\begin{equation}
\label{eq:HermiMetIdeA}
 \sigma \left[  \acoup{\fs\left(\sigma^{\cE}(\xi_{1})\right)}{\xi_{2}}_{\cE}\right] = \sigma \left[  
 \acoup{\left(\sigma^{\cE}(\xi_{1})\right)}{\xi_{2}}_{0}\right]  = \acoup{ \xi_{1}}{\sigma^{\cE}(\xi_{2})}_{0}=  
 \acoup{\xi_{1}}{\fs\left(\sigma^{\cE}(\xi_{2})\right)}_{\cE}.
\end{equation}
This shows that the pair $\left\{\acoup{\cdot}{\cdot}_{\cE}, \fs\right\}$ defines a $\sigma$-Hermitian 
structure on $\cE$. 

Suppose now that $\cE$ is an arbitrary Hermitian finitely generated projective module. By Remark~\ref{rmk:sigma-translate} there exist an idempotent 
$e\in M_{q}(\cA)$ such that $\sigma(e)^{*}=e$ and right-module isomorphisms $\phi:\cE\rightarrow e\cA^{q}$ and 
$\phi_{\sigma}:\cE^{\sigma}\rightarrow \sigma(e)\cA^{q}$ such that $\phi_{\sigma}\circ \sigma^{\cE}=\sigma\circ \phi$. 
We equip $e\cA^{q}$ with the Hermitian metric $\acoup{\xi}{\eta}_{e}=\acoup{\phi^{-1}(\xi)}{\phi^{-1}(\eta)}$ where 
$\xi,\eta\in e\cA^{q}$. We let $\fs_{e}:\sigma(e)\cA^{q}\rightarrow e\cA^{q}$ be the right-module isomorphism defined as 
in~(\ref{eq:SigmaHermitianSigmaE=e*}) by using the Hermitian metric $\acoup{\cdot}{\cdot}_{e}$. Using $\phi$ and $\phi_{\sigma}$ we pull it 
back to the right-module isomorphism $\fs:\cE^{\sigma}\rightarrow \cE$ defined by
\begin{equation*}
    \fs=\phi^{-1}\circ \fs_{e}\circ \phi_{\sigma}.
\end{equation*}Let $\xi_{1}$ and $\xi_{2}$ be  in $\cE$. Using~(\ref{eq:HermiMetIdeA}) we see that
$\acoup{\xi_{1}}{\fs\circ\sigma^{\cE}(\xi_{2})}$ is equal to
\begin{align*}
    \acoup{\phi(\xi_{1})}{\phi\circ \fs\circ \phi_{\sigma}^{-1}\circ 
    \sigma\circ \phi(\xi_{2})}_{e}=\acoup{\phi(\xi_{1})}{\fs_{e}\circ \sigma(\phi(\xi_{2}))}_{e} & = 
    \sigma\left[\acoup{\fs_{e}\circ \sigma(\phi(\xi_{1}))}{\phi(\xi_{2})}_{e}\right] \\ 
    &=  \sigma\left[\acoup{\fs\circ 
    \sigma^{\cE}(\xi_{1})}{(\xi_{2})}\right].
\end{align*}
Therefore the pair $\{\acoup{\cdot}{\cdot},\fs\}$ defines a $\sigma$-Hermitian structure on $\cE$. The proof is 
complete.
\end{proof}

From now one we assume that $\cE$ carries a $\sigma$-Hermitian structure given by a Hermitian metric 
$\acoup{\cdot}{\cdot}$ and a right-module isomorphism $\fs:\cE^{\sigma} \rightarrow\cE$. We then endow $\cE^{\sigma}$ with the Hermitian metric defined by
\begin{equation}
\label{eq:HermitianMetricInducedESigma}
\sacoups{\eta_1}{\eta_2}:=\acoup{\fs\eta_1}{\fs\eta_2} \qquad \forall \eta_j\in \cE^\sigma.
\end{equation}
In addition, we define  a nondegenerate $\cA$-sesquilinear pairing $\sacoup{\cdot}{\cdot}:\cE^{\sigma}\times \cE \rightarrow \cA$ by 
\begin{equation}
     \sacoup{\eta}{\xi}:= \acoup{\fs\eta}{\xi} \qquad \forall (\eta,\xi) \in \cE^{\sigma}\times \cE.
     \label{eq:LSigmaPairing}
\end{equation}
Similarly, we a have a non-degenerate $\cA$-sesquilinear pairing $\acoups{\cdot}{\cdot}:\cE\times \cE^{\sigma} 
\rightarrow \cA$ given by 
\begin{equation}
     \acoups{\xi}{\eta}:= \acoup{\xi}{\fs\eta} \qquad \forall (\xi,\eta) \in \cE\times \cE^{\sigma}.
     \label{eq:RSigmaPairing}
\end{equation}
We observe that~(\ref{eq:SigmaHermitianStructure}) implies that
\begin{equation}
             \acoups{\xi}{\eta}=\sigma\left[ \sacoup{\sigma^{\cE}(\xi)}{(\sigma^{\cE})^{-1}(\eta)}\right] \qquad \forall (\xi,\eta) \in \cE\times \cE^{\sigma}.
                     \label{eq:RelationLRSigmaPairing}
\end{equation}

\begin{example}
\label{ex:LRSigmaPairingCanonical}
     Suppose that $\cE=e\cA^{q}$ with $e=e^{2}\in M_{q}(\cA)$ such that $\sigma(e)=e^{*}$. Let us equip $\cE$ with a 
     $\sigma$-Hermitian structure as in the proof of Lemma~\ref{lem:SqrtSigmaSigmaHermitian}. 
    In this case~(\ref{eq:SigmaHermitianSigmaE=e*}) and~(\ref{eq:LSigmaPairing}) show that the pairing $\sacoup\cdot\cdot$ agrees with the nondegenerate pairing between 
    $\cE^{\sigma}=\cE^{*}$ and $\cE$ induced by the canonical Hermitian metric of $\cA^{q}$. Furthermore, 
    using~(\ref{eq:CanonicalSigmaHermitStrucFreeModule}) and~(\ref{eq:RelationLRSigmaPairing}) we also see that $\acoups\cdot\cdot$ as well agrees with that pairing.  
\end{example}

Using the Hermitian metric $\acoup{\cdot}{\cdot}$ we form the Hilbert space $\cH(\cE)$ as $\cE\otimes_\cA \cH$ equipped with the inner product defined 
by~(\ref{eq:HermitianInnerProductH(E)}). 
We similarly form the Hilbert space $\cH(\cE^\sigma)$ by using the Hermitian metric $\sacoups{\cdot}{\cdot}$ on $\cE^\sigma$. We shall denote by $\sacous{\cdot}{\cdot}$ 
the associated inner product on $\cH(\cE^\sigma)$. 

We observe that~(\ref{eq:HermitianMetricInducedESigma}) implies that 
$\fs\otimes 1_\cH$ is a unitary operator from $\cH(\cE^{\sigma})$ onto $\cH(\cE)$. Therefore, in the same as way as in~(\ref{eq:LSigmaPairing}) and~(\ref{eq:RSigmaPairing}), 
we define (continuous) nondegenerate sesquilinear maps 
$\sacou{\cdot}{\cdot}:\cH(\cE^\sigma)\times \cH(\cE)\rightarrow \C$ and $\acous{\cdot}{\cdot}:\cH(\cE)\times \cH(\cE^\sigma)\rightarrow \C$ by letting, 
for all $\zeta\in \cH(\cE)$ and $\zeta^\sigma\in \cH(\cE^\sigma)$, 
\begin{equation}
 \sacou{\zeta^\sigma}{\zeta}:=\acou{(\fs\otimes 1_\cH)\zeta^\sigma}{\zeta} \qquad \text{and} \qquad  \acous{\zeta}{\zeta^\sigma}:=\acou{\zeta}{(\fs\otimes 1_\cH)\zeta^\sigma}.
 \label{eq:defLRSigmaPairingHilbert}
\end{equation}
We note that, for $(\xi,\eta)\in \cE\times \cE^\sigma$ and $\zeta_j\in \cH$, $j=1,2$, we have 
\begin{equation}
 \sacou{\eta\otimes\zeta_1}{\xi\otimes \zeta_2}=\acou{\zeta_1}{\sacoup{\eta}{\xi}\zeta_2} \qquad \text{and} \qquad 
 \acous{\xi\otimes\zeta_1}{\eta\otimes \zeta_2}=\acou{\zeta_1}{\acoups{\xi}{\eta}\zeta_2}. 
 \label{eq:LRSigmaPairingHilbertSpace}
\end{equation}
These duality pairings between $\cH(\cE)$ and $\cH(\cE^\sigma)$ provide us with a corresponding notion of adjoint. 

\begin{definition}
Let $T$ be a densely defined operator from $\cH(\cE)$ to $\cH(\cE^{\sigma})$.
\begin{enumerate}
    \item  The $\fs$-adjoint of $T$, denoted $T^{\dagger}$, is operator from $\cH(\cE)$ to $\cH(\cE^{\sigma})$ with 
    graph,
\begin{equation}
\label{eq:GrapfAdjointDirac}
    G(T^{\dagger}):=\left\{ (\xi,\eta) \in \cH(\cE) \times \cH(\cE^{\sigma}); 
    \acous{\xi}{T\zeta}=\sacou{\eta}{\zeta} \ \forall \zeta \in \dom T\right\}.
\end{equation}

    \item  The operator $T$ is called $\fs$-selfadjoint when $T=T^{\dagger}$. 
\end{enumerate}
\end{definition}

The following lemma relates the $\fs$-adjoint $T^{\dagger}$ to the usual adjoint $T^{*}$. 

\begin{lemma}\label{lem:STDagger=ST*} Let $T$ be a densely defined operator from $\cH(\cE)$ to $\cH(\cE^{\sigma})$. Then
\begin{enumerate}
    \item  $T^{*}=(\fs\otimes 1_{\cH})T^{\dagger} (\fs\otimes 1_{\cH})$. 

    \item  $T$ is $\fs$-selfadjoint if and only if $(\fs\otimes 1_{\cH})T$ is selfadjoint.
\end{enumerate}
\end{lemma}
\begin{proof}
Set $S=\fs\otimes 1_{\cH}$. The very definitions~(\ref{eq:defLRSigmaPairingHilbert}) of the pairing $\acous{\cdot}{\cdot}$ and $\sacou{\cdot}{\cdot}$ show 
    that $(\xi,\eta)\in \cH(\cE^{\sigma})\times \cH(\cE)$ belongs to the graph of $T^{\dagger}$ if and 
    only if 
    \begin{equation*}
        \acou{\xi}{ST\zeta}=\acou{S\eta}{\zeta} \qquad \text{for all $\zeta \in \dom T$}.
    \end{equation*}As $\dom ST = \dom T$ this shows that $(\xi,\eta)\in 
    G(T^{\dagger})$ if and only if $(\xi,S\eta)\in G\left( (ST)^{*}\right)$, i.e., $ST^{\dagger}=(ST)^{*}$. 
    We note that $S$ is a unitary operator since this is an isometric isomorphism. Thus, 
    \begin{equation*}
        ST^{\dagger}=(ST)^{*}=T^{*}S^{*}=T^{*}S^{-1}.
    \end{equation*}Furthermore, as $S$ is an 
    isomorphism,  we further see that $T=T^{\dagger}$ if and only if $ST=ST^{\dagger}=(ST)^{*}$. Thus $T$ is $\fs$-selfadjoint if and only if 
    $(\fs\otimes 1_{\cH})T$ is selfadjoint. The proof is complete. 
\end{proof}

\begin{definition}
Let $T$ be a closed densely defined operator $T$ from $\cH(\cE)$ to $\cH(\cE^{\sigma})$.
\begin{enumerate}
    \item   The spectrum of $(\fs \otimes 1_{\cH})T$ is called the $\fs$-spectrum of $T$.

    \item  An eigenvalue of $(\fs \otimes 1_{\cH})T$ is called an $\fs$-eigenvalue of $T$
\end{enumerate}
\end{definition}

From now on, we let $T$ be a densely defined operator from $\cH(\cE)$ to $\cH(\cE^{\sigma})$ which is Fredholm and 
$\fs$-selfadjoint.  Then $(\fs \otimes 1_{\cH})T$ is  a Fredholm operator, which is selfadjoint by Lemma~\ref{lem:STDagger=ST*}. 
Therefore, we obtain the following result. 

\begin{proposition}
     The $\fs$-spectrum of $T$ consists of a discrete set of real $\fs$-eigenvalues with finite 
    multiplicity.
\end{proposition}

It follows from this that the $\fs$-eigenvalues of $T$ can be arranged in a sequence $\left( 
\lambda_{j}(T)\right)_{j \geq 1}$ in such a way that each $\fs$-eigenvalue is repeated according to 
multiplicity and 
\begin{equation}
\label{eq:jthEigenvalueSD}
    \left| \lambda_{1}(T)\right| \leq \left| \lambda_{2}(T)\right| \leq \cdots.
\end{equation}

We shall now proceed to establish a max-min principle for $\fs$-eigenvalues. To this end we need to recall the 
definition of the characteristic values of $T$. The operator $T^{*}T$ 
is a selfadjoint densely defined operator of $\cH(\cE)$ which is Fredholm and has nonnegative spectrum. Thus its spectrum consists of an unbounded sequence of nonnegative 
eigenvalues with finite multiplicity. Let $|T|:=\sqrt{T^{*}T}$  be the absolute value of $T$. This is a selfadjoint operator of 
$\cH(\cE)$ with same domain as $T$ and its spectrum consists of an unbounded sequence of  nonnegative 
eigenvalues with finite multiplicity. For $j=1,2,\ldots$ we denote by $\mu_{j}(T)$ the $j$-th characteristic value of $T$, i.e., 
the $j$-th eigenvalue of $|T|$ counting with multiplicity. 

\begin{proposition}
\label{prop:max-min}
    For $j=1,2,\ldots$, we have 
    \begin{equation}
      \left| \lambda_{j}(T)\right|= \mu_{j}( T)=  \sup_{\substack{E\subset \cH\\ \dim E=j-1}} \inf\left\{  
    \left\|T\zeta\right\|; \ \zeta \in E^{\perp} \cap \dom 
    T, \  \|\zeta \|=1\right\}. 
     \label{eq:VF.min-maxT}
    \end{equation}
\end{proposition}
\begin{proof}
  The second equality is the classical max-min principle, so we only need to show the first equality.  Set $S=(\fs \otimes 1_{\cH})$. As $S$ is a unitary operator, we have
   \begin{equation*}
     T^{*}T= T^{*}S^{*}ST=(ST)^{*}ST=|ST|^{2},
   \end{equation*}and hence $|T |=\sqrt{T^{*}T}=|S T |$. Moreover, the selfadjointness of $ST$ implies 
   that $|\lambda_{j}(T)|$ is the $j$-th eigenvalue of $|ST|$. Thus $|\lambda_{j}(T)|=\mu_{j}(ST)=\mu_{j}(T)$. The proof is complete.
\end{proof}

\section{$\sigma$-Hermitian $\sigma$-Connections}
\label{sec:sigmaHermit-sigmaConnection}
In this section, we construct a class of $\sigma$-connections such that the associated coupled operators $D_{\nabla^{\cE}}$ are 
$\fs$-selfadjoint in the sense of the previous section. We shall continue using the notations of the previous section. 
In particular,  $\cE$ is a 
finitely generated projective right-module over $\cA$ together with a $\sigma$-translation 
$(\cE^{\sigma},\sigma^{\cE})$ 
and a $\sigma$-Hermitian structure $\left\{ \acoup\cdot\cdot, \fs\right\}$. 

\begin{definition}
\label{def:SigmaHermitianConnection}
    A $\sigma$-connection $\nabla^{\cE}:\cE\rightarrow \cE^{\sigma}\otimes 
    \Omega_{D,\sigma}^{1}$ on $\cE$ is $\sigma$-Hermitian when
    \begin{equation}
        \acoups{\xi}{\nabla^{\cE}\eta}-\sacoup{\nabla^{\cE}\xi}{\eta}=d_{\sigma}\left(\sacoup{\sigma^{\cE}(\xi)}{\eta}\right) \qquad \text{for all $\xi,\eta\in \cE$}.
            \label{eq:SigmaHermitianConnection}
    \end{equation}
\end{definition}

\begin{remark}
    In case $\sigma$ is trivial and $\fs=1_{\cE}$ , the above definition reduces to the usual  definition of a Hermitian connection (see~\cite{Co:NCG}). 
\end{remark}

\begin{lemma}
\label{lem:GrassmannianSigmaHermitianConnection}
    Suppose that $\cE=e\cA^{q}$ where $e\in M_{q}(\cA)$ is an idempotent such that $\sigma(e)=e^{*}$. We endow $\cE$ with 
    the $\sigma$-Hermitian structure as defined in the proof of Lemma~\ref{lem:SqrtSigmaSigmaHermitian}, so that the $\fs$-map is given 
    by~(\ref{eq:SigmaHermitianSigmaE=e*}). Then the Grassmannian $\sigma$-connection $\nabla_{0}^{\cE}$ is a $\sigma$-Hermitian $\sigma$-connection.
\end{lemma}
\begin{proof}
As mentioned in Example~\ref{ex:LRSigmaPairingCanonical} the pairings $\sacoup{\cdot}{\cdot}$ and $\acoups{\cdot}{\cdot}$ agree with the canonical Hermitian 
metric $\acoup{\cdot}{\cdot}_{0}$ of $\cA^{q}$ on their domains. Let 
    $\xi=(\xi_{j})$ and $\eta=(\eta_{j})$ be elements of $\cE$. As $\sigma(e)=e^{*}$ we have
    \begin{equation}
    \label{eq:LNablaPairing}
      \acoups{\xi}{\nabla_{0}^{\cE}\eta}  
      =\acoup{\xi}{e^{*}d_{\sigma}\xi}_{0}=\acoup{e\xi}{d_{\sigma}\eta}_{0}=\acoup{\xi}{d_{\sigma}\eta}_{0}=\sum 
      \xi_{j}^{*}d_{\sigma}\eta_{j}.
    \end{equation}Similarly, we have 
    \begin{equation}
        \label{eq:RNablaPairing}
      \sacoup{\nabla_{0}^{\cE}\xi}{\eta}= \acoup{e^{*}d_{\sigma}\xi}{\eta}_{0}=\acoup{d_{\sigma}\xi}{\eta}_{0}=\sum (d_{\sigma}\xi_{j})^{*}\eta_{j}. 
    \end{equation}
       
Let $a$ and $b$ be elements of $\cA$. Using~(\ref{eq:SigmaDerivation}) we see that
     \begin{gather*}
%       \begin{equation*}
            d_{\sigma}(\sigma(a)^{*}b)=d_{\sigma}(\sigma^{-1}(a^{*})b)=d_{\sigma}(\sigma^{-1}(a^{*}))b+a^{*}d_{\sigma}b,\\
%       \end{equation*}
%     and 
%   \begin{equation*}
            d_{\sigma}(\sigma^{-1}(a^{*}))=D\sigma^{-1}(a^{*})-a^{*}D=D\sigma(a)^{*}-a^{*}D=-(d_{\sigma}a)^{*}.
%   \end{equation*}
     \end{gather*}
  Thus,
    \begin{equation}
        d_{\sigma}(\sigma(a)^{*}b)=a^{*}d_{\sigma}b-(d_{\sigma}a)^{*}b.
    \label{eq:TwistedDerivationProperty}
    \end{equation}
     Combining (\ref{eq:LNablaPairing})--(\ref{eq:RNablaPairing}) with~(\ref{eq:TwistedDerivationProperty}) we see that 
    $\acoups{\xi}{\nabla_{0}^{\cE}\eta}-  \sacoup{\nabla_{0}^{\cE}\xi}{\eta}$ is equal to
    \begin{equation*}
       \sum    \left(\xi_{j}^{*}d_{\sigma}\eta_{j}-(d_{\sigma}\xi_{j})^{*}\eta_{j}\right) = \sum 
      d_{\sigma}(\sigma(\xi_{j})^{*}\eta_{j})= d_{\sigma}\acoup{\sigma^{\cE}(\xi)}{\eta}_{0}= 
      d_{\sigma}\left(\sacoup{\sigma^{\cE}(\xi)}{\eta}\right)
    \end{equation*}This shows that $\nabla_{0}^{\cE}$ is a $\sigma$-Hermitian $\sigma$-connection. The proof is complete. 
\end{proof}

\begin{remark}
    Let us denote by $\Hom^{\dagger}_{\cA}(\cE,\cE^{\sigma})$ the real vector space of right-module morphisms $A\in 
    \Hom_{\cA}(\cE,\cE^{\sigma})$ such that
    \begin{equation*}
        \sacoup{A\xi_{1}}{\xi_{2}}=\acoups{\xi_{1}}{A\xi_{2}} \qquad \forall \xi_{j}\in \cE.
    \end{equation*}Then the set of $\sigma$-Hermitian $\sigma$-connections is a real affine space modeled on 
    $\Hom^{\dagger}_{\cA}(\cE,\cE^{\sigma})$. Moreover, if $\nabla^{\cE}$ is a $\sigma$-connection, then there is a 
    unique $A\in \Hom^{\dagger}_{\cA}(\cE,\cE^{\sigma})$ such that $\nabla^{\cE}+iA$ is $\sigma$-Hermitian. Namely, $A$ 
    is defined by
    \begin{equation*}
        \acoup{A\xi}{\eta}=\frac{1}{2i}\left\{\acoups{\xi}{\nabla^{\cE}\eta}-\sacoup{\nabla^{\cE}\xi}{\eta}- 
        d_{\sigma}\left(\sacoup{\sigma^{\cE}(\xi)}{\eta}\right) \right\} \qquad \text{for all $\xi,\eta\in \cE$}.
    \end{equation*}
\end{remark}

\begin{proposition}\label{prop:HermitianConnSA}
    Let $\nabla^{\cE}$ be a $\sigma$-Hermitian $\sigma$-connection on $\cE$. Then the operator $D_{\nabla^{\cE}}$ is 
    $\fs$-selfadjoint and there is a $\Z_{2}$-graded isomorphism $\coker D_{\nabla^{\cE}}^{\pm}\simeq \ker 
    D_{\nabla^{\cE}}^{\mp}$.
\end{proposition}
\begin{proof}
Let us first show that $D_{\nabla^{\cE}}^{\dagger}$ is an extension of $D_{\nabla^{\cE}}$. For $j=1,2$ let $\xi_{j}\in 
\cE$ and $\zeta_{j}\in \dom D$. Using~(\ref{eq:LRSigmaPairingHilbertSpace}) and (\ref{eq:SigmaHermitianConnection}) we get 
\begin{align*}
    \acous{\xi_{1}\otimes \zeta_{1}}{(\nabla^{\cE}\xi_{2})\zeta_{2}}-  
    \sacou{(\nabla^{\cE}\xi_{1})\zeta_{1}}{\xi_{2}\otimes \zeta_{2}}& = 
    \left\langle{\zeta_{1}},{\left(\acoups{\xi_{1}}{\nabla^{\cE}\xi_{2}}-\sacoup{\nabla^{\cE}\xi_{1}}{\xi_{2}}\right)\zeta_{2}}\right\rangle\\
    & =  \left\langle{\zeta_{1}},{d_{\sigma}\left(\sacoup{\sigma^{\cE}(\xi_{1})}{\xi_{2}}\right)\zeta_{2}}\right\rangle.
\end{align*}
Thanks to~(\ref{eq:RelationLRSigmaPairing}) we see that $d_{\sigma}\left(\sacoup{\sigma^{\cE}(\xi_{1})}{\xi_{2}}\right)$ is equal to
\begin{equation*}
  D\left[\sacoup{\sigma^{\cE}(\xi_{1})}{\xi_{2}}\right]- 
  \sigma\left[ \sacoup{\sigma^{\cE}(\xi_{1})}{\xi_{2}}\right]D
 = D\left[\sacoup{\sigma^{\cE}(\xi_{1})}{\xi_{2}}\right]- 
  \left[\acoups{\xi_{1}}{\sigma^{\cE}(\xi_{2})}\right]D. 
\end{equation*}
Thus $\acous{\xi_{1}\otimes \zeta_{1}}{(\nabla^{\cE}\xi_{2})\zeta_{2}}-  
    \sacou{(\nabla^{\cE}\xi_{1})\zeta_{1}}{\xi_{2}\otimes \zeta_{2}}$ is equal to
\begin{multline*}
 \acou{\zeta_{1}}{ D\sacoup{\sigma^{\cE}(\xi_{1})}{\xi_{2}}\zeta_2}- 
   \acou{\zeta_{1}}{\acoups{\xi_{1}}{\sigma^{\cE}(\xi_{2})}D\zeta_{2}} \\
    =  \acou{D\zeta_{1}}{ \sacoup{\sigma^{\cE}(\xi_{1})}{\xi_{2}}\zeta_2}- 
   \acou{\zeta_{1}}{\acoups{\xi_{1}}{\sigma^{\cE}(\xi_{2})}D\zeta_{2}}\\
   =
   \sacou{\sigma^{\cE}(\xi_{1})\otimes (D\zeta_{1})}{\xi_{2}\otimes \zeta_{2}}- 
   \acous{\xi_{1}\otimes \zeta_{1}}{\sigma^{\cE}(\xi_{2})\otimes (D\zeta_{2})} .
\end{multline*}
 As $D_{\nabla^{\cE}}(\xi_{j}\otimes \zeta_{j})=\sigma^{\cE}(\xi_{j})\otimes 
 (D\zeta_{j})+c(\nabla^{\cE})(\xi_{j}\otimes\zeta_{j})$ we deduce that 
 \begin{align*}
     \acous{\xi_{1}\otimes \zeta_{1}}{D_{\nabla^{\cE}}(\xi_{2}\otimes \zeta_{2})} &= 
     \acous{\xi_{1}\otimes \zeta_{1}}{\sigma^{\cE}(\xi_{2})\otimes  (D\zeta_{2})} + \acous{\xi_{1}\otimes \zeta_{1}}{c(\nabla^{\cE})(\xi_{2}\otimes\zeta_{2})}\\ 
 & = \sacou{\sigma^{\cE}(\xi_{1})\otimes  (D\zeta_{1})}{\xi_{2}\otimes \zeta_{2}} + 
 \sacou{c(\nabla^{\cE})(\xi_{1}\otimes\zeta_{1})}{\xi_{2}\otimes \zeta_{2}}\\
 &= \sacou{D_{\nabla^{\cE}}(\xi_{1}\otimes \zeta_{1})}{\xi_{2}\otimes \zeta_{2}} .
 \end{align*}
This shows that the graph of $D_{\nabla^{\cE}}$ is contained in that of $D_{\nabla^{\cE}}^{\dagger}$, i.e., 
$D_{\nabla^{\cE}}^{\dagger}$ is an extension of $D_{\nabla^{\cE}}$. 

In order to  show that the operators $D_{\nabla^{\cE}}$ and $D_{\nabla^{\cE}}^{\dagger}$ agree it remains to show that 
they  have the same domain. We note that by Lemma~\ref{lem:STDagger=ST*} we have
\begin{equation}
\label{eq:D_nablaE^dagger&*}
    D_{\nabla^{\cE}}^{\dagger}=(\fs^{-1}\otimes 1_{\cH})D^{*}_{\nabla^{\cE}}(\fs^{-1}\otimes 1_{\cH}),
\end{equation}
where $D^{*}_{\nabla^{\cE}}$ is the adjoint of $D_{\nabla^{\cE}}$. Therefore, it is enough to look 
at the domain of $D^{*}_{\nabla^{\cE}}$.

\begin{claim*}
    For any Hermitian metrics on $\cE$ and $\cE^{\sigma}$ and any $\sigma$-connection $\nabla^{\cE}$ on $\cE$, the 
    domain of $D_{\nabla^{\cE}}^{*}$ is $\cE^{\sigma}\otimes_{\cA}\dom D$.
\end{claim*}
\begin{proof}[Proof of the claim]
Let us first assume that  $\cE=e\cA^{q}$ with $e=e^{2}\in M_{q}(\cA)$.  We note that the domain of $D_{\nabla^{\cE}}^{*}$ is independent of the choice of the Hermitian metrics on 
$e\cA^{q}$ and $\sigma(e)\cA^{q}$. Indeed, a change of Hermitian metrics amounts to replacing $D_{\nabla^{\cE}}^{*}$ by 
$(a\otimes 1_{\cH}) D_{\nabla^{\cE}}^{*}(b\otimes 1_{\cH})$ for suitable elements $a$ and $b$ of $\GL_{q}(\cA)$. 
However, this does not affect the domain of  $D_{\nabla^{\cE}}^{*}$. 

We also note that the domain of $D_{\nabla^{\cE}}^{*}$ is actually independent of the choice of the 
$\sigma$-connection $\nabla^{\cE}$. Indeed, if $\tilde{\nabla}^{\cE}$ is another $\sigma$-connection on $\cE=e\cA^{q}$, then 
it differs from $\nabla^{\cE}$ by an element of $\Hom_{\cA}\left(\cE, 
\cE\otimes_{\cA}\Omega_{D,\sigma}^{1}(\cA)\right)$. Incidentally, the operators $D_{\nabla^{\cE}}$ and 
$D_{\tilde{\nabla}^{\cE}}$ agree up to a bounded operator, and  so do their adjoints. Thus $D_{\nabla^{\cE}}^{*}$ and 
$D_{\tilde{\nabla}^{\cE}}^{*}$ have same domain.

It follows from these observations that we may assume that $\nabla^{\cE}$ is the Grassmannian $\sigma$-connection 
$\nabla^{\cE}_{0}$ and the Hermitian metrics of $\cE=e\cA^{q}$ and $\cE^{\sigma}=\sigma(e)\cA^{q}$ are the 
Hermitian metrics induced by  the canonical Hermitian metric of $\cA^{q}$. In this case, Eq.~(\ref{eq:sigma-connections.DnablaE0}) shows there are 
unitary operators $U_{e}:\cH(\cE)\rightarrow e\cH^{q}$ and $U_{\sigma(e)}:\cH(\cE^{\sigma})\rightarrow 
\sigma(e)\cH^{q}$ such that $D_{\nabla^{\cE}_{0}}=U_{\sigma(e)}^{-1}D_{e,\sigma}U_{e}$. Thus, 
\begin{equation}
\label{eq:D_nabla_0^E*&D_sigma(e)*sigma}
   D_{\nabla^{\cE}_{0}}^{*}=\left(U_{\sigma(e)}^{-1}D_{e,\sigma}U_{e} 
   \right)^{*}=U_{e}^{-1}D_{e,\sigma}^{*}U_{\sigma(e)}.
\end{equation}
In addition, as shown in the proof of~\cite[Lemma~4.2]{PW:Index}, we have
\begin{equation}
\label{eq:D_esigma*&D_sigma(e)*sigma}
    D_{e,\sigma}^{*}=S_{e^{*}}^{-1}D_{\sigma(e)^{*},\sigma}S_{\sigma(e)^{*}},
\end{equation}
where the isomorphism $S_{e^{*}}:e\cH^{q}\rightarrow e^{*}\cH^{q}$ (resp., $S_{\sigma(e)^{*}}:e\cH^{q}\rightarrow 
e^{*}\cH^{q}$) is given by the restriction of $e^{*}$ (resp., $\sigma(e)^{*}$) to $e\cH^{q}$ (resp., 
$\sigma(e)^{*}\cH^{q}$). 
Noting that $U_{\sigma(e)}$ maps $\sigma(e)\cA^{q}\otimes_{\cA}\dom D$ to $\sigma(e)(\dom D)^{q}$ and 
$S_{\sigma(e)^{*}}$ maps $\sigma(e)(\dom D)^{q}$ to $\sigma(e)^{*}(\dom D)^{q}$, from~(\ref{eq:D_nabla_0^E*&D_sigma(e)*sigma}) 
and~(\ref{eq:D_esigma*&D_sigma(e)*sigma}) we deduce that $\dom D_{\nabla_{0}^{\cE}}^{*}$ is equal to
\begin{equation*}
    U_{\sigma(e)}^{-1}\left( 
    S_{\sigma(e)^{*}}^{-1} \left( \dom D_{\sigma(e)^{*},\sigma}\right)\right)= 
    U_{\sigma(e)}^{-1}\left( 
    S_{\sigma(e)^{*}}^{-1} \left( \sigma(e)^{*}(\dom D)^{q}\right)\right)=\sigma(e)\cA^{q}\otimes_{\cA} \dom D.
\end{equation*}
This proves the claim when $\cE=e\cA^{q}$ for some idempotent $e\in M_{q}(\cA)$, $q\geq 1$. 

When $\cE$ is a general finitely generated projective module, by Remark~\ref{rmk:sigma-translate} there always are an idempotent $e\in M_{q}(\cA)$, $q\geq 
1$, and right module isomorphisms $\phi:\cE\rightarrow e\cA^{q}$ and $\phi_{\sigma}:\cE^{\sigma}\rightarrow  \sigma(e)\cA^{q}$ 
such that $\phi_{\sigma}\circ \sigma^{\cE}=\sigma^{e}\circ \phi$. Using $\phi$ and $\phi_{\sigma}$ we pushforward the Hermitian metrics of $\cE$ and 
$\cE^{\sigma}$ to Hermitian metrics on $\cF:=e\cA^{q}$ and $\cF^{\sigma}=\sigma(e)\cA^{q}$, respectively. We also pushforward 
$\nabla^{\cE}$ to the $\sigma$-connection on $\cF$ defined by
    \begin{equation*}
        \nabla^{\cF}:=\left(\phi_{\sigma}\otimes 1_{\Omega_{D,\sigma}^{1}(\cA)}\right)\circ \nabla^{\cE}\circ 
        \phi^{-1}.
    \end{equation*}
Set $U_{\phi}=\phi\otimes 1_{\cE}\in \cL\left(\cH(\cE),\cH(\cF)\right)$ and $U_{\phi_{\sigma}}=\phi_{\sigma}\otimes 
1_{\cE}\in \cL\left(\cH(\cE^{\sigma}),\cH(\cF^{\sigma})\right)$. Both $U_{\phi}$ and $U_{\phi_{\sigma}}$ are unitary 
operators. Moreover, it can be shown (see~\cite[Eq.~(5.17)]{PW:Index}) that 
\begin{equation}
    D_{\nabla^{\cE}}=  U_{\phi_{\sigma}}^{-1}D_{\nabla^{\cF}}U_{\phi}.
\end{equation}
It then follows that $D_{\nabla^{\cE}}^{*}=U_{\phi}^{-1}D_{\nabla^{\cF}}^{*}U_{\phi_{\sigma}}$. In particular, the domain of 
$D_{\nabla^{\cE}}^{*}$ is equal to
\begin{equation*}
U_{\phi_{\sigma}}^{-1}( \dom   D_{\nabla^{\cF}})=(\phi_{\sigma}^{-1}\otimes 1_{\cH} )(\cF^{\sigma}\otimes_{\cA} \dom 
D)=\cE^{\sigma}\otimes_{\cA}\dom D.
\end{equation*}
The claim is thus proved. 
\end{proof}

Combining the above claim with~(\ref{eq:D_nablaE^dagger&*}) we obtain
\begin{equation*}
    \dom D_{\nabla^{\cE}}^{\dagger}=(\fs \otimes 1_{\cH})\left( \dom D_{\nabla^{\cE}}^{*}\right) = (\fs \otimes 
    1_{\cH})\left( \cE^{\sigma}\otimes_{\cA}\dom D\right)=\cE\otimes_{\cA}\dom D=\dom D_{\nabla^{\cE}}.
\end{equation*}This shows that the operators $D_{\nabla^{\cE}}^{\dagger}$ and $D_{\nabla^{\cE}}$ agree, i.e., 
$D_{\nabla^{\cE}}$ is $\fs$-selfadjoint. Moreover, by Lemma~\ref{lem:STDagger=ST*} the $\fs$-selfadjointness of $D_{\nabla^{\cE}}$ implies the selfadjointness of $(\fs \otimes 1_{\cH})D_{\nabla^{\cE}}$. Thus,
\begin{equation*}
    \coker D_{\nabla^{\cE}}\simeq \coker (\fs \otimes 1_{\cH})D_{\nabla^{\cE}}\simeq \ker \left((\fs \otimes 
    1_{\cH})D_{\nabla^{\cE}}\right)^{*}=\ker (\fs \otimes 1_{\cH})D_{\nabla^{\cE}}=\ker D_{\nabla^{\cE}}.
\end{equation*}
We note that $D_{\nabla^{\cE}}$ is an odd operator and $\fs\otimes 1_{\cH}$ is an even operator, so we get graded 
isomorphisms by using the $\Z_{2}$-gradings $\ker D_{\nabla^{\cE}}= \ker D_{\nabla^{\cE}}^{+}\oplus \ker 
D_{\nabla^{\cE}}^{-}$ and $\coker D_{\nabla^{\cE}}= \coker D_{\nabla^{\cE}}^{-}\oplus \coker 
D_{\nabla^{\cE}}^{+}$. That is, $\coker D_{\nabla^{\cE}}^{\pm}\simeq \ker D_{\nabla^{\cE}}^{\mp}$. The proof is 
complete.
\end{proof}

\begin{remark}
 In the special case where $\sigma=\op{id}_\cA$ and we take $\cE^\sigma=\cE$ and $\fs=\op{id}_\cE$, we recover from Proposition~\ref{prop:HermitianConnSA} the property that  the coupled operator $D_{\nabla^\cE}$ associated with a Hermitian connection is selfadjoint. 
\end{remark}

The isomorphisms $\coker D_{\nabla^{\cE}}^{\pm}\simeq \ker D_{\nabla^{\cE}}^{\mp}$ imply that
\begin{equation*}
    \ind D_{\nabla^{\cE}}^{+}=\dim \ker D_{\nabla^{\cE}}^{+}-\dim \ker D_{\nabla^{\cE}}^{-}=-\ind 
    D_{\nabla^{\cE}}^{-}.
\end{equation*}Combining this with Proposition~\ref{thm.IndexTwisted-connection} we arrive at the following index formula. 

\begin{proposition}
\label{prop:Ind_DSigmaHermitianConnection}
   Let $\cE$ be a $\sigma$-Hermitian finitely generated projective right module over $\cA$. Then, for any 
   $\sigma$-Hermitian $\sigma$-connection $\nabla^{\cE}$ on $\cE$, we have
\begin{equation*} 
    \ind_{D,\sigma}[\cE]=\dim \ker D_{\nabla^{\cE}}^{+}-\dim \ker D_{\nabla^{\cE}}^{-}.
        \end{equation*}
\end{proposition}

Using the obvious inequality,
\begin{equation*}
    \dim \ker D_{\nabla^{\cE}}= \dim \ker D_{\nabla^{\cE}}^{+}+\dim \ker D_{\nabla^{\cE}}^{-}\geq \left| \dim \ker 
    D_{\nabla^{\cE}}^{+}-\dim \ker D_{\nabla^{\cE}}^{-}\right|,
\end{equation*}we obtain the following corollary. 

\begin{corollary}
\label{cor:indexneq0nontrivialSolution}
 Let $\cE$ be a $\sigma$-Hermitian finitely generated projective right module over $\cA$.   If 
 $\ind_{D,\sigma}[\cE]\neq 0$, then, for 
 any $\sigma$-Hermitian $\sigma$-connection $\nabla^{\cE}$ on $\cE$, the equation 
    $D_{\nabla^{\cE}}u=0$ has nontrivial solutions.
\end{corollary}

\section{Poincar\'e Duality for Ordinary Spectral Triples}\label{sec:spectral-triples}
There are various versions of Poincar\'e duality in noncommutative geometry. We refer to~\cite{BMRS:CMP} and the 
references therein for a survey  and comparison of these various notions. In its strongest form Poincar\'e duality is 
formulated in terms of Kasparov's bivariant $K$-theory~\cite{Ka:Invent88}. However, for our purpose we only need to use a rational version of Poincar\'e duality 
in the sense of~\cite{Co:NCG, Co:NCGR, Mo:EIPDNCG}. 

% \subsection{Poincar\'e duality for ordinary spectral triples}
Let $(\cA_{1}, \cH, D)$ and $(\cA_{2}, \cH, D)$ be ordinary spectral triples such that
\begin{gather}
    \text{The algebras $\cA_1$ and $\cA_{2}$ commute in $\cL(\cH)$},
    \label{PDAlgebraCommute}\\
   [[D, a_{1}], a_{2}]=0 \qquad  \forall a_{j}\in\cA_{j}.
    \label{PDOrder1Condition}
\end{gather}
The first condition ensures us that the algebra $\cA_1\otimes\cA_2$ is represented in $\cH$ by $a_1\otimes a_2\rightarrow a_1a_2$. 
The second condition implies that, for all $a_j\in\cA_j$, 
\begin{equation*}
[D, a_1\otimes a_2]=[D, a_1a_2]=[D, a_1]a_2+a_1[D, a_2]\in\cL(\cH).
\end{equation*} 
It then follows that $(\cA_1\otimes\cA_2, \cH, D)$ is a spectral triple, and hence there is a well-defined index map 
$\ind_{D}:K_{0}(\cA_1\otimes \cA_2)\rightarrow \Z$. Composing this index map with 
the natural bi-additive map $K_{0}(\cA_1)\times K_{0}(\cA_2)\rightarrow K_{0}(\cA_1\otimes\cA_2)$ and taking tensor products 
with $\Q$ we get a bilinear pairing, 
\begin{equation}
\label{eq:BilinearPairing0}
   (\cdot,\cdot)_D: (K_{0}(\cA_1)\otimes\Q)\times (K_{0}(\cA_2)\otimes\Q) \longrightarrow \Q.
\end{equation}

\begin{definition}(See \cite{Mo:EIPDNCG}.)
\label{def:STPD}
$(\cA_{1}, \cH, D)$ is in Poincar\'e duality with $(\cA_{2}, \cH, D)$ when 
the conditions~(\ref{PDAlgebraCommute})--(\ref{PDOrder1Condition}) hold and the pairing $(\cdot,\cdot)_{D}$ is nondegenerate.
\end{definition}

\begin{remark}
\label{rem:PDbidualIdentity}
  For $a_{1}\in \cA_{1}$ and $a_{2}\in \cA_{2}$, Jacobi's identity gives $[[D,a_{2}],a_{1}]=[[D,a_{1}],a_{2}]+[[a_{1},a_{2}],D]$. Therefore,  
  the conditions~(\ref{PDAlgebraCommute})--(\ref{PDOrder1Condition}) are symmetric with respect to the 
  algebras $\cA_{1}$ and $\cA_{2}$. Thus, if $(\cA_{1}, \cH, D)$ is in Poincar\'e duality with $(\cA_{2}, 
  \cH, D)$, then $(\cA_{2},\cH, D)$ is in Poincar\'e duality with $(\cA_{1}, \cH, D)$. 
\end{remark}

\begin{example}(See {\cite[VI.4.{$\beta$}]{Co:NCG}}.)
\label{ex:DiracSpectralTriplePD}
     Given a closed Riemannian spin manifold $M$ of even dimension, the Dirac spectral triple 
     $(C^{\infty}(M),L^{2}_{g}(M,\sS),\sD_{g})$  is in 
     Poincar\'e duality with itself. 
\end{example}

\begin{example}(See {\cite[VI.4.{$\beta$}]{Co:NCG}}.)
Let $(M,g)$ be a closed Riemannian oriented manifold  of even dimension. Consider the signature spectral triple, 
\begin{equation*}
    \left( C^{\infty}(M), L^{2}(M, \Lambda^{*}_{\C}T^{*}M), d+d^{*}\right),
\end{equation*}
where $d$ is the de Rham differential and  $\Lambda^{*}_{\C}T^{*}M$ has the $\Z_{2}$-grading given by the Hodge 
$\star$-operator. This spectral triple is in Poincar\'e duality with 
\begin{equation*}
    \left( C^{\infty}\left(M,\op{Cl}_{\C}(T^{*}M)\right), L^{2}(M, \Lambda^{*}_{\C}T^{*}M), d+d^{*}\right),
\end{equation*}where $\op{Cl}_{\C}(T^{*}M)$ is the Clifford bundle of $M$. This duality 
continues to hold if we only assume $M$ to be a Lipschitz manifold and 
replace the algebras $C^{\infty}(M)$ and $C^{\infty}\left(M,\op{Cl}_{\C}(T^{*}M)\right)$ by their Lipschitz versions 
(see~\cite{Hi:FKTBVL}). 
\end{example}

\begin{example}(See {\cite[{VI.4.{$\beta$}}]{Co:NCG}}.)
\label{ex:NCTorusPD}
Over a noncommutative torus $\cA_{\theta}$, $\theta \in \R$, the spectral triples $(\cA_{\theta},\cH,D)$ and 
$(\cA_{\theta}^{\opp},\cH,D)$ described in Section~\ref{subsec:TSTNCTorus} are in Poincar\'e duality. 
\end{example}

\begin{example}(See {\cite[{IV.9.{$\alpha$}}]{Co:NCG}}, {\cite[Section 5]{Mo:EIPDNCG}}.)
\label{ex:BaumConnesST}
Let $\Gamma$ be a torsion-free cocompact discrete subgroup of a connected semisimple Lie group $G$. 
Consider the symmetric space $X=G/K$,  where $K$ is a maximal compact subgroup of $G.$ We endow $X$ with its canonical $G$-invariant metric and 
denote by $\cH=L^2(X,\Lambda_\C^*T^*X)$ the Hilbert space of $L^2$-forms on $X$.  We equip $\cH$ with the 
$\Z_{2}$-grading given by the parity of the degrees of forms. 

The group $\Gamma$ acts isometrically on $\cH$ by left translations preserving this $\Z_{2}$-grading. This action thus gives rise to an even unitary representation of the reduced 
$C^{*}$-algebra $C_{r}^{*}(\Gamma)$ in $\cH$. We then let $\cA_\Gamma$ be 
the closure of $\C\Gamma$ in $C^{*}_{r}(\Gamma)$ under holomorphic functional calculus. 

We also note that the action of 
$\Gamma$ on $X$ is free and proper, so that the quotient $\Gammab X $ is a manifold over which the canonical projection 
$\pi: X\rightarrow \Gammab X$ is a smooth fibration. Let $C^{\infty}(X)^{\Gamma}$ be the space of  smooth $\Gamma$-periodic functions on $X$. Any function $a \in 
C^{\infty}(\Gammab X)$ lifts to the $\Gamma$-periodic function $\tilde{a}=a\circ \pi\in C^{\infty}(X)^{\Gamma}$. Conversely, any smooth 
$\Gamma$-periodic function $\tilde{a}$ on $X$ descends to a function $a\in C^{\infty}(\Gammab X)$ such that 
$\tilde{a}=a\circ \pi$. It then follows that $C^{\infty}(\Gammab X)$ acts on $\cH$ by
\begin{equation}
    (a\zeta)(x):=\tilde{a}(x)\zeta(x), \qquad a \in C^{\infty}(\Gammab X), \quad \zeta \in \cH. 
    \label{eq:PD.duals-cocompact.action-cB}
\end{equation}
We note that this action commutes with the action of $\cA_\Gamma$. 

In addition, let $\varphi(x)$ be the Morse function given by the square of the geodesic distance from $x$ to the base 
point $o=\{K\}\in 
X$. Following Witten~\cite{Wi:JDG} we define
\begin{equation}
D_\tau:=d_{\tau}+d_{\tau}^*,\qquad d_{\tau}:=e^{-\tau \varphi}d \, e^{\tau \varphi}, \ \tau\neq 0.
\label{eq:PD.dual-group-Dtau}
\end{equation}
As it turns out, $(\cA_\Gamma,\cH,D_\tau)$ and $\left(C^{\infty}(\Gamma\backslash X),\cH,D_\tau\right)$ are spectral triples satisfying the 
condition~(\ref{PDOrder1Condition}) (see~\cite[IV.9.{$\alpha$}]{Co:NCG}). In addition, 
the Poincar\'e duality pairing~(\ref{eq:BilinearPairing0}) can be expressed in terms of the Baum-Connes assembly map, 
\begin{equation}
% \label{eq:BaumConnesMap}
\mu_{r}^{\Gamma}: K_{0}(B\Gamma)\rightarrow K_0(C^{*}_{r}(\Gamma)), 
\end{equation}which was conjectured by Baum-Connes~\cite{BCC:1982, BC:EM} to be an isomorphism. 
More precisely, by~\cite[Theorem~IV.9.4]{Co:NCG} it holds that, for any $K$-homology class $x \in K_{0}(B\Gamma)$ and $K$-theory class $y\in 
K_{0}\left(C^{\infty}(\Gamma\backslash X)\right)=K^{0}(\Gamma\backslash X)$, we have
\begin{equation*}
    (\mu_{r}^{\Gamma}(x),y)_{D_{\tau}}=\acou{\Ch_{*}(x)}{\Ch^{*}(y)},
\end{equation*}where $\Ch_{*}(x)$ is the Chern character in $H_{*}(B\Gamma)=H_{*}(\Gamma\backslash X)$ and 
$\Ch^{*}(y)$ is the Chern character in $H^{*}(\Gamma\backslash X)$. As the Chern character maps are rational 
isomorphisms, it then follows that the pairing $(\cdot, \cdot)_{D_{\tau}}$ is nondegenerate whenever the Baum-Connes assembly map is an isomorphism. 

The Baum-Connes conjecture holds for discrete cocompact subgroups of $SO(n,1)$ 
(Kasparov~\cite{Ka:LGKTURCP}) and $SU(n,1)$ (Julg-Kasparov~\cite{JK:OKTGSUn1}). More generally, thanks to the results of V.\ Lafforgue~\cite{La:KTBABCBC} it holds 
when $\Gamma$ is a hyperbolic group or a discrete subgroup with the rapid decay property. In particular, it holds for discrete cocompact 
subgroups of $Sp(n,1)$, $SL(3,\R)$, $SL(3,\C)$ and  rank $1$ real Lie groups. 
\end{example}

\begin{remark}
  Poincar\'e duality is also satisfied by spectral triples over Podle\'s quantum 
spheres~\cite{DS:DOSPSQS, Wa:NCSGSPSIC} and 
quantum projective lines~\cite{DL:GQPS}, and by spectral triples describing the standard model of particle physics~\cite{Co:NCGR, Co:GCMFNCG, CCM:GSTNM}.
\end{remark}

\section{Poincar\'e Duality for Twisted Spectral Triples}
\label{sec:Poincare Duality for Twisted Spectral Triples}
In this section, we define a notion of Poincar\'e duality for twisted spectral triples and present various examples. In the next section we will give a geometric interpretation of this duality in 
terms of $\sigma$-Hermitian $\sigma$-connections. 

Let $(\cA_{1}, \cH, D)_{\sigma_{1}}$ and $(\cA_{2}, \cH, D)_{\sigma_{2}}$ be twisted spectral triples such that
\begin{gather}
    \text{The algebras $\cA_{1}$ and $\cA_{2}$ commute in $\cL(\cH)$},
    \label{PDTAlgebraCommute}\\
  [[D, a_1]_{\sigma_1}, a_2]_{\sigma_2}=0 \qquad  \forall a_{j}\in\cA_{j}, j=1, 2.
    \label{PDTOrder1Condition}
\end{gather}

The algebra $\cA_{1}\otimes\cA_{2}$ is represented in $\cH$ by $a_1\otimes a_2\rightarrow a_1a_2$. The tensor product 
$\sigma:=\sigma_{1}\otimes \sigma_{2}$ is an automorphism of $\cA_1\otimes\cA_2$. Moreover, for all $a_{j}\in \cA_{j}$, 
\begin{equation}
\label{eq:CommutatorDa_1a_2}
    [D, a_1\otimes a_2]_{\sigma}=[D, a_1a_2]_{\sigma}=[D, a_1]_{\sigma_1}a_2+\sigma_1(a_1)[D, a_2]_{\sigma_2}\in \cL(\cH).
\end{equation} 
It then follows that $(\cA_1\otimes\cA_2, \cH, D)_{\sigma}$ is a twisted spectral triple. Composing its index map 
$\ind_{D,\sigma}:K_{0}(\cA_{1}\otimes \cA_{2})\rightarrow \frac{1}{2}\Z$ with the natural bi-additive map 
$K_{0}(\cA_1)\times K_{0}(\cA_2)\rightarrow K_{0}(\cA_1\otimes\cA_2)$ we get a bilinear pairing, 
\begin{equation}
\label{eq:BilinearPairing}
   (\cdot,\cdot)_{D,\sigma}: (K_{0}(\cA_1)\otimes \Q)\times (K_{0}(\cA_2)\otimes \Q) \longrightarrow \Q.
\end{equation}

\begin{definition}
\label{def:TSTRationalPoincareDual}
$(\cA_{1}, \cH, D)_{\sigma_{1}}$ is in Poincar\'e duality with $(\cA_{2}, \cH, D)_{\sigma_2}$ when 
the conditions~(\ref{PDTAlgebraCommute})--(\ref{PDTOrder1Condition}) hold and the pairing $(\cdot,\cdot)_{D,\sigma}$ is nondegenerate.
\end{definition}
\begin{remark}
 Let $a_{j}\in \cA_{j}$, $j=1,2$. Then 
 \begin{equation*}
     [[D,a_{2}]_{\sigma_{2}},a_{1}]_{\sigma_{1}}= 
     [[D,a_{1}]_{\sigma_{1}},a_{2}]_{\sigma_{2}}+D[a_{1},a_{2}]-[\sigma_{1}(a_{1}),\sigma_{2}(a_{2})]D.
 \end{equation*}Therefore, in the same way as in Remark~\ref{rem:PDbidualIdentity}, we see that the 
 conditions~(\ref{PDTAlgebraCommute})--(\ref{PDTOrder1Condition}) are symmetric with respect to $\cA_{1}$ and $\cA_{2}$ and Poincar\'e duality for twisted 
 spectral triples is reflexive. 
\end{remark}

As we shall now see, pseudo-inner twistings of ordinary Poincar\'e dual pairs provide us with a  wealth of examples of Poincar\'e 
duality between twisted spectral triples. 

Let $(\cA_1, \cH, D)$ and $(\cA_2, \cH, D)$ be ordinary spectral triples that are in Poincar\'e duality. 
It is convenient to regard $\cA_{1}$ and $\cA_{2}$ as subalgebras of $\cA:=\cA_{1}\otimes \cA_{2}$. 
As mentioned above $(\cA, \cH, D)$ is an ordinary spectral triple.  Let $\omega= 
\begin{pmatrix}
    \omega^{+} & 0 \\
    0 & \omega^{-}
\end{pmatrix}\in \cL(\cH)$ be a pseudo-inner twisting operator for both spectral triples $(\cA_1, \cH, D)$ and $(\cA_2, \cH, D)$. Thus, there are positive invertible elements 
$k_{j}^{\pm}\in \cA_{j}$, $j=1,2$, such that $k_{j}^{+}k_{j}^{-}= k_{j}^{-}k_{j}^{+}$ and $ \omega^{\pm} a 
(\omega^{\pm})^{-1}=\sigma_{j}^{\pm}(a)$ for all $a\in \cA_{j}$, where 
$\sigma^{\pm}_{j}(a):=k_{j}^{\pm}a \left(k^{\pm}_{j}\right)^{-1}$. Denote by $\sigma_{j}$ the automorphism of $\cA_{j}$ given by $\sigma_{j}(a)=k_{j}ak_{j}^{-1}$, $a \in \cA_{j}$, where 
$k_{j}:=k_{j}^{+}k_{j}^{-}$. 
Setting $D_{\omega}=\omega D \omega$, both   $(\cA_{1},\cH,D_{\omega})_{\sigma_{1}}$ and  
$(\cA_{2},\cH,D_{\omega})_{\sigma_{2}}$ are twisted spectral triples by Proposition~\ref{Prop:ConformalPerturbationw}.

\begin{proposition}\label{Prop:TSTPD}
$(\cA_1, \cH, D_{\omega})_{\sigma_1}$ and $(\cA_2, \cH, D_{\omega})_{\sigma_2}$ are in 
Poincar\'e duality.
\end{proposition}
\begin{proof}
 We already know that the algebras $\cA_{1}$ and $\cA_{2}$ commute with each other. Moreover, as $(\cA_1, \cH, D)$ 
 and $(\cA_2, \cH, D)$ are in Poincar\'e duality $[[D,a_{1}],a_{2}]=0$ for all $a_{j}\in \cA_{j}$. 
 
 Let $a_{j}\in  \cA_{j}$, $j=1,2$. By~(\ref{eq:Commutator1}) and (\ref{eq:Commutator2}) we have
 \begin{equation*}
  [D_{\omega}^{+},a_{1}]_{\sigma_{1}}= \omega 
  \begin{pmatrix}
      0 & [D^{-},\sigma^{-}_{1}(a_{1})] \\
      [D^{+},\sigma^{+}_{1}(a_{1})] & 0
  \end{pmatrix}\omega. 
 \end{equation*}
In fact, by arguing as~(\ref{eq:Commutator1}) and (\ref{eq:Commutator2}) we further see that
 \begin{equation*}
     [[D_{\omega}^{+},a_{1}]_{\sigma_{1}},a_{2}]_{\sigma_{2}}= \omega \begin{pmatrix}
      0 & [[D^{-},\sigma^{-}_{1}(a_{1})],\sigma^{-}_{2}(a_{2})] \\
      [[D^{+},\sigma^{+}_{1}(a_{1})],\sigma^{+}_{1}(a_{1})] & 0
  \end{pmatrix}\omega =0.
 \end{equation*}
Thus the condition~(\ref{PDTOrder1Condition}) is satisfied. Incidentally, if we set  $\sigma=\sigma_{1}\otimes \sigma_{2}$, then $\left(\cA,\cH,D_{\omega}\right)_{\sigma}$ is a 
 twisted spectral triple. 
 
Set $k^{\pm}=k_{1}^{\pm}\otimes k_{2}^{\pm}=k_{1}^{\pm}k_{2}^{\pm}$. These are positive invertible 
elements of $\cA$ and are commuting with each other. Set $k=k^{+}k^{-}=k_{1}k_{2}$. Then, for all 
$a_{j}\in \cA_{j}$, 
\begin{gather*}
    \sigma(a_{1}\otimes 
    a_{2})=\sigma_{1}(a_{1})\sigma_{2}(a_{2})=k_{1}a_{1}k_{1}^{-1}k_{2}a_{2}k_{2}^{-1}=ka_{1}a_{2}k^{-1}, \\
    \omega^{\pm} a_{1}\otimes a_{2}(\omega^{\pm})^{-1}=\omega^{\pm} a_{1}(\omega^{\pm})^{-1} \omega^{\pm} 
    a_{2}(\omega^{\pm})^{-1}=k_{1}^{\pm}a_{1}(k_{1}^{\pm})^{-1}k_{2}^{\pm}a_{2}(k_{2}^{\pm})^{-1}=k^{\pm}(a_{1}\otimes 
    a_{2})(k^{\pm})^{-1}. 
\end{gather*}
Thus  $\left(\cA, \cH,D_{\omega}\right)_{\sigma}$ is the pseudo-inner twisting by $\omega$ of the ordinary spectral triple $\left(\cA, \cH,D\right)$.

For $j=1,2$ let $e_{j}\in M_{q_{j}}(\cA_{j})$ where $e_{j}^{2}=e_{j}$. Then $e_{1}\otimes e_{2}$ is an idempotent in 
$M_{q_{1}q_{2}}(\cA)$. As $\left(\cA, \cH,D_{\omega}\right)_{\sigma}$ is the pseudo-inner 
twisting of $\left(\cA, \cH,D\right)$, using Lemma~\ref{lm:EqualityOfIndices}  we see that 
\begin{equation*}
 \acoup{[e_{1}]}{[e_{2}]}_{D_{\omega},\sigma}= \ind (D_{\omega})_{e_{1}\otimes e_{2},\sigma}=\ind D_{e_{1}\otimes 
   e_{2}}=\acoup{[e_{1}]}{[e_{2}]}_{D}.
\end{equation*}Thus the Poincar\'e duality pairings $(\cdot, \cdot)_{D_{\omega},\sigma}$ and $(\cdot,\cdot)_{D}$ agree. As the latter is nondegenerate, so is the former. The proof is complete.
\end{proof}

Let us now look at more specific examples of Poincar\'e dualities between twisted spectral triples. 
\renewcommand{\thesubsubsection}{\Alph{subsubsection}}

\subsection{Conformal Deformations of Spectral Triples} 
The case of conformal deformations of spectral triples is a special case of Proposition~\ref{Prop:TSTPD}. Thus let 
$(\cA_1, \cH, D)$ and $(\cA_2, \cH, D)$ be ordinary spectral triples that are in Poincar\'e duality. In addition, for 
$j=1,2$ let us pick a positive invertible element $k_{j}\in \cA_{j}$ and let $\sigma_{j}$ be the inner automorphism of $\cA_{j}$ 
given by $\sigma_{j}(a)=k_{j}^{2}ak_{j}^{-2}$, $a\in \cA_{j}$. Applying Proposition~\ref{Prop:TSTPD} to 
$\omega=k$, where $k=k_{1}k_{2}$,  we then arrive at the following result.

\begin{proposition}
    Under the above assumptions, $(\cA_1, \cH, kDk)_{\sigma_{1}}$ and $(\cA_2, \cH, 
    kDk)_{\sigma_{2}}$ are in Poincar\'e duality. 
\end{proposition}

If we specialize this to the case where $k_{2}=1$, then $\sigma_{2}$ is trivial, and so $(\cA_2, \cH, 
    kDk)$ is an ordinary spectral triple.  Therefore, we obtain the following corollary. 

\begin{corollary}\label{Cor:PDTSTk>0inA_1}
    Under the above assumptions, the twisted spectral $(\cA_1, \cH, k_{1}Dk_{1})_{\sigma_{1}}$ and the ordinary 
    spectral triple $(\cA_2, \cH, k_{1}Dk_{1})$ are in Poincar\'e duality.
\end{corollary}

\begin{remark}
    As we see in the above example, twisted spectral triples may naturally appear as 
    Poincar\'e duals of  some ordinary spectral triples.
    Another instances of such duality occur in Proposition~\ref{Prop:PDNCtorusTST} 
    and  Proposition~\ref{prop:PD.duals-discrete-groups-ordinary} below. 
\end{remark}

\subsection{Dirac spectral triples} Let $(M,g)$ be an even dimensional 
closed Riemannian spin manifold. As mentioned in Example~\ref{ex:DiracSpectralTriplePD} the associated Dirac spectral triple 
$(C^{\infty}(M), L^{2}(M, \sS), \sD)$ is in Poincar\'e duality with itself. Furthermore, it was mentioned in 
Example~\ref{ex:PseudoInnerTwistDiracSpectralTriple} that a pseudo-inner twisting is given by any smooth positive invertible even section $\omega$ of $\End \sS$. In 
this case the associated automorphisms $\sigma^{\pm}$ are trivial. Therefore, from Proposition~\ref{Prop:TSTPD} we immediately 
obtain the following statement. 

\begin{proposition}
 Let $\omega$ be a smooth positive invertible even section of $\End \sS$. Then  the Dirac spectral triple $(C^{\infty}(M),L^{2}(M,\sS),\omega\sD \omega)$ is 
 in Poincar\'e duality with itself.
\end{proposition}

\subsection{Noncommutative tori and conformal weights} Consider a conformal weight $\varphi(a)=\varphi_{0}(ak^{-2})$ on a noncommutative torus $\cA_{\theta}$, $\theta \in 
\R$, where $k$ is a positive element of $\cA_{\theta}$. Recall that the inverse of the unitary operator $W:\cH \rightarrow \cH_{\varphi}$ 
defined by~(\ref{eq:IsoHandH_phi}) implements a 
unitary equivalence between  $(\cA_{\theta}^{\opp}, \cH_{\varphi}, D_{\varphi})_{\sigma}$ and the pseudo-inner twisted spectral triple 
$(\cA_{\theta}^{\opp}, \cH, \omega D\omega)_{\sigma}$, where $\omega$ is given by~(\ref{eq:TST.Rk-inner-twisting}). 

As the left-regular action of $\cA_{\theta}$ commutes with the right-multiplication operator $R_{k}$, we 
see that $\cA_{\theta}$ commutes with the operators $\omega$ and $W$. 
It then follows that $\omega$ is a pseudo-inner twisting operator for $\cA_{\theta}$ with trivial associated 
automorphisms $\sigma^{\pm}$, so that 
$(\cA_{\theta},\cH, \omega D \omega)$ is an ordinary spectral triple. Furthermore, the unitary operator $W$ implements 
a unitary equivalence between the ordinary spectral triples
$(\cA_{\theta},\cH, \omega D \omega)$ and $(\cA_{\theta},\cH_{\varphi}, D_{\varphi})$. 

As mentioned in Example \ref{ex:NCTorusPD} $(\cA_{\theta}^{\opp},\cH, D)$ and $(\cA_{\theta},\cH, D)$ are in Poincar\'e duality, so it follows from 
Proposition~\ref{Prop:TSTPD} that $(\cA_{\theta}^{\opp},\cH, \omega D \omega)$  and $(\cA_{\theta},\cH, \omega D \omega)$  are in 
Poincar\'e duality. As Poincar\'e duality is preserved by unitary equivalence we arrive at the following statement. 

\begin{proposition}
\label{Prop:PDNCtorusTST}
Let $\varphi$ be a conformal weight on the noncommutative torus $\cA_{\theta}$, $\theta \in \R$. Then the associated twisted spectral triple 
$(\cA_{\theta}^{\opp}, \cH_{\varphi}, D_{\varphi})_{\sigma}$ is in Poincar\'e duality with 
the ordinary spectral triple $(\cA_{\theta}, \cH_{\varphi}, D_{\varphi})$.
\end{proposition}

\subsection{Duals of discrete cocompact subgroups of Lie groups} Let $\Gamma$ be a torsion-free discrete cocompact subgroup of 
a connected semisimple Lie group $G$. We shall keep using the 
notation of Example~\ref{ex:BaumConnesST}. Thus, the spectral triples 
$\left(\cA_\Gamma,L^2(X,\Lambda_\C^*T^*X),D_\tau\right)$ and 
$\left(C^{\infty}(\Gamma\backslash X),L^2(X,\Lambda_\C^*T^*X),D_\tau\right)$ are in Poincar\'e duality whenever 
$\Gamma$ satisfies the Baum-Connes conjecture. Recall that $\cA_{\Gamma}$ is the closure under the closure under holomorphic
functional calculus of the group algebra $\C\Gamma$. Therefore, if $h$ is any selfadjoint element of $\C \Gamma$, then 
$k=e^{h}$ is a positive invertible element of $\cA_{\Gamma}$. Combining this with Corollary~\ref{Cor:PDTSTk>0inA_1} we obtain the following result.  

\begin{proposition}\label{prop:PD.duals-discrete-groups-ordinary}
 Suppose that  $\Gamma$ satisfies the Baum-Connes conjecture. Let $h$ be a selfadjoint element of $\C  
 \Gamma$. Set $k=e^{h}\in \cA_{\Gamma}$ and let $\sigma$ be the inner automorphism of $\cA_{\Gamma}$ given by 
 $\sigma(a)=k^{2}ak^{-2}$, $a \in \Gamma$.  Then 
 the twisted spectral triple $\left(\cA_\Gamma,L^2(X,\Lambda_\C^*T^*X) ,kD_\tau k \right)_{\sigma}$ is in Poincar\'e duality with the ordinary spectral triple 
 $\left(C^{\infty}(\Gamma\backslash X), L^2(X,\Lambda_\C^*T^*X),kD_\tau k\right)$.
\end{proposition}

\section{Poincar\'e Duality and $\sigma$-Hermitian $\sigma$-Connections}
\label{sec:PDSigmaHermitianConnections}
In this section, we explain how to interpret Poincar\'e duality for twisted spectral triples in terms of 
$\sigma$-Hermitian $\sigma$-connections. To this end we need to discuss the tensor product of $\sigma$-connections. 

Let  $(\cA_1,\cH,D)_{\sigma_1}$ and $(\cA_2,\cH,D)_{\sigma_2}$ be twisted spectral 
triples satisfying the commutativity conditions~(\ref{PDTAlgebraCommute})--(\ref{PDTOrder1Condition}). In what follows we regard $\cA_{1}$ and $\cA_{2}$ 
as subalgebras of $\cA:=\cA_1\otimes\cA_2$, so that each automorphism 
$\sigma_{j}$, $j=1,2$, can be seen as the restriction to $\cA_{j}$ of $\sigma:=\sigma_{1}\otimes \sigma_{2}$. 
We observe that the properties~(\ref{PDTAlgebraCommute})--(\ref{PDTOrder1Condition}) imply that, for all $a_{j}\in \cA_{j}$ and $\omega_{j}\in 
 \Omega^{1}_{D,\sigma_{j}}(\cA_{j})$, $j=1,2$, we have
\begin{equation}
\label{eq:CommutingAlg&Forms}
    \omega_{1}a_{2}=\sigma_{2}(a_{2})\omega_{1} \qquad \text{and} \qquad \omega_{2}a_{1}=\sigma_{1}(a_{1})\omega_{2}. 
\end{equation}
In particular, we can regard  $ \Omega^{1}_{D,\sigma_{1}}(\cA_{1})$ and $ \Omega^{1}_{D,\sigma_{2}}(\cA_{2})$ as subspaces of 
$ \Omega^{1}_{D,\sigma}(\cA)$. 

Let $\cE$ be  a finitely generated projective right module over $\cA_{1}$ equipped with a 
$\sigma_{1}$-connection $\nabla^{\cE}$ and $\cF$ a finitely generated projective right module over $\cA_{2}$ equipped with a 
$\sigma_{2}$-connection $\nabla^{\cF}$. We observe that~(\ref{eq:CommutingAlg&Forms}) implies that, for all $a_{1}\in \cA_{1}$ and $a_{2}\in 
\cA_{2}$, we have
\begin{equation}
    c\left(\nabla^{\cE}\right)a_{2}=\sigma_{2}(a_{2})c\left(\nabla^{\cE}\right) \qquad \text{and} \qquad  
    c\left(\nabla^{\cF}\right)a_{1}=\sigma_{1}(a_{1})c\left(\nabla^{\cF}\right). 
    \label{eq:CommutingAlg&Connections}
\end{equation}

We define the tensor product $\sigma$-connection $\nabla^{\cE\otimes \cF}=\nabla^{\cE}\otimes \nabla^{\cF}$ as follows. 
There is a natural linear map $\iota_{1}:\cE^{\sigma_{1}}\otimes  \left( \cF^{\sigma_{2}}\otimes_{\cA_{2}} 
  \Omega^{1}_{D,\sigma_{2}}(\cA_{2})\right) \rightarrow \left(\cE \otimes 
  \cF\right)^{\sigma}\otimes_{\cA}\Omega^{1}_{D,\sigma}(\cA)$ defined by
\begin{equation*}
    \iota_{1}\left( \xi \otimes (\eta \otimes_{\cA_{2}}\psi)\right) = (\xi\otimes \eta)\otimes_{\cA}\psi \qquad \forall 
    (\xi,\eta,\psi)\in \cE^{\sigma_{1}}\times \cF^{\sigma_{2}}\times \Omega^{1}_{D,\sigma_{2}}(\cA_{2}). 
\end{equation*}
We also have a linear map  $\iota_{2}:\left(\cE^{\sigma_{1}}\otimes_{\cA_{1}} \Omega^{1}_{D,\sigma_{1}}(\cA_{1})\right)\otimes \cF^{\sigma_{2}} \rightarrow \left(\cE \otimes 
  \cF\right)^{\sigma}\otimes_{\cA}\Omega^{1}_{D,\sigma}(\cA)$ given by
  \begin{equation*}
    \iota_{2}\left( (\xi \otimes_{\cA_{1}} \omega)\otimes \eta \right) = (\xi\otimes \eta)\otimes_{\cA}\omega \qquad \forall 
    (\xi,\eta,\omega)\in \cE^{\sigma_{1}}\times \cF^{\sigma_{2}}\times \Omega^{1}_{D,\sigma_{1}}(\cA_{1}).   
  \end{equation*}
We define $\nabla^{\cE}\otimes \nabla^{\cF}$ as the $\C$-linear map from $\cE\otimes  \cF$  to $\left(\cE^{\sigma_{1}}\otimes 
  \cF^{\sigma_{2}}\right)\otimes_{\cA}\Omega^{1}_{D,\sigma}(\cA)= \left(\cE \otimes 
  \cF\right)^{\sigma}\otimes_{\cA}\Omega^{1}_{D,\sigma}(\cA)$ given by
  \begin{equation*}
      \nabla^{\cE}\otimes \nabla^{\cF}=\iota_{2}\circ \left( \nabla^{\cE}\otimes \sigma_{2}^{\cF}\right) + 
      \iota_{1}\circ \left(\sigma_{1}^{\cE}\otimes \nabla^{\cF}\right).  
  \end{equation*}
  Alternatively, let $\xi\in \cE$ and $\eta \in \cF$, and let us write 
  $\nabla^{\cE}\xi= \sum \xi_{\alpha}\otimes \omega_{\alpha}$ and $\nabla^{\cF}\eta = \sum \eta_{\beta}\otimes \psi_{\beta}$, where 
  $(\xi_{\alpha},\omega_{\alpha})\in \cE^{\sigma_{1}}\times  \Omega^{1}_{D,\sigma_{1}}(\cA_{1})$ and $(\eta_{\beta},\psi_{\beta})\in  
  \cF^{\sigma_{2}}\times  \Omega^{1}_{D,\sigma_{2}}(\cA_{2})$. Then we have
  \begin{equation}
  \label{eq:NablaEotimesF}
      \nabla^{\cE\otimes \cF}(\xi\otimes \eta)= \sum \left(\xi_{\alpha}\otimes \sigma_{2}^{\cF}(\eta)\right)\otimes_{\cA} \omega_{\alpha} + \sum 
      \left(\sigma_{1}^{\cE}(\xi)\otimes \eta_{\beta}\right)\otimes_{\cA} \psi_{\beta}.
%       \in \left(\cE^{\sigma_{1}}\otimes 
%    \cF^{\sigma_{2}}\right)\otimes_{\cA}
%        \Omega^{1}_{D,\sigma}(\cA).
  \end{equation}
  
\begin{lemma}
The map  $\nabla^{\cE\otimes \cF}$ is a $\sigma$-connection on $\cE\otimes \cF$. 
 \end{lemma}
 \begin{proof}
Let $\xi\in \cE^{\sigma_{1}}$ and $\eta \in \cF^{\sigma_{2}}$.  For $j=1,2$ let $a_{j}\in \cA_{1}$ and $\omega_{j} \in  \Omega^{1}_{D,\sigma_{j}}(\cA_{j})$. 
We note that in $ \left( \cE^{\sigma_{1}}\otimes \cF^{\sigma_{2}}\right)\otimes_{\cA} \Omega^{1}_{D,\sigma}(\cA)$ we have 
\begin{equation*}
 \left(\xi \sigma_{1}(a_{1})\right)\otimes \eta \otimes (\omega_{2}a_{2})   = \left( (\xi \otimes 
    \eta)(\sigma_{1}(a_{1})\otimes 1)\right) \otimes(\omega_{2} a_{2})
    = \xi \otimes \eta  \otimes_{\cA}\left( \sigma_{1}(a_{1})\omega_{2}a_{2}\right).   
\end{equation*}
Combining this with~(\ref{eq:CommutingAlg&Forms}) we then get
\begin{equation}
\label{eq:1}
 \left(\xi \sigma_{1}(a_{1})\right)\otimes \eta \otimes (\omega_{2}a_{2})   = \xi \otimes \eta \otimes (  \omega_{2} a_{1}a_{2}) 
     =\left(\xi \otimes \eta \otimes  \omega_{2} \right) (a_{1}\otimes a_{2}).   
\end{equation}
Similarly, we have
\begin{equation}
\label{eq:2}
   \xi \otimes \left( \eta \sigma_{2}(a_{2})\right) \otimes (\omega_{1}a_{1})   
     =\left(\xi \otimes \eta \otimes  \omega_{1} \right) (a_{1}\otimes a_{2}). 
\end{equation}
We also observe that~(\ref{eq:CommutingAlg&Forms}) and~(\ref{eq:SigmaDerivation}) imply that
\begin{equation*}
   \sigma_{2}(a_{2}) d_{\sigma_{1}}a_{1}+\sigma_{1}(a_{1})d_{\sigma_{2}}a_{2}=  
   (d_{\sigma}a_{1}) a_{2}+\sigma(a_{1})d_{\sigma}a_{2}=d_{\sigma}(a_{1}a_{2})=d_{\sigma}(a_{1}\otimes a_{2}).    
\end{equation*}

Let $\xi\in \cE$ and $\eta \in \cF$, and let us write  $\nabla^{\cE}\xi= \sum \xi_{\alpha}\otimes \omega_{\alpha}$ and $\nabla^{\cF}\eta = \sum \eta_{\beta}\otimes 
  \psi_{\beta}$ with $(\xi_{\alpha},\omega_{\alpha})$ and $(\eta_{\eta},\psi_{\beta})$ as above. For $a_{j}\in 
  \cA_{j}$, $j=1,2$, the $\sigma$-connection property~(\ref{eq:SigmaConnectionModuleMulti}) shows that
 $\nabla^{\cE}(\xi a_{1})= \sum \xi_{\alpha}\otimes (\omega_{\alpha}a_{1})  +\sigma_{1}^{\cE}(\xi)\otimes 
    d_{\sigma_{1}}a_{1}$ and $\nabla^{\cF}(\eta a_{2}) = \sum \eta_{\beta}\otimes (\psi_{\beta}a_{2})+ \sigma_{2}^{\cF}(\eta)\otimes 
      d_{\sigma_{2}}a_{2}$. Therefore~(\ref{eq:NablaEotimesF}) gives
      \begin{align*}
       \nabla^{\cE\otimes \cF}\left( (\xi a_{1})\otimes (\eta a_{2})\right) = &   \sum \xi_{\alpha} \otimes 
       \sigma_{2}^{\cF}(\eta) \sigma_{2}(a_{2})\otimes (\omega_{\alpha}a_{1})  + \sum  \sigma_{1}^{\cE}(\xi)\sigma_{1}(a_{1})\otimes \eta_{\beta}\otimes (\psi_{\beta}a_{2}) \\    
  & +  \sigma_{1}^{\cE}(\xi) \otimes 
       \sigma_{2}^{\cF}(\eta) \sigma_{2}(a_{2}) \otimes 
    d_{\sigma_{1}}(a_{1}) + 
  \sigma_{1}^{\cE}(\xi)\sigma_{1}(a_{1})\otimes \sigma_{2}^{\cF}(\eta)\otimes 
      d_{\sigma_{2}}(a_{2}).
      \end{align*}
Combining this with (\ref{eq:1})--(\ref{eq:2}) we deduce that
\begin{equation*}
  \nabla^{\cE\otimes \cF}\left( (\xi a_{1})\otimes (\eta a_{2})\right) =
  \left( \nabla^{\cE\otimes \cF} (\xi \otimes \eta)\right)  (a_{1}\otimes a_{2})  + \sigma^{\cE \otimes \cF}(\xi 
  \otimes \eta) \otimes d_{\sigma}(a_{1}\otimes a_{2}).
\end{equation*}
This proves the lemma. 
 \end{proof}

 Let us further assume that $\cE$ carries a $\sigma_{1}$-Hermitian  structure and $\cF$ carries a 
 $\sigma_{2}$-Hermitian structure. 
 Taking tensor products of the respective Hermitian metrics and $\fs$-maps of $\cE$ and $\cF$ defines a natural 
 $\sigma$-Hermitian structure on $\cE\otimes \cF$. We note that all pairings (\ref{eq:HermitianMetricInducedESigma}), 
 (\ref{eq:LSigmaPairing}) and~(\ref{eq:RSigmaPairing}) are the tensor 
 products of the corresponding pairings associated with $\cE$ and $\cF$. 
 
 \begin{lemma}
     Suppose that $\nabla^{\cE}$  is a $\sigma_{1}$-Hermitian connection and $\nabla^{\cF}$ is a $\sigma_{2}$-Hermitian 
     connection. Then $\nabla^{\cE\otimes \cF}$ is a $\sigma$-Hermitian $\sigma$-connection.
 \end{lemma}
 \begin{proof}
    For $j=1,2$ let $\xi_{j}\in \cE$ and $\eta_{j} \in \cF$. Let us write $\nabla^{\cE}\xi_{2}= \sum 
    \xi_{\alpha}\otimes \omega_{\alpha}$ with $\xi_{\alpha}\in \cE^{\sigma_{1}}$ and $\omega_{\alpha}\in 
     \Omega^{1}_{D,\sigma_{1}}(\cA_{1})$. Then~(\ref{eq:NablaEotimesF}) 
    gives
    \begin{equation*}
        \acoups{\xi_{1}\otimes \eta_{1}}{\nabla^{\cE\otimes \cF}(\xi_{2}\otimes \eta_{2})} =\sum 
        \acoupsa{\xi_{1}}{\xi_{\alpha}} \acoupsb{\eta_{1}}{\sigma_{2}^{\cF}(\eta_{2})} \omega_{\alpha} + 
        \acoupsa{\xi_{1}}{\sigma_{1}^{\cE}(\xi_{2})} \acoupsb{\eta_{1}}{\nabla^{\cF}\eta_{2}}.
    \end{equation*}
    Moreover~(\ref{eq:CommutingAlg&Forms}) and~(\ref{eq:RelationLRSigmaPairing}) imply that 
    \begin{equation}
    \label{eq:3}
    \begin{split}
     \sum  \acoupsa{\xi_{1}}{\xi_{\alpha}} \acoupsb{\eta_{1}}{\sigma_{2}^{\cF}(\eta_{2})} \omega_{\alpha} &= \sum 
        \acoupsa{\xi_{1}}{\xi_{\alpha}}  \omega_{\alpha}   
        \sigma_2^{-1}\left(\acoupsb{\eta_{1}}{\sigma_{2}^{\cF}(\eta_{2})}\right) \\ & =
        \acoupsa{\xi_{1}}{\nabla^{\cE}\xi_{2}} \sacoupb{\sigma_{2}^{\cF}(\eta_{1})}{\eta_{2}}.
     \end{split}
    \end{equation}
    
Let us write $\nabla^{\cF}\eta_{1} = \sum \eta_{\beta}\otimes 
  \psi_{\beta}$ with $\eta_{\beta}\in \cF^{\sigma_{2}}$ and $\psi_{\beta}\in  \Omega^{1}_{D,\sigma_{2}}(\cA_{2})$. Then 
    \begin{equation*}
        \sacoup{\nabla^{\cE\otimes \cF}(\xi_{1}\otimes \eta_{1})}{(\xi_{2}\otimes \eta_{2})} = 
        \sacoupa{\nabla^{\cE}\xi_{1}}{\xi_{2}} \sacoupb{\sigma_{2}^{\cF}(\eta_{1})}{\eta_{2}} + 
         \sum \psi_{\beta}  \left( \sacoupa{\sigma_{1}^{\cE}(\xi_{1})}{\xi_{2}}   \right)
         \sacoupb{\eta_{\beta}}{\eta_{2}}.
    \end{equation*}
Furthermore, as in~(\ref{eq:3}) we have 
\begin{align*}
     \sum \psi_{\beta}  \left( \sacoupa{\sigma_{1}^{\cE}(\xi_{1})}{\xi_{2}}   \right)
         \sacoupb{\eta_{\beta}}{\eta_{2}} & = 
        \sum    \sigma_{1} \left( \sacoupa{\sigma_{1}^{\cE}(\xi_{1})}{\xi_{2}}   \right)  \psi_{\beta} 
         \sacoupb{\eta_{\beta}}{\eta_{2}}   \\ 
         & = \acoupsa{\xi_{1}}{\sigma_{1}^{\cE}(\xi_{2})} \sacoupb{\nabla^{\cF}\eta_{1}}{\eta_{2}}.
\end{align*}    

It follows from all this that $ \acoups{\xi_{1}\otimes \eta_{1}}{\nabla^{\cE\otimes \cF}(\xi_{2}\otimes \eta_{2})} - 
\sacoup{\nabla^{\cE\otimes \cF}(\xi_{1}\otimes \eta_{1})}{(\xi_{2}\otimes \eta_{2})}$ is equal to
\begin{multline*}
    \left( \acoupsa{\xi_{1}}{\nabla^{\cE}\xi_{2}} -  \sacoupa{\nabla^{\cE}\xi_{1}}{\xi_{2}}  \right)  
   \left[ \sacoupb{\sigma_{2}^{\cF}(\eta_{1})}{\eta_{2}}\right] +\left[\acoupsa{\xi_{1}}{\sigma_{1}^{\cE}(\xi_{2})} \right] \left(    
    \acoupsb{\eta_{1}}{\nabla^{\cF}\eta_{2}} - 
    \sacoupb{\nabla^{\cF}\eta_{1}}{\eta_{2}}\right)  \\    
   = d_{\sigma_{1}}\left( \sacoupa{\sigma_{1}^{\cE}(\xi_{1})}{\xi_{2}}\right) 
  \left[ \sacoupb{\sigma_{2}^{\cF}(\eta_{1})}{\eta_{2}} \right] +  \sigma_{1} \left( \sacoupa{\sigma_{1}^{\cE}(\xi_{1})}{\xi_{2}}   
   \right)d_{\sigma_{2}}\left( \sacoupb{\sigma_{2}^{\cF}(\eta_{1})}{\eta_{2}}\right)  \\
    = d_{\sigma} \left( \sacoup{\sigma^{\cE\otimes \cF}(\xi_{1}\otimes \eta_{1})}{\xi_{2}\otimes \eta_{2}}\right),
\end{multline*}
The lemma is proved. 
 \end{proof}
  
In what follows we shall denote by $D_{\nabla^{\cE},\nabla^{\cF}}$ the operator $D_{\nabla^{\cE\otimes \cF}}$ 
associated with the tensor product $\sigma$-connection $\nabla^{\cE\otimes \cF}=\nabla^{\cE}\otimes \nabla^{\cF}$. By definition of the Poincar\'e duality 
pairing~(\ref{eq:BilinearPairing}) we have
\begin{equation*}
    \acoupd{[\cE]}{[\cF]}=\ind_{D,\sigma}[\cE\otimes \cF]. 
\end{equation*}
As $\nabla^{\cE\otimes \cF}$ is a $\sigma$-Hermitian $\sigma$-connection, using Proposition~\ref{prop:Ind_DSigmaHermitianConnection} we arrive at the following 
statement. 

\begin{proposition}
\label{prop:GeometricDescriptionPD} Let $\cE$ be a $\sigma_{1}$-Hermitian finitely generated projective right module 
over $\cA_{1}$ and let $\cF$ be a $\sigma_{2}$-Hermitian finitely generated projective right module 
over $\cA_{2}$. Then, for any $\sigma_{1}$-Hermitian $\sigma_{1}$-connection over $\cE$ and any  $\sigma_{2}$-Hermitian 
$\sigma_{2}$-connection over $\cF$, we have
     \begin{equation*}
    \acoupd{[\cE]}{[\cF]}=\dim \ker D_{\nabla^{\cE},\nabla^{\cF}}^{+}-\dim \ker D_{\nabla^{\cE},\nabla^{\cF}}^{-}.
\end{equation*}
\end{proposition}

In the same way that we deduced Corollary~\ref{cor:indexneq0nontrivialSolution} from Proposition~\ref{prop:Ind_DSigmaHermitianConnection}, we obtain the following corollary. 
\begin{corollary}
\label{eq:PairingNonzero=>NontrivialSolution}
Let $\cE$ be a $\sigma_{1}$-Hermitian finitely generated projective right module 
over $\cA_{1}$ and let $\cF$ be a $\sigma_{2}$-Hermitian finitely generated projective right module 
over $\cA_{2}$. 
If $\acoupd{[\cE]}{[\cF]}\neq 0$, then, for any $\sigma_{1}$-Hermitian $\sigma_{1}$-connection $\nabla^{\cE}$ over $\cE$ and any  $\sigma_{2}$-Hermitian 
$\sigma_{2}$-connection $\nabla^{\cF}$ over $\cF$, the equation $D_{\nabla^{\cE},\nabla^{\cF}}u=0$ has nontrivial solutions. 
\end{corollary}

\section{Inequalities of $\fs$-Eigenvalues. Main Results}
 \label{sec:VWInequalities}
In this section, we prove versions of the inequality of Vafa-Witten~\cite{VW:CMP84} for twisted spectral triples. 
In what follows we shall say that a (twisted or ordinary) spectral triple over an algebra $\cA$ has \emph{finite 
topological type} when $\dim K_{0}(\cA)\otimes \Q<\infty$. 

\begin{theorem}\label{thm:EigenvalueInequality}
    Let $(\cA_1,\cH,D)_{\sigma_1}$ be a twisted spectral triple such that
    \begin{enumerate}
        \item[(i)] $(\cA_1,\cH,D)_{\sigma_1}$ has finite topological type.   
    
        \item[(ii)] The automorphism $\sigma_{1}$ is inner in the sense that $\sigma_{1}(a)=kak^{-1}$ for some positive invertible element 
        $k \in \cA_{1}$.
    
        \item[(iii)] $(\cA_1,\cH,D)_{\sigma_1}$ is in Poincar\'e duality with a twisted spectral triple 
        $(\cA_2,\cH,D)_{\sigma_2}$, where $\sigma_{2}$ is a ribbon automorphism in the sense of~(\ref{eq:Index.square-root-sigma}). 
    \end{enumerate}
    Then there is a constant $C>0$ 
    such that, for any Hermitian finitely generated  projective module $\cE$ over $\cA_1$ equipped with a 
    $\sigma_1$-Hermitian $\sigma_1$-connection $\nabla^{\cE}$, we have
     \begin{equation}
     \label{eq:VWboundMainthm}
         \left| \lambda_{1}( D_{\nabla^{\cE}}) \right| \leq C\|k^{-1}\|
     \end{equation}where $ \lambda_{1}( D_{\nabla^{\cE}})$ is the $\fs$-eigenvalue of $D_{\nabla^{\cE}}$ having the smallest absolute value. 
 \end{theorem} 
 \begin{proof}
Let  $\cE$ be a Hermitian finitely generated  projective module  over $\cA_1$ equipped with a $\sigma_1$-Hermitian $\sigma_1$-connection $\nabla^{\cE}$. We note that, as $\sigma_{1}$ is inner, by 
Lemma~\ref{lem:innerSigmaHermitian} the Hermitian metric of $\cE$ canonically defines a $\sigma_{1}$-Hermitian structure with $\fs$-map given 
by~(\ref{eq:SigmaHermitianInner}). In addition, let $\cF$ be a $\sigma_{2}$-Hermitian finitely generated projective module over $\cA_2$ 
  equipped with a $\sigma_2$-connection $\nabla^{\cF}$. We endow $\cE\otimes \cF$ with the tensor product 
  $\sigma$-Hermitian structure and the tensor product $\sigma$-Hermitian $\sigma$-connection $\nabla^{\cE\otimes \cF}=\nabla^{\cE}\otimes\nabla^{\cF}$ defined 
  by~(\ref{eq:NablaEotimesF}). 
  
  We observe that if $\nabla_{1}^{\cF}$ and $\nabla_{2}^{\cF}$ are two $\sigma_{2}$-connections on $\cF$ and we 
  denote by $\nabla_{1}^{\cE\otimes\cF}$ and $\nabla_{2}^{\cE\otimes\cF}$ the respective $\sigma$-connections on 
  $\cE\otimes \cF$, then 
  \begin{equation*}
      \nabla_{1}^{\cE\otimes\cF}-\nabla_{2}^{\cE\otimes\cF}=\sigma_{1}^{\cE}\otimes \left(\nabla_{1}^{\cF}-\nabla_{2}^{\cF}\right).
  \end{equation*}Therefore, using~(\ref{eq:Index.Dnabla}) we see that
  \begin{equation}
  \label{eq:DifferenceD_NablaEF}
      \left(D_{\nabla^{\cE},\nabla_{1}^{\cF}}-D_{\nabla^{\cE},\nabla_{2}^{\cF}}\right)=
      \sigma^{\cE}_{1}(\xi)\otimes \left( 
      c\left(\nabla_{1}^{\cF}\right)-c\left(\nabla_{2}^{\cF}\right)\right).
  \end{equation}

Let $\cF'$ be a finitely generated projective right module over $\cA_{2}$ such that $\cF\oplus \cF'$ is a free module. 
As $\sigma$ is ribbon, Lemma~\ref{lem:SqrtSigmaSigmaHermitian} ensures us that $\cF'$ carries a $\sigma_{2}$-Hermitian 
structure. Let $\nabla^{\cF'}$ be a $\sigma_{2}$-Hermitian $\sigma_{2}$-connection on $\cF'$. We endow $\cF\oplus 
\cF'$ with the $\sigma_{2}$-Hermitian structure given by the direct sum of the $\sigma_{2}$-Hermitian structures of 
$\cF$ and $\cF'$. Then $\nabla^{\cF}\oplus \nabla^{\cF'}$ is a $\sigma_{2}$-Hermitian $\sigma_{2}$-connection on 
$\cF\oplus \cF'$. 
  
Let $q$ be the  rank of $\cF\oplus \cF'$ and denote by $\cF_{0}$ the free module $\cA^{q}_{2}$. We endow $\cF_{0}$ with the 
$\sigma_{2}$-Hermitian structure defined by the canonical Hermitian metric of $\cA^{q}$ and the identity map 
from $\cF_{0}^{\sigma_{2}}=\cF_{0}$ to $\cF_{0}$ as in~Example~\ref{ex:SigmaHermitianStrutureOnFreeModule}. 
Then by Lemma~\ref{lem:GrassmannianSigmaHermitianConnection} the trivial Grassmannian 
$\sigma_{2}$-connection $\nabla_{0}=d_{\sigma_{2}}$ is a $\sigma_{2}$-Hermitian $\sigma_2$-connection on $\cF_{0}$. 
Let $\phi:\cF_{0}\rightarrow \cF\oplus \cF'$ be a right-module isomorphism. Consider the right-module isomorphism 
$\phi_{\sigma_{2}}:\cF_{0}\rightarrow 
\cF^{\sigma_{2}}\oplus \left( \cF'\right)^{\sigma_{2}}$ defined by
\begin{equation}
\label{eq:phiSigma}
    \phi_{\sigma_{2}}:=\left(\sigma^{\cF}\oplus \sigma^{\cF'}\right)^{-1}\circ \phi \circ \sigma^{-1},
\end{equation}
where $\sigma$ is the canonical lift of $\sigma$ to $\cF_{0}=\cA_{2}^{q}$. Using $\phi$ and 
$\phi_{\sigma_{2}}$ we pullback $\nabla^{\cF}\oplus \nabla^{\cF'}$ to the $\sigma_{2}$-connection on $\cF_{0}$ given by
\begin{equation*}
   \nabla_{1}:= \left( \phi_{\sigma_{2}}^{-1}\otimes 
    1\right)\circ \left( \nabla^{\cF}\oplus \nabla^{\cF'}\right)\circ \phi.
\end{equation*}

We point out that $\nabla_{1}$ need not to be $\sigma_{2}$-Hermitian. Nevertheless, by Lemma~\ref{lem:DifferenceSigmaConnection}, the $\sigma_{2}$-connections 
$\nabla_{0}$ and $\nabla_{1}$ differ by an element of $\Hom_{\cA_{2}}\left(\cF_{0},\cF_{0}\otimes 
 \Omega^{1}_{D,\sigma_{2}}(\cA_{2})\right)$. Set
\begin{equation}
\label{eq:T_F}
    T_{\cF}:=c\left( \nabla_{0}\right)-c\left(\nabla_{1}\right)\in 
\cL\left(\cH(\cF_{0})\right).
\end{equation}
Then~(\ref{eq:DifferenceD_NablaEF})  shows that
\begin{equation*}
   D_{\nabla^{\cE},\nabla_{0}}- D_{\nabla^{\cE},\nabla_{1}}=\sigma^{\cE}_{1}\otimes T_{\cF}.
\end{equation*}
Using the max-min principle~(\ref{eq:VF.min-maxT}) we then obtain
\begin{equation}
\label{eq:LambdaDNabla0Estimate}
   \left| \lambda_{1}\left( D_{\nabla^{\cE},\nabla_{0}}\right)\right| = 
   \mu_{1}\left( D_{\nabla^{\cE},\nabla_{0}}\right) \leq \left|  \mu_{1}\left( D_{\nabla^{\cE},\nabla_{1}}\right)\right|  + 
    \|\sigma^{\cE}_{1}\otimes T_{\cF}\| .
\end{equation}

\begin{claim}
\label{claim:AddNabla_0}
It holds that $ \lambda_{1}\left( D_{\nabla^{\cE},\nabla_{0}}\right)=\lambda_{1}\left( D_{\nabla^{\cE}}\right)$.
\end{claim}
\begin{proof}[Proof of Claim~\ref{claim:AddNabla_0}]
 As $\cF_{0}=\cA_{2}^{q}$ and the algebras $\cA_{1}$ and $\cA_{2}$ commute with each other, there is a canonical isomorphism 
$U:\left(\cE\otimes\cF_{0}\right)\otimes_{\cA}\cH \rightarrow \left(\cE{\otimes_{\cA_{1}}}\cH \right)^{q}$ such that,  
for all $(\xi,\zeta)\in \cE\times \cH$ and $\eta=(\eta_{j})\in \cA^{q}_{2}$, 
\begin{equation}
\label{eq:UnitaryH(E)qH(EF)}
    U(\xi \otimes \eta \otimes \zeta)= \left( \xi\otimes \left(\eta_{1}\zeta  \right),\ldots, \xi\otimes 
    \left(\eta_{q}\zeta  \right) \right).
\end{equation}
This gives rise to an isometric isomorphism from $\cH(\cE\otimes \cF_{0})$ onto $\cH(\cE)^{q}$ with inverse, 
\begin{equation*}
    U^{-1}\left( \xi_{1}\otimes \zeta_{1}, \ldots, \xi_{q}\otimes \zeta_{q} \right) = \xi_{1}\otimes \varepsilon_{1} 
    \otimes \zeta_{1}+ \cdots + \xi_{q}\otimes \varepsilon_{q} \otimes \zeta_{q}, \qquad  \xi_{j} \in \cE, \ \zeta_{j}\in \cH,
\end{equation*}where $\varepsilon_{1},\ldots,\varepsilon_{q}$ is the canonical basis of $\cA^{q}_{2}$. 

Let us denote by $U^{\sigma_{1}}$ the isomorphism~(\ref{eq:UnitaryH(E)qH(EF)}) corresponding to $\cE^{\sigma_{1}}$. We 
observe that if we set $S=\fs\otimes1_{\cF_0}$, then it follows from~(\ref{eq:HermitianInnerProductH(E)}) that $S\otimes1_{\cH}$ is an isometric isomorphism from 
$\cH(\cE^{\sigma_1}\otimes\cF_0)$ onto 
$\cH(\cE\otimes{\cF_0})$ and $U^{\sigma_{1}}$ agrees with $(S\otimes1_{\cF_0})^*U(S\otimes1_{\cF_0})$.   

Set  $\nabla^{\cE\otimes\cF_{0}}=\nabla^{\cE}\otimes\nabla^{\cF_{0}}$. Let $(\xi,\zeta) \in \cE\times \cH $ and 
$\eta=(\eta_{j})\in \cF_{0}$, and let us write $\nabla^{\cE}\xi= \sum 
    \xi_{\alpha}\otimes \omega_{\alpha}$ with $\xi_{\alpha}\in \cE^{\sigma_{1}}$ and $\omega_{\alpha}\in 
     \Omega^{1}_{D,\sigma_{1}}(\cA_{1})$. As $\nabla_{0}\eta=d_{\sigma_{2}}\eta$, we have
    \begin{equation*}
        \nabla^{\cE\otimes\cF_{0}}(\xi\otimes \eta)=\sum \xi_{\alpha}\otimes \sigma_{2}(\eta)\otimes \omega_{\alpha} 
        +\sigma_{1}^{\cE}(\xi)\otimes d_{\sigma_{2}}\eta.
    \end{equation*}Using (\ref{eq:Index.Dnabla}) we then get
    \begin{equation*}
       D_{\nabla^{\cE},\nabla_{0}}(\xi\otimes \eta \otimes \zeta)= \sigma^{\cE}_{1}(\xi)\otimes \sigma_{2}(\eta)\otimes 
       D\zeta + \sum \xi_{\alpha}\otimes \sigma_{2}(\eta)\otimes \omega_{\alpha} (\zeta)
        +\sigma_{1}^{\cE}(\xi)\otimes\left( \left( d_{\sigma_{2}}\eta\right)\zeta\right).
    \end{equation*} 
Combining this with~(\ref{eq:UnitaryH(E)qH(EF)}) we see that, for $j=0, \ldots, q$, we have 
\begin{align*}
    U^{\sigma_{1}}D_{\nabla^{\cE},\nabla_{0}}(\xi\otimes \eta \otimes \zeta)_{j} & = \sigma^{\cE}_{1}(\xi)\otimes\left( 
    \left(\sigma_{2}(\eta_{j})D+d_{\sigma_{2}}\eta_{j}\right)\zeta\right)  + \sum \xi_{\alpha}\otimes \left(\sigma_{2}(\eta_{j}) 
    \omega_{\alpha} (\zeta)\right) \\
    & = \sigma^{\cE}_{1}(\xi)\otimes D(\eta_{j}\zeta)+\sum \xi_{\alpha}\otimes \omega_{\alpha} (\eta_{j}\zeta) \\
    & = D_{\nabla^{\cE}}\left( U(\xi\otimes \eta\otimes \zeta)_{j}\right).
\end{align*}
As $U^{\sigma_{1}}= (\fs\otimes 1_{\cF_0}\otimes 1_{\cH})^{-1}U(\fs\otimes 1_{\cF_0}\otimes 1_{\cH})$ we deduce that
\begin{equation*}
    U(S\otimes1_{\cF_0})D_{\nabla^{\cE},\nabla_{0}}U^{-1}= \overbrace{SD_{\nabla^{\cE}} \oplus \cdots \oplus SD_{\nabla^{\cE}}}^{\text{$q$ copies}}.
\end{equation*}
Thus the spectrum of $(\fs\otimes 1_{\cF_0}\otimes1_{\cH})D_{\nabla^{\cE},\nabla_{0}}$ is $q$ copies of that
of $(\fs \otimes 1_{\cH})D_{\nabla^{\cE}}$. That is, the $\fs\otimes 1_{\cF_{0}}$-eigenvalue set of 
$D_{\nabla^{\cE},\nabla_{0}}$ is $q$-copies of the $\fs$-eigenvalue set of $D_{\nabla^{\cE}}$. Hence the claim.
\end{proof}

\begin{claim}
\label{Claim:TensorEsti}
It holds that   $\|\sigma^{\cE}_{1}\otimes T_{\cF}\|\leq \|k^{-1}\| \|T_{\cF}\|$.
\end{claim}
\begin{proof}[Proof of Claim~\ref{Claim:TensorEsti}]
Let $\hat{\cH}=\cH(\cF_{0})\simeq\cH^{q}$, which we equip with an orthonormal basis $\{\zeta_{\alpha}\}$. The algebra $\cA_{1}$ is represented in $\hat{\cH}$ by $a_{1}\rightarrow 
1_{\cF^{0}}\otimes a_{1}$. This is a unitary representation since $\cA_{1}$ commutes with $\cA_{2}$. We also note that $T_{\cF}$ is a bounded operator of $\hat{\cH}$. 
Moreover, as $\fs\otimes 1_{\hat{\cH}}$ is a unitary operator of 
$\cH(\cE\otimes \cF_{0})=\hat{\cH}(\cE)$, using~(\ref{eq:SigmaHermitianInner}) we have
\begin{equation}
\label{eq:SigmaT=1HatT}
    \|\sigma^{\cE}\otimes T_{\cF}\|=\| \fs\circ \sigma^{\cE}\otimes T_{\cF}\|=\|k^{-1}\otimes 
    T_{\cF}\|=\|1_{\cE}\otimes\hat{T}\|, 
\end{equation}
where we have set $\hat{T}:=k^{-1}T_{\cF}$. 

By the construction of the $\sigma$-translation in Section~\ref{sec:IndexMapSigmaConnections}, there is a free right module 
$\cE_{0}\simeq \cA_{1}^{q'}$ such that $\cE$ and $\cE^{\sigma}$ are direct summands of $\cE_{0}$. Moreover, we can 
choose $\sigma^{\cE_{0}}$ so that it agrees with $\sigma^{\cE}$ on $\cE$. We further equip $\cE_{0}$ with a 
$\sigma_{1}$-Hermitian structure $\left\{\acoup{\cdot}{\cdot},\fs^{\cE_{0}}\right\}$ such that $\fs^{\cE_{0}}$ is given 
by~(\ref{eq:CanonicalSigmaHermitStrucFreeModule}) and the Hermitian metric $\acoup{\cdot}{\cdot}$ agrees with that of $\cE$ on $\cE$. This implies that the 
inclusion of $\cH(\cE\otimes \cF_{0})=\hat{\cH}(\cE)$ into $\hat{\cH}(\cE_{0})$ is isometric.

Let $a_{1}\in \cA_{1}$. It follows from~(\ref{eq:CommutingAlg&Connections}) that 
$c\left(\nabla_{i}\right)a_{1}=\sigma_{1}(a_{1})c\left(\nabla_{i}\right)$, $i=0,1$. As 
$T_{\cF}=c(\nabla_{0})-c(\nabla_{1})$, we see that
\begin{equation*}
    \hat{T}a_{1}=k^{-1}T_{\cF}a_{1}=k^{-1}\sigma_{1}(a_{1})T_{\cF}=a_{1}k^{-1}T_{\cF}=a_{1}\hat{T}.
\end{equation*}
Thus $\hat{T}$ commutes with the algebra $\cA_{1}$. Therefore, for all $\xi \in \cE_{0}$ and $\zeta\in 
\hat{\cH}$, 
\begin{equation}
\label{eq:1otimesTandT}
\begin{split}
 \| (1_{\cE_{0}}\otimes \hat{T})(\xi\otimes \zeta)\|^{2}=\acou{\hat{T}\zeta }{\hat{T}\acoup{\xi}{\xi}\zeta}= 
 \acou{\hat{T}\acoup{\xi}{\xi}^{\frac{1}{2}}\zeta}{\hat{T}\acoup{\xi}{\xi}^{\frac{1}{2}}\zeta} & \leq \|\hat{T}\|^{2} 
 \left\|\acoup{\xi}{\xi}^{\frac{1}{2}}\zeta\right\|^{2}\\ & \leq \|\hat{T}\|^{2}\|\xi\otimes \zeta\|^{2}.
\end{split}
\end{equation}
The Gram-Schmidt process enables us to produce out of any given $\cA_{1}$-basis of $\cE_{0}$ an orthogonal basis. In fact, 
as $\cA_{1}$ is closed under holomorphic functional calculus we can push through the process to get an orthonormal basis
$\varepsilon_{1},\ldots, \varepsilon_{q'}$ (where  $q'$ is the rank of $\cE_{0}$). Alternatively, the Hermitian structure of $\cE_{0}$ is isomorphic to the canonical Hermitian structure of 
$\cA_{1}^{q'}$. In any case $\{\varepsilon_{i}\otimes \zeta_{\alpha}\}$ is an orthonormal basis of $\hat{\cH}(\cE_{0})$. Therefore, using~(\ref{eq:1otimesTandT}), we get 
\begin{equation*}
    \|1_{\cE_{0}}\otimes \hat{T}\|\leq \op{max}_{i,\alpha}  \| (1_{\cE_{0}}\otimes \hat{T})(\varepsilon_{i}\otimes 
\zeta_{\alpha})\|\leq \|\hat{T}\|=\|k^{-1}T_{\cF}\|\leq \|k^{-1}\|\|T_{\cF}\|.
\end{equation*}
Combining this with~(\ref{eq:SigmaT=1HatT}) and the isometricness of the inclusion of $\cH(\cE)$ into $\cH(\cE_{0})$, we get
\begin{equation*}
    \|\sigma^{\cE}\otimes T_{\cF}\|=\|1_{\cE}\otimes \hat{T}\|\leq \|1_{\cE_{0}}\otimes \hat{T}\|\leq 
    \|k^{-1}\|\|T_{\cF}\|. 
\end{equation*}The claim is proved. 
\end{proof}

Combining~(\ref{eq:LambdaDNabla0Estimate}) with Claim~\ref{claim:AddNabla_0} and Claim~\ref{Claim:TensorEsti} we obtain 
\begin{equation}
\label{eq:LambdaDNablaEInequality}
  \left|\lambda_{1}\left( D_{\nabla^{\cE}}\right)\right|   \leq  \left|\mu_{1}\left( 
  D_{\nabla^{\cE},\nabla_{1}}\right)\right|  + \|k^{-1}\| \|T_{\cF}\|.
\end{equation}
Observe that $\nabla^{\cE}\otimes \nabla_{1}= \left( \sigma_{1}^{\cE}\otimes \left(\phi^{\sigma_{2}}\right)^{-1}\otimes 
    1_{ \Omega^{1}_{D,\sigma}(\cA)}\right)\circ \left(\nabla^{\cE\otimes\cF}\oplus \nabla^{\cE\otimes\cF'}\right)\circ(1_{\cE}\otimes 
    \phi)$. Therefore, 
\begin{equation*}
    D_{\nabla^{\cE},\nabla_{1}}=  \left( \sigma_{1}^{\cE}\otimes \left(\phi^{\sigma_{2}}\right)^{-1}\otimes 
    1_{\cH}\right)\circ \left(D_{\nabla^{\cE},\nabla^{\cF}}\oplus D_{\nabla^{\cE'},\nabla^{\cF'}}\right)\circ(1_{\cE}\otimes 
    \phi \otimes 1_{\cH}).
\end{equation*}It then follows that
\begin{equation}
\label{eq:SumOfKer}
    \ker D_{\nabla^{\cE},\nabla_{1}} \simeq \ker D_{\nabla^{\cE},\nabla^{\cF}}\oplus \ker D_{\nabla^{\cE'},\nabla^{\cF'}}. 
\end{equation}
Suppose that $\ker D_{\nabla^{\cE},\nabla^{\cF}}$ is nontrivial. Using~(\ref{eq:SumOfKer}) we see that $\ker D_{\nabla^{\cE},\nabla_{1}}$ too is 
nontrivial. It then follows from the max-min principle~(\ref{eq:VF.min-maxT}) (or from the fact that $D_{\nabla^{\cE},\nabla_{1}}$ and 
$\left| D_{\nabla^{\cE},\nabla_{1}}\right|$ have same kernel) that $\mu_{1}\left( 
D_{\nabla^{\cE},\nabla_{1}}\right)=0$. Combining this with~(\ref{eq:LambdaDNablaEInequality}) we then deduce that
\begin{equation}
\label{eq:KerNon0=>Inequality}
 \ker D_{\nabla^{\cE},\nabla^{\cF}} \neq \{0\} \ \Longrightarrow \     \left|\lambda_{1}\left( 
 D_{\nabla^{\cE}}\right)\right|   \leq   \|k^{-1}\| \|T_{\cF}\|.
\end{equation}
Therefore, to complete the proof it is enough show that the condition  $ \ker D_{\nabla^{\cE},\nabla^{\cF}} \neq \{0\}$ can always 
be realized by taking $(\cF,\nabla^{\cF})$ among a fixed given \emph{finite} family of such pairs.

By assumption $K_{0}(\cA_{1})\otimes \Q$ has finite dimension and is in duality with $K_{0}(\cA_{2})\otimes \Q$, so
$K_{0}(\cA_{2})\otimes \Q$ too has finite dimension. Therefore, $K_{0}(\cA_{2})\otimes \Q$ has a finite spanning set 
$\cB=\{\beta_{0},\ldots, \beta_{N}\}$ consisting of equivalence classes of finitely generated projective modules over 
$\cA_{2}$. In fact, we can take $\beta_{0}$ to be the class of the rank 1 free module $\cA_{2}$ and choose 
$\beta_{1},\ldots,\beta_{N}$ in such a way that $\beta_{1}-m_{1}\beta_{0},\ldots,\beta_{N}-m_{N}\beta_{0}$ form a basis of 
$K_{0}(\cA_{2})\otimes \Q$ for some integers $m_{1},\ldots, m_{N}$. 

As $\sigma_{2}$ is ribbon, by Lemma~\ref{lem:GrassmannianSigmaHermitianConnection} we may represent each class $\beta_{j}$, $j=0,\ldots,N$,  by a 
$\sigma_{2}$-Hermitian  finitely generated projective module $\cF_{j}$ over $\cA_{2}$ equipped with a 
$\sigma_{2}$-Hermitian $\sigma_{2}$-connection. We then set
\begin{equation*}
    C:= \op{max}\left\{ \left\|T_{\cF_{0}}\right\|,\ldots,  \left\|T_{\cF_{N}}\right\|\right\}. 
\end{equation*}
The nondegeneracy of the pairing $\acoupd{\cdot}{\cdot}$ and the fact that $\cB$ is a spanning set imply the existence of some $j\in \{0,\ldots,N\}$ such that  
\begin{equation*}
 0\neq \acoupd{[\cE]}{\beta_{j}}=  \acoupd{[\cE]}{[\cF_{j}]}.   
\end{equation*}Thanks to Corollary~\ref{eq:PairingNonzero=>NontrivialSolution} this implies that the kernel of $D_{\nabla^{\cE},\nabla^{\cF_{j}}}$ is 
nontrivial. Therefore, using~(\ref{eq:KerNon0=>Inequality}) we obtain
\begin{equation*}
    \left|\lambda_{1}\left( D_{\nabla^{\cE}}\right)\right|   \leq   \|k^{-1}\| \|T_{\cF_{j}}\| \leq C \|k^{-1}\|. 
\end{equation*}
As $C$ is independent of the pair $(\cE, \nabla^{\cE})$, the result is proved.
\end{proof}

\begin{remark}
 The assumption that $\sigma_1$ is inner is used in the proof of Claim~\ref{Claim:TensorEsti} above.  It 
 is also a crucial assumption to have a bound in~(\ref{eq:SigmaT=1HatT}) that is independent of $\cE$. Indeed, when $\sigma_{1}$ is not inner we 
 also can get a bound for $|\lambda_{1}(D_{\nabla^{\cE}})|$, but we need to replace 
    $\|k^{-1}\|$ by some function of $\fs\circ \sigma_1^{\cE}$. Thus in this case we obtain a bound which \emph{a priori} 
    depends on $\cE$. Note that, in all the examples of twisted spectral triples in this paper, the automorphism 
    $\sigma_1$ is always inner. 
\end{remark}

The Vafa-Witten bound $C$ in~(\ref{eq:VWboundMainthm}) depends on the Poincar\'e dual of $(\cA_1, \cH, D)_{\sigma_1}$, namely, the twisted spectral triple 
$\left( \cA_{2},\cH,D\right)_{\sigma_{2}}$. To better understand its dependence we look at the special case of 
pseudo-inner twistings of ordinary Poincar\'e dual pairs. 

Let $(\cA_{1}, \cH, D)$ be an ordinary spectral triple which has finite topological type and is in Poincar\'e duality with 
an ordinary spectral triple $(\cA_{2}, \cH, D)$.  Let $\omega= 
\begin{pmatrix}
    \omega^{+} & 0 \\
    0 & \omega^{-}
\end{pmatrix}\in \cL(\cH)$ be a pseudo-inner twisting operator. Thus, there are positive invertible elements 
$k_{j}^{\pm}\in \cA_{j}$, $j=1,2$, such that $k_{j}^{+}k_{j}^{-}= k_{j}^{-}k_{j}^{+}$ and $ \omega^{\pm} a 
(\omega^{\pm})^{-1}=\sigma_{j}^{\pm}(a)$ for all $a\in \cA_{j}$, where 
$\sigma^{\pm}_{j}(a):=k_{j}^{\pm}a \left(k^{\pm}_{j}\right)^{-1}$. 

Denote by $\sigma_{j}$ the automorphism of $\cA_{j}$ given by $\sigma_{j}(a)=k_{j}ak_{j}^{-1}$, $a \in \cA_{j}$, where $k_{j}=k_{j}^{+}k_{j}^{-}$. 
Setting $D_{\omega}=\omega D \omega$, we know from Proposition~\ref{Prop:ConformalPerturbationw} and Proposition~\ref{Prop:TSTPD} 
that $(\cA_{1},\cH,D_{\omega})_{\sigma_{1}}$ and $(\cA_{2},\cH,D_{\omega})_{\sigma_{2}}$ are twisted spectral triples 
which are in Poincar\'e duality. 

\begin{theorem}\label{eq:ConformalVersionInequality}
Under the above assumptions, there is a constant $C>0$, which is independent of $\omega$ and $k_{j}^{\pm}$, such 
that, for any Hermitian finitely generated projective module $\cE$ over $\cA_1$ and any $\sigma_{1}$-Hermitian $\sigma_{1}$-connection $\nabla^{\cE}$ on $\cE$, 
we have 
\begin{equation*}
\left|\lambda_1(D_{\omega, \nabla^{\cE}})\right|\leq  C \|k_{1}^{-1}\| 
\|\omega^{+}\| \|\omega^{-}\| \|h_{2}\|\|h_{2}^{-1}\|\left( 1+\|h_{2}^{-1}[D^{+},h_{2}]\|\right),
\end{equation*}
where $ \lambda_{1}( D_{\omega, \nabla^{\cE}})$ is the first $\fs$-eigenvalue of $D_{\omega, \nabla^{\cE}}$ and we 
have set $h_{2}=\left( k_{2}^{+}\right)^{\frac{1}{2}} \left(k_{2}^{-}\right)^{-\frac{1}{2}}$.
\end{theorem}
\begin{proof}
Let $f\in M_{q}(\cA_{2})$ be an idempotent and denote by $\cF=f\cA_{2}^{q}$ the associated finitely projective module. 
Set $\cF'=(1-f)\cA^{q}$. We equip $\cF$ and $\cF'$ with their respective $\sigma_{2}$-Hermitian structures given by the 
Hermitian metrics induced by the canonical Hermitian metric of $\cF_{0}=\cA_{2}^{q}$ and the left-multiplication by 
$k_{2}^{-1}$ (See Example~\ref{ex:fs=k^-1} and Lemma~\ref{lem:SqrtSigmaSigmaHermitian}). 
Their respective $\sigma_{2}$-Grassmannian connections $\nabla_{0}^{\cF}$ and 
  $\nabla_{0}^{\cF'}$ then are $\sigma_{2}$-Hermitian $\sigma_{2}$-connections by Lemma~\ref{lem:GrassmannianSigmaHermitianConnection}. 
  There is a canonical right-module 
  isomorphism $\phi:\cF_{0}\rightarrow \cF\oplus \cF'$ given by
  \begin{equation*}
      \phi(\xi)=\left(f\xi,(1-f)\xi\right) \qquad \forall \xi \in \cF_{0}.
  \end{equation*}Its inverse is given by 
  \begin{equation*}
      \phi^{-1}(\xi,\xi')=\xi+\xi' \qquad \text{for all $\xi\in \cF$ and $\xi'\in \cF'$}.
  \end{equation*}
  The corresponding right-module isomorphism $\phi^{\sigma_{2}}$  in~(\ref{eq:phiSigma}) is given by the same formula upon 
  replacing $f$ in $\phi$ by $\sigma(f)$. We then denote by $T_{\cF}$ the bounded operator of $\cL(\cH(\cF_{0}))$ defined 
  by~(\ref{eq:T_F}) with 
  $\nabla_{1}=\left(\left(\phi^{\sigma_{2}}\right)^{-1}\otimes1\right)\circ\left(\nabla^{\cF}_{0}\oplus \nabla^{\cF'}_{0}
 \right)\circ \phi$.
 
 Let $\xi\in \cA_{2}^{q}$. We note that by~(\ref{eq:GrassmannianSigmaConnection}) $\nabla^{\cF}_{0}=\sigma_{2}(f)\nabla_{0}$ on $\cF$ and 
  $\nabla^{\cF'}_{0}=\sigma_{2}(1-f)\nabla_{0}$ on $\cF'$, where $\nabla_{0}=d_{\sigma_{2}}$ is the trivial 
  $\sigma_{2}$-connection on $\cA_{2}^{q}$. Therefore, we see that $\left(\nabla_{0}-\nabla_{1}\right)\xi$ is equal to
  \begin{equation*}
     \nabla_{0}\xi-\nabla^{\cF}_{0}(f\xi)-\nabla^{\cF'}_{0}\left((1-f)\xi\right)=
      \nabla_{0}\xi - 
     \sigma_2(f)\nabla_{0}(f\xi) - \sigma_2(1-f)\nabla_{0}\left((1-f)\xi\right).
  \end{equation*}
Using~(\ref{eq:SigmaConnectionModuleMulti}) we get
\begin{equation*}
   \sigma_2(f)\nabla_{0}(f\xi) = \sigma_{2}(f) \left( (\nabla_{0}f)\xi+\sigma_{2}(f) \nabla_{0}\xi\right)=  
   \sigma_{2}(f)(d_{\sigma_{2}}f)\xi+ \sigma_{2}(f)\nabla_{0}\xi.
\end{equation*}\
Similarly, we have 
\begin{equation*}
   \sigma_2(1-f)\nabla_{0}\left((1-f)\xi\right) = \left(\sigma_{2}(f)-1\right)(d_{\sigma_{2}}f)\xi+ 
   \left(1-\sigma_{2}(f)\right)\nabla_{0}\xi,
\end{equation*}where we have used the fact that $d_{\sigma_{2}}(1-f)=-d_{\sigma_{2}}f$. Thus, 
\begin{equation*}
 \left( \nabla_{0} -\nabla_{1}\right)\xi= \left(1-2\sigma_{2}(f)\right) (d_{\sigma_{2}}f)\xi. 
\end{equation*}

Let $\zeta \in \cH$. Then 
\begin{equation*}
   T_{\cF}(\xi\otimes \zeta)= c \left(\nabla_{0}\right)(\xi\otimes\zeta)-c\left(\nabla_{1}\right)(\xi\otimes\zeta) = \left(1-2\sigma_{2}(f)\right) 
  (d_{\sigma_{2}}f)(\xi\otimes \zeta).
\end{equation*}
 Combining this with~(\ref{eq:Commutator1})--(\ref{eq:Commutator2}) we see that $T_{\cF}$ is equal to
 \begin{equation}
 \label{eq:T_FMatrix}
     \left(1-2\sigma_{2}(f)\right) 
     \begin{pmatrix}
         0 & \omega^{+}[D^{-},\sigma^{-}_{2}(f)]\omega^{-} \\
       \omega^{-}[D^{+},\sigma^{+}_{2}(f)]\omega^{+}   & 0
     \end{pmatrix} = 
     \begin{pmatrix}        
      0   & \omega^{+}\tilde{T}_{\cF}^{-}\omega^{-} \\
         \omega^{-}\tilde{T}_{\cF}^{+}\omega^{+} & 0
     \end{pmatrix},
 \end{equation}
 where we have set $\tilde{T}_{\cF}^{\pm}= \left(\omega^{\mp}\right)^{-1}\left(1-2\sigma_{2}(f)\right) 
 \omega^{\mp}[D^{\pm},\sigma^{\pm}_{2}(f)]$.
 As $\sigma_{2}(f)=\sigma_{2}^{+}\circ \sigma_{2}^{-}(f)=\sigma_{2}^{-}\circ \sigma_{2}^{+}(f)$, we have 
 \begin{equation*}
  \left(\omega^{\mp}\right)^{-1}\sigma_{2}(f)\omega^{\mp}=\left( \sigma_{2}^{\mp}\right)^{-1}\circ  \sigma_{2}^{+}\circ 
  \sigma_{2}^{-}(f)=\sigma^{\pm}_{2}(f),
 \end{equation*}and hence  $\tilde{T}_{\cF}^{\pm}=\left(1-2\sigma_{2}^{\pm}(f)\right) [D^{\pm},\sigma^{\pm}_{2}(f)]$. 
We also note that
 \begin{equation}
 \label{eq:T_FInequality}
        \| T_{\cF}\|= \op{max} \left\{ \| T_{\cF}^{+}\|,  \|T_{\cF}^{-}\| \right\} 
        \leq \|\omega^{+}\| \|\omega^{-}\| \op{max} \left\{ \| \tilde{T}_{\cF}^{+}\|,  \|\tilde{T}_{\cF}^{-}\| 
       \right\}.  
 \end{equation}
 Thus,
 \begin{equation*}
     \| T_{\cF}\|
        \leq \|\omega^{+}\| \|\omega^{-}\| \op{max} \left\{   
       \left\|\left(1-2\sigma_{2}^{+}(f)\right) [D^{+},\sigma^{+}_{2}(f)]\right\|,  
       \left\|\left(1-2\sigma_{2}^{-}(f)\right) [D^{-},\sigma^{-}_{2}(f)]\right\|\right\}.
 \end{equation*}

 Consider the projective module $\hat{\cF}=\hat{f}\cA^{q}$, where 
 $\hat{f}=k_{2}^{-\frac{1}{2}}f k_{2}^{\frac{1}{2}}$. We note that if $f^{*}=f$, then 
 $\sigma_{2}(\hat f)=k_{2}^{\frac{1}{2}}f k_{2}^{-\frac{1}{2}}=\hat{f}^{*}$. We also observe that
 \begin{equation*}
     \sigma_{2}^{\pm}(\hat{f})
     =k_{2}^{\pm}\left(k_{2}^{+}k_{2}^{-}\right)^{-\frac{1}{2}} f     \left(k_{2}^{+}k_{2}^{-}\right)^{\frac{1}{2}}\left(k_{2}^{\pm}\right)^{-1}
     =h_{2}^{\pm 1} f h_{2}^{\mp 1}, \qquad 
     \text{where }
     h_{2}:=\left( k_{2}^{+}\right)^{\frac{1}{2}} \left(k_{2}^{-}\right)^{-\frac{1}{2}}.
 \end{equation*}Therefore, we see that 
 \begin{equation}
 \label{eq:T_hatF^+}
 \begin{split}
 \tilde{T}_{\hat{\cF}}^{+}     & = h_{2}(1-2f)h_{2}^{-1}[D^{+},h_{2}fh_{2}^{-1}]\\
      & = h_{2}(1-2f) \left( h_{2}^{-1}[D^{+},h_{2}]f+[D^{+}, f]+f[D^{+},h_{2}^{-1}]h_{2}\right)h_{2}^{-1}
      \\
      & =  h_{2}(1-2f) \left( h_{2}^{-1}[D^{+},h_{2}]f+[D^{+}, f]-fh_{2}^{-1}[D^{+},h_{2}]\right)h_{2}^{-1}.
 \end{split}
 \end{equation}
 Likewise, 
  \begin{align*}
 \tilde{T}_{\hat{\cF}}^{-}     & = h_{2}^{-1}(1-2f)h_{2}[D^{-},h_{2^{-1}}fh_{2}]\\
      & = h_{2}^{-1}(1-2f) \left( h_{2}[D^{-},h_{2}^{-1}]f+[D^{-}, f]+f[D^{-},h_{2}^{-1}]h_{2}^{-1}\right)h_{2}
      \\
      & =  h_{2}^{-1}(1-2f) \left( -[D^{+},h_{2}]h_{2}^{-1}f+[D^{+}, f]-f[D^{-},h_{2}^{-1}]h_{2}^{-1}\right)h_{2}.
 \end{align*}
 Noting that  $\|f\|\geq 1$ and  $[D^{-},h_{2}]h_{2}^{-1}=\left( h_{2}^{-1}[D^{+},h_{2}]\right)^{*}$ we deduce that 
 \begin{equation*}
     \|\tilde{T}_{\hat{\cF}}^{\pm}\|\leq \|1-2f\| \|f\|\|h_{2}\|\|h_{2}^{-1}\|\left( 
     \|[D^{\pm},f]\|+2\|h_{2}^{-1}[D^{+},h_{2}]\|\right).
 \end{equation*}Combining this with~(\ref{eq:T_FInequality}) we see there is a constant $C=C(f)$ depending only on $f$ such that
 \begin{equation}
 \label{eq:UpperBoundT_F}
      \| T_{\hat\cF}\| \leq C(f) \|\omega^{+}\| \|\omega^{-}\| \|h_{2}\|\|h_{2}^{-1}\|\left( 1+\|h_{2}^{-1}[D^{+},h_{2}]\|\right).
 \end{equation}
 
 Let $\cB=\{\beta_{0},\ldots, \beta_{N}\}$ be a spanning set of $K_{0}(\cA_{2})\otimes \Q$ 
 consisting of equivalence 
 classes of finitely generated projective modules over $\cA_{2}$. For $j=0,\ldots, N$ we represent $\beta_{j}$ by a projective module
 $\cF_{j}=f_{j}\cA_{2}^{q_{j}}$ with $f_{j}=f_{j}^{*}=f_{j}^{2}\in 
 M_{q_{j}}(\cA_{2})$.  Consider also the module $\hat{\cF}_{j}=\hat{f}_{j}\cA_{2}^{q_{j}}$, where $\hat{f}=k_{2}^{-\frac{1}{2}}f k_{2}^{\frac{1}{2}}$.
 As $\sigma_2(\hat{f}_{j})=\hat{f}_{j}^{*}$ we may endow 
 $\hat{\cF}_{j}$ and $\hat{\cF}_{j}'=(1-\hat{f}_{j})\cA_{2}^{q}$ with the $\sigma_{2}$-Hermitian structures given by 
 Lemma~\ref{lem:GrassmannianSigmaHermitianConnection}. Moreover, by Lemma~\ref{lem:GrassmannianSigmaHermitianConnection} the 
 Grassmannian $\sigma_{2}$-connections on $\hat{\cF}_{j}$ and 
 $\hat{\cF}'_{j}$ are $\sigma_{2}$-Hermitian $\sigma_{2}$-connections. It then follows from the proof of Theorem~\ref{thm:EigenvalueInequality} that if we set 
 $C_{0}:=\max\left\{\|T_{\hat{\cF}_{1}}\|,\ldots ,\|T_{\hat{\cF}_{N}}\|\right\}$, then, 
 for any finitely generated 
Hermitian projective module $\cE$ over $\cA_1$ and any $\sigma_{1}$-Hermitian $\sigma_{1}$-connection $\nabla^{\cE}$ on $\cE$, we have 
\begin{equation}
\label{eq:InequalityLastTheorem}
\left|\lambda_1(D_{\omega, \nabla^{\cE}})\right|\leq  C_{0} \|k_{1}^{-1}\|.
\end{equation}
In view of~(\ref{eq:UpperBoundT_F}) there is a constant $C=C(f_{1},\ldots, f _{N})$ depending only 
on $f_{0},\ldots,f_{N}$ such that $C_{0}\leq C \|\omega^{+}\| \|\omega^{-}\| \|h_{2}\|\|h_{2}^{-1}\|\left( 
1+\|h_{2}^{-1}[D^{+},h_{2}]\|\right)$.
Thus, 
\begin{equation*}
\left|\lambda_1(D_{\omega, \nabla^{\cE}})\right|\leq  C \|k_{1}^{-1}\|  \|\omega^{+}\| \|\omega^{-}\| 
\|h_{2}\|\|h_{2}^{-1}\|\left( 1+\|h_{2}^{-1}[D^{+},h_{2}]\|\right).
\end{equation*}
This proves the result. 
\end{proof}

%%%% Ordinary spectral triples.

\section{Inequalities of $\fs$-Eigenvalues. Geometric Applications}
\label{sec:InequalityApp}
In this section, we derive various applications of the $\fs$-eigenvalue inequalities 
from the previous section. These applications concern ordinary spectral triples, a conformal version of the original Vafa-Witten inequality for Dirac operators, conformal deformations of spectral triples, spectral triples over noncommutative tori and duals of discrete cocompact subgroups of Lie groups.
\setcounter{subsubsection}{0}
\subsection{Ordinary spectral triples}
The version of Vafa-Witten inequality for ordinary spectral triples of Moscovici~\cite{Mo:EIPDNCG} holds for ordinary spectral triples 
satisfying Poincar\'e duality. However, as we saw in 
Section~\ref{sec:Poincare Duality for Twisted Spectral Triples}, there are various examples of ordinary spectral triples that are in 
Poincar\'e duality with twisted spectral triples. We have the following extension of Moscovici's result to this setting. 

\begin{theorem}\label{thm:VWSTdualTST}
    Let $(\cA_{1},\cH,D)$ be an ordinary spectral triple which  has finite topological type and is in Poincar\'e 
    duality with a twisted spectral triple $(\cA_2,\cH,D)_{\sigma_2}$, where $\sigma_{2}$ is ribbon. Then there is a constant $C>0$ such that, 
    for any Hermitian finitely generated  projective module $\cE$ over $\cA_1$ and any 
    Hermitian connection $\nabla^{\cE}$ on $\cE$, we have
     \begin{equation*}
         \left| \lambda_{1}( D_{\nabla^{\cE}}) \right| \leq C, 
     \end{equation*}where $\lambda_{1}( D_{\nabla^{\cE}}) $ is the eigenvalue of $D_{\nabla^{\cE}}$ with the smallest absolute value. 
\end{theorem}
\begin{proof}
    This is an immediate consequence of Theorem~\ref{thm:EigenvalueInequality}. For ordinary spectral triples $\sigma$-Hermitian structures on finitely generated 
projective modules are just usual Hermitian 
structures by taking $\fs$ to be the identity map, and so $\fs$-eigenvalues are just ordinary eigenvalues.  
\end{proof}

\subsection{Dirac spectral triples}For Dirac operators coupled with Hermitian connections on spin manifolds, 
the Vafa-Witten bound in~(\ref{eq:Intro.Vafa-Witten-inequality}) depends on the metric on a somewhat elusive way. 
We refer to~\cite{An:OFVWBTDT, Ba:UBFEDOSM, DM:VWBCPS, Go:VWECSS, He:EVWBS} for various attempts to understand this dependence on the metric.  
As a consequence of the results of the previous section, we shall establish a control of this bound under conformal changes of metric. This result can be seen as a version in conformal geometry of the Vafa-Witten inequality. We also note that this result is stated without any reference to noncommutative geometry whatsoever. 

\begin{theorem}\label{thm:CoupledDiracOpWV}
Let $(M^{n},g)$ be an even dimensional compact Riemannian spin manifold.   Then, there is a constant $C>0$ such that, for any conformal factor $k\in C^{\infty}(M)$, $k>0$, and 
any Hermitian vector bundle $\cE$ equipped with a Hermitian connection $\nabla^{\cE}$, we have
  \begin{equation}
  \label{eq:UpperBdConfFactor}
      \left| \lambda_{1}(D_{\hat{g},\nabla^{\cE}})\right| \leq C \|k\|_{\infty}, \qquad \hat{g}:=k^{-2}g,
  \end{equation}
  where $\|k\|_{\infty}$ is the maximum value of $k$. 
\end{theorem}
\begin{proof}
    Set $\cA=C^{\infty}(M)$. As $\cA$ is a commutative algebra we can identify left and right modules over $\cA$. It would be more convenient to work with left modules instead of right modules as we have been doing so far. In addition,  this also provides us with a natural identification of $\cA$-modules $\cE_{1}\otimes_{\cA} \cE_{2}\simeq 
    \cE_{2}\otimes_{\cA} \cE_{1}$ for the tensor products of two modules $\cE_{1}$ and $\cE_{2}$; the isomorphism being given by the flip map $\xi_1\otimes \xi_2\rightarrow \xi_2\otimes \xi_1$. 
    
   Let $k\in C^\infty(M)$, $k>0$, and set $\hat{g}=k^{-2}g$. We denote by $\cH_{g}$ (resp., $\cH_{\hat{g}}$) the Hilbert space $L^{2}_{g}(M,\sS)$ 
    (resp., $L^{2}_{\hat{g}}(M,\sS)$).  Let $U:\cH_{\hat{g}}\rightarrow \cH_{g}$ be the multiplication operator $k^{\frac{n}{2}}$.
    Then $U$ is a unitary operator and by Proposition~\ref{prop:ConformalChangeDiracST} it intertwines the spectral triple 
    $\left(\cA, \cH_{\hat{g}}, \sD_{\hat{g}}\right)$ with the pseudo-inner twisted spectral triple 
    $\left( \cA, \cH_{g}, \sD_{\sqrt{k}}\right)$, where $\sD_{\sqrt{k}}=\sqrt{k}\,\sD_{g}\!\!\sqrt{k}$. In particular, 
    $\sD_{\sqrt{k}}=U^{-1}\sD_{\hat{g}}U$. 
   Let $c:\Lambda^{*}_{\C}T^{*}M\rightarrow \End \sS$ be the Clifford 
    representation with respect to the metric $g$ and set $c_k=kc$. For $a$ and $b$ in $\cA$ we have
   \begin{equation}
       a[\sD_{\sqrt{k}},b]=a \left[ \sqrt{k}\,\sD_{g}\!\!\sqrt{k},b\right]=ka[\sD_{g},a]=kc(adb)=c_k(adb).
       \label{eq:VW.commutator-c-Dirac}
   \end{equation}Therefore, we see that
    \begin{equation*}
        \Omega_{\sD_{\sqrt{k}}}^{1}(\cA)=\op{Span}\left\{ c_{k}(\omega); \ \omega\in C^{\infty}(M,T^{*}_{\C}M)\right\}.
    \end{equation*}
   
    Let $E$ be a Hermitian vector bundle and $\nabla^{E}:C^{\infty}(M,E)\rightarrow 
    C^{\infty}(M, T^{*}M\otimes E)$ a Hermitian connection on $E$.  Set $\cE=C^{\infty}(M, E)$. This is a finitely 
    generated projective module over $\cA$ and the Hermitian metric of $E$ defines a Hermitian metric on $\cE$ in the 
    sense of Definition~\ref{def:connection.Hermitian-metric}. Note that $\nabla^{E}$ is a linear map from $\cE$ to $C^{\infty}(M, T^{*}M\otimes 
    E)=C^{\infty}(M,T^{*}_{\C}M)\otimes_{\cA}\cE$. Consider the linear map $\nabla^{\cE}$ from $\cE$ to 
    $\Omega_{\sD_{\sqrt{k}}}^{1}(\cA)\otimes_{\cA}\cE\simeq \cE\otimes_{\cA}\Omega_{\sD_{\sqrt{k}}}^{1}(\cA)$ defined by
    \begin{equation*}
        \nabla^{\cE}:=(c_{k}\otimes 1_{\cE})\circ \nabla^{E}.
    \end{equation*}
   Let $\xi\in \cE$ and $a \in \cA$. Using~(\ref{eq:VW.commutator-c-Dirac}) we get
   \begin{equation*}
       \nabla^{\cE}(a\xi)=(c_{k}\otimes 1_{\cE})\left(da\otimes 
       \xi+a\nabla^{E}\xi\right)=c_{k}(da)+a\nabla^{\cE}\xi=[\sD_{\sqrt{k}},a]\xi+a\nabla^{\cE}\xi. 
   \end{equation*}Therefore, $\nabla^{\cE}$ is a connection on the finitely generated projective module $\cE$. 
   
   Let $\xi$ and  $\eta$ be in $\cE$. We write $\nabla^{E}\xi=\sum \omega_{\alpha}\otimes \xi_{\alpha}$, with
   $\omega_{\alpha}\in C^{\infty}(M,T^{*}_{\C}M)$ and $\xi_{\alpha}\in \cE$. Then \[\nabla^{\cE}\xi= \sum_{\alpha} c_k(\omega_{\alpha})\otimes \xi_{\alpha}.\] 
      As $c_{k}(\omega_{\alpha})^{*}=kc(\omega_{\alpha})^{*}=-kc(\overline{\omega_{\alpha}})$, we have
      \begin{equation*}
          \acoup{\nabla^{\cE}\xi}{\eta}=\sum c_{k}(\omega_{\alpha})^{*}\acoup{\xi_{\alpha}}{\eta}= \sum 
          -c_{k}(\overline{\omega_{\alpha}})\acoup{\xi_{\alpha}}{\eta} = -c_{k}\left( 
          \acoup{\nabla^{E}\xi}{\eta}\right).
      \end{equation*}Thus,
      \begin{equation*}
          \acoup{\xi}{\nabla^{\cE}\eta}- \acoup{\nabla^{\cE}\xi}{\eta}=c_{k}\left( 
          \acoup{\xi}{\nabla^{E}\eta}\right)+ c_{k}\left( 
          \acoup{\nabla^{E}\xi}{\eta}\right)= c_{k}\left\{ 
          \acoup{\xi}{\nabla^{E}\eta}+\acoup{\nabla^{E}\xi}{\eta}\right\}. 
      \end{equation*}As $\nabla^{E}$ preserves the Hermitian metric of $E$ we have $ 
      \acoup{\xi}{\nabla^{E}\eta}+\acoup{\nabla^{E}\xi}{\eta}=d\acoup{\xi}{\eta}$. Therefore, using~(\ref{eq:VW.commutator-c-Dirac}) we get
      \begin{equation*}
          \acoup{\xi}{\nabla^{\cE}\eta}- \acoup{\nabla^{\cE}\xi}{\eta}= c_{k}\left( d\acoup{\xi}{\eta}\right)= 
          [\sD_{\sqrt{k}},\acoup{\xi}{\eta}].
      \end{equation*}
      This shows that the connection $\nabla^{\cE}$ is Hermitian. 
    
As $\nabla^{\cE}$ is a connection on $\cE$ we can form the operator 
$\sD_{\sqrt{k},\nabla^{\cE}}:=(\sD_{\sqrt{k}})_{\nabla^{\cE}}$. In what follows we identify $\cH_{g}(\cE)=\cE\otimes_{\cA}\cH_{g}$ with $\cH_{g}\otimes_{\cA}\cE\simeq 
 L^{2}(M,\sS\otimes E)$, so that we regard $ \sD_{\sqrt{k},\nabla^{\cE}}$ as an unbounded operator of  $L^{2}(M,\sS\otimes E)$. 
Let $\zeta \in C^{\infty}(M,\sS)$ and $\xi \in \cE$. We write $\nabla^{E}\xi=\sum \omega_{\alpha}\otimes \xi_{\alpha}$, where 
   $\omega_{\alpha}\in C^{\infty}(M,T^{*}_{\C}M)$ and $\xi_{\alpha}\in \cE$. Then
   \begin{equation}
       \sD_{\sqrt{k},\nabla^{\cE}}(\zeta\otimes \xi)= \sD_{\sqrt{k}}\zeta \otimes \xi +\sum c_{k}(\omega_{\alpha})\zeta 
       \otimes \xi_{\alpha}. 
       \label{eq:VW.DsqrtknablaE}
   \end{equation}
    
   Bearing this in mind, let $\sD_{\hat{g},\nabla^{E}}$ be the coupling of the Dirac operator $\sD_{\hat{g}}$ with the Hermitian connection $\nabla^{E}$. This operator acts on the 
   sections of $\sS\otimes E$. If we
    let $\hat{c}:\Lambda^{*}_{\C}T^{*}M\rightarrow \End \sS$ be the Clifford 
    representation with respect to the metric $\hat{g}$, then we have
        \begin{equation}
        \sD_{\hat{g},\nabla^{E}}(\zeta\otimes \xi)=\sD_{\hat{g}}\zeta\otimes \xi +\sum \hat{c}(\omega_{\alpha})\zeta 
        \otimes \xi_{\alpha}.
            \label{eq:DhatgnablaE}
    \end{equation}
    Let $a$ and $b$ be smooth functions on $M$. Using~(\ref{eq:VW.commutator-c-Dirac}) and the fact that $\sD_{\sqrt{k}}=U^{-1}\sD_{\hat{g}}U$ 
    we see that $c_{k}(adb)=a[\sD_{\sqrt{k}},b]=U^{*}a[\sD_{\hat{g}},b]U=U^{*}\hat{c}(adb)U$.
    Thus,
    \begin{equation*}
        c_{k}(\omega)=U^{*}\hat{c}(\omega)U\qquad \forall \omega \in C^{\infty}(M,T^{*}_{\C}M).
    \end{equation*}Combining this with~(\ref{eq:VW.DsqrtknablaE}) and~(\ref{eq:DhatgnablaE}) we then obtain
    \begin{equation*}
         \sD_{\sqrt{k},\nabla^{\cE}}(\zeta\otimes \xi)= U^{*}D_{\hat{g}}U\zeta\otimes \xi + \sum 
         U^{*}\hat{c}(\omega_{\alpha})U\zeta\otimes \xi= (U^{*}\otimes 1_{\cE})\sD_{\hat{g},\nabla^{E}}(U\otimes 
         1_{\cE})(\zeta\otimes \xi).
    \end{equation*}This shows that \[\sD_{\sqrt{k},\nabla^{\cE}}=(U^{*}\otimes 1_{\cE})  
    \sD_{\hat{g},\nabla^{E}}(U\otimes 1_{\cE}).\]
    It then follows that the operators $\sD_{\hat{g},\nabla^{E}}$ and $\sD_{\sqrt{k},\nabla^{\cE}}$ have the same 
spectrum.

We may apply Theorem~\ref{eq:ConformalVersionInequality} to the Dirac spectral triple $\left( 
\cA,\cH_{g},\sD_{g}\right)$ and the pseudo-inner twisting $\omega=\sqrt{k}$. In this case $\omega^{\pm}=\sqrt{k}$ and $k_{1}=k_{2}=h_{2}=1$. Therefore, there is a 
constant $C>0$, which is independent of $k$ and of the pair $\left(E,\nabla^{E}\right)$, such that
\begin{equation*}
    \left| \lambda_{1}(\sD_{\hat{g},\nabla^{E}})\right|=\left| 
    \lambda_{1}(\sD_{\sqrt{k},\nabla^{\cE}})\right|\leq C\|\sqrt{k}\|\|\sqrt{k}\|=C\|k\|_{\infty}.
\end{equation*}This completes the proof.
\end{proof}

\begin{remark}
    As it follows from Remark~\ref{rem:kLip}, Theorem~\ref{thm:CoupledDiracOpWV} continues to hold if we only require the conformal factor $k$ to be 
    Lipschitz. 
\end{remark}

\subsection{Conformal deformations of ordinary spectral triples} 
We shall now use Theorem~\ref{eq:ConformalVersionInequality} to obtain a noncommutative version of Theorem~\ref{thm:CoupledDiracOpWV}, that is, a conformal version of Moscovici's inequality for ordinary spectral triples. 

\begin{theorem}
\label{thm:ConformalPerST-WVIneq1}
 Let $(\cA_{1}, \cH, D)$ be an ordinary spectral triple which has finite topological type and admits an ordinary 
 Poincar\'e dual $(\cA_{2}, \cH, D)$. In addition, for $j=1,2$ let $k_{j}$ be a positive invertible element of $\cA_{j}$. 
 Set $k=k_{1}k_{2}$ and denote by $\sigma_{1}$ the inner automorphism of $\cA_{1}$ given by $\sigma_{1}(a)=k_{1}^{2}ak_{1}^{-2}$, $a\in \cA_{1}$. 
 Then, there is a constant $C>0$ independent of $k_{1}$ and $k_{2}$, such that for any  Hermitian finitely generated 
 projective module $\cE$ over $\cA_1$ equipped with a $\sigma_{1}$-Hermitian connection $\nabla^{\cE}$, we have 
\begin{equation*}
    \left|\lambda_1\left( D_{k,\nabla^{\cE}}\right)\right|\leq  C \|k_{1}^{-1}\| \|k_{1}k_{2}\|^{2},
\end{equation*}
 where $\lambda_1\left( D_{k,\nabla^{\cE}}\right)$ is the first $\fs$-eigenvalue of the operator 
 $D_{k,\nabla^{\cE}}:=\left(kDk\right)_{\nabla^{\cE}}$.  
\end{theorem}
\begin{proof}
        This follows from Theorem~\ref{eq:ConformalVersionInequality} by taking $\omega=k_{1}k_{2}$, noting that in
    this case $\omega^{\pm}=k_{1}k_{2}$ and $h_{2}=1$.
\end{proof}

\subsection{Noncommutative tori and conformal weights}
Let $\theta \in \R$. Given a conformal weight $\varphi$ in the sense of~(\ref{eq:ConformalWeight}) on the noncommutative 
torus $\cA_{\theta}$, we denote by $(\cA_{\theta}^{\opp},\cH_{\varphi},D_{\varphi})_{\sigma}$ and 
$(\cA_{\theta},\cH_{\varphi},D_{\varphi})$ the associated (twisted) spectral triples as described in Section~\ref{subsec:TSTNCTorus}.  
By Proposition~\ref{Prop:PDNCtorusTST} the twisted spectral triple $(\cA_{\theta}^{\opp},\cH_{\varphi},D_{\varphi})_{\sigma}$ and the 
ordinary spectral triple $(\cA_{\theta},\cH_{\varphi},D_{\varphi})$ are in Poincar\'e duality. Moreover, these 
spectral triples have finite topological type since $K_{0}(\cA_{\theta})\simeq \Z^{2}$ thanks to a result of 
Pimsner-Voiculescu~\cite{PV:IIRCAAFA}. Therefore, the $\fs$-eigenvalue inequality of Theorem~\ref{thm:EigenvalueInequality} holds for 
$(\cA_{\theta}^{\opp},\cH_{\varphi},D_{\varphi})_{\sigma}$. As we shall now refine this inequality to include a control 
of the dependence of the bound on the conformal weight. 

\begin{theorem}\label{thm:VWNCtorus1}
Let $\theta \in \R$. Then there is a constant $C_{\theta}>0$ such that for any  conformal weight $\varphi$ with the Weyl factor $k\in \cA_{\theta}$, $k>0$, 
and any Hermitian finitely generated  projective module $\cE$ over $\cA_{\theta}^{\opp}$ equipped with a  
$\sigma$-Hermitian $\sigma$-connection $\nabla^{\cE}$,  we have 
        \begin{equation*}
             \left|  \lambda_{1}(D_{\varphi,\nabla^{\cE}}) \right| \leq C_{\theta}\|k\|^{2},
        \end{equation*}where $\lambda_{1}(D_{\varphi,\nabla^{\cE}})$ is the first $\fs$-eigenvalue of 
        $D_{\varphi,\nabla^{\cE}}:=(D_{\varphi})_{\nabla^\cE}$. 
\end{theorem}
\begin{proof}
Let $\cE$ be a Hermitian finitely generated projective module over $\cA_{\theta}^{\opp}$ equipped with a 
$\sigma$-Hermitian $\sigma$-connection $\nabla^{\cE}$ on $\cE$.  As mentioned in Section~\ref{subsec:TSTNCTorus}, the unitary operator 
$W:\cH\rightarrow \cH_{\varphi}$ given by~(\ref{eq:IsoHandH_phi}) intertwines the twisted spectral triple 
$(\cA_{\theta}^{\opp},\cH_{\varphi},D_{\varphi})_{\sigma}$ with the pseudo-inner twisted spectral triple 
$(\cA_{\theta}^{\opp},\cH,\omega D \omega)_{\sigma}$ with $\omega=
 \begin{pmatrix}
     R_{k} & 0 \\
     0 & 1
 \end{pmatrix}$, where $R_{k}$ is the right multiplication by $k$. The operator $W$ gives rise to an algebra isomorphism 
 $\Ad_{W}:\cL(\cH_{\varphi})\rightarrow \cL(\cH)$ given by 
 \begin{equation*}
     \Ad_{W}T=W^{-1}TW \qquad \forall T\in \cL(\cH_{\varphi}).
 \end{equation*}As $W$ is an intertwiner, $ \Ad_{W}\left( a^{\opp}_{\varphi}[D_{\varphi},b^{\opp}_{\varphi}]_{\sigma}\right)= 
     a^{\opp}[D_{\omega},b^{\opp}]_{\sigma}$ for all $a$ and $b$ in $\cA_{\theta}^{\opp}$. 
     Here $a_{\varphi}^{\opp}$ is given by~(\ref{eq:aphi0}).
     Thus $\Ad_{W}$ induces an isomorphism 
 $\Ad_{W}:\Omega_{D_{\varphi},\sigma}^{1}(\cA_{\theta}^{\opp})\rightarrow \Omega_{D_{\omega},\sigma}^{1}(\cA_{\theta}^{\opp})$.  

 Let $\tilde{\nabla}^{\cE}:\cE\rightarrow 
 \cE^{\sigma}\otimes_{\cA^{\opp}_{\theta}}\Omega_{D_{\omega},\sigma}^{1}(\cA_{\theta}^{\opp})$ be the linear map defined by
 \begin{equation*}
     \tilde{\nabla}^{\cE}\xi= \left[1_{\cE}\otimes \Ad_{W}\right]\left( \nabla^{\cE}\xi\right) \qquad \forall \xi \in \cE.
 \end{equation*}
 Let $\xi$ and $\eta$ be in $\cE$ and $a\in \cA^{\opp}_{\theta}$. Then $\tilde{\nabla}^{\cE}(\xi a)$ is equal to
  \begin{equation*}
      \sigma^{\cE}(\xi)\otimes \Ad_{W}\left( [D_{\varphi},a_{\varphi}^{\opp}]_{\sigma}\right)+ 
       \left[1_{\cE}\otimes \Ad_{W}\right]\left( \left(\nabla^{\cE}\xi\right)a^{\opp}_{\varphi}\right) 
       =\sigma^{\cE}(\xi)\otimes [D_{\omega},a^{\opp}]_{\sigma}+ 
        \left(\tilde{\nabla}^{\cE}\xi\right)a^{\opp}. 
  \end{equation*}
%  This shows that $\tilde{\nabla}^{\cE}$ is a $\sigma$-connection
 Moreover, as $\nabla^{\cE}$ is a $\sigma$-Hermitian $\sigma$-connection, the difference $  
 \acoups{\xi}{\tilde{\nabla}^{\cE}\eta}-\sacoup{\tilde{\nabla}^{\cE}\xi}{\eta}$ is equal to
  \begin{equation*}
  \Ad_{W} \acoups{\xi}{\nabla^{\cE}\eta}- \Ad_{W}\sacoup{\nabla^{\cE}\xi}{\eta}= 
      \Ad_{W}\left(\left[D_{\varphi},\sacoup{\sigma^{\cE}(\xi)}{\eta}^{\opp}_{\varphi}\right]\right) = 
      \left[D_{\omega},\sacoup{\sigma^{\cE}(\xi)}{\eta}^{\opp}\right].
  \end{equation*}
 All this shows that $\tilde{\nabla}^{\cE}$ is a $\sigma$-Hermitian $\sigma$-connection on $\cE$.
 
 Let us denote by $\acou{\cdot}{\cdot}$ and $\acou{\cdot}{\cdot}_{\varphi}$ the respective inner products of $\cH(\cE)$ 
 and $\cH_{\varphi}(\cE)$. Let $\xi\in \cE$ and $\zeta\in \cH$. Then
 \begin{equation*}
     \acou{\xi\otimes W\zeta}{\xi\otimes W\zeta}_{\varphi}= 
     \acou{W\zeta}{\acoup{\xi}{\xi}^{\opp}_{\varphi}W\zeta}=\acou{W\zeta}{W\acoup{\xi}{\xi}^{\opp}\zeta}=\acou{\zeta}{\acoup{\xi}{\xi}^{\opp}\zeta}= 
         \acou{\xi\otimes \zeta}{\xi\otimes \zeta}.
 \end{equation*}
 This shows that $1_{\cE}\otimes W$ is a unitary operator from $\cH(\cE)$ onto $\cH_{\varphi}(\cE)$. Moreover, 
 \begin{equation*}
     c\left( \tilde{\nabla}^{\cE}\right)=  c\left( \left(1_{\cE}\otimes
     \Ad_{W}\right)\circ \tilde{\nabla}^{\cE}\right)  = \left( 1_{\cE}\otimes W^{-1}\right)\circ 
     c\left(\nabla^{\cE}\right)\circ \left( 1_{\cE}\otimes W\right).
 \end{equation*}Thus,
 \begin{align*}
     D_{\omega,\tilde{\nabla}^{\cE}} =\sigma^{\cE}\otimes D_{\omega}+c_{D_{\omega},\sigma}\left( 
     \tilde{\nabla}^{\cE}\right)& = \sigma^{\cE}\otimes W^{-1}D_{\varphi}W+ \left( 1_{\cE}\otimes W^{-1}\right)\circ 
     c_{D_{\varphi},\omega}(\nabla^{\cE})\circ \left( 1_{\cE}\otimes W\right)\\ 
     & =  \left( 1_{\cE}\otimes W\right)^{-1}
     D_{\varphi,\nabla^{\cE}} \left( 1_{\cE}\otimes W\right)^{-1}. 
 \end{align*}We then deduce that $( \fs\otimes 1_{\cH_{\varphi}})D_{\omega,\tilde{\nabla}^{\cE}}$ is equal to
 \begin{equation*}
    ( \fs\otimes 1_{\cH_{\varphi}})\left( 1_{\cE}\otimes W\right)^{-1}
     D_{\varphi,\nabla^{\cE}} \left( 1_{\cE}\otimes W\right)^{-1}=\left( 1_{\cE}\otimes W\right)^{-1}
     ( \fs\otimes 1_{\cH_{\varphi}})D_{\varphi,\nabla^{\cE}} \left( 1_{\cE}\otimes W\right)^{-1}.
 \end{equation*}Therefore, the operators $ ( \fs\otimes 1_{\cH_{\varphi}})D_{\omega,\tilde{\nabla}^{\cE}}$ and $ ( 
 \fs\otimes 1_{\cH_{\varphi}})D_{\varphi,\nabla^{\cE}} $ have same eigenvalues, that is,  
 $D_{\omega,\tilde{\nabla}^{\cE}}$ and $D_{\varphi,\nabla^{\cE}}$ have same $\fs$-eigenvalues. 
 
Bearing this in mind, we may apply Theorem~\ref{eq:ConformalVersionInequality} to the spectral triple $(\cA_{\theta}^{\opp},\cH, D)$ and the pseudo-inner twisting $\omega$. In this case $\omega^{+}=R_{k}$ and $\omega^{-}=1$, so that $\|\omega^{+}\|=\|k\|$ and $\|\omega^{-}\|=1$. Moreover, 
$k_{1}=k^{-1}$ and $k_{2}^{\pm}=h_{2}=1$. Therefore, there is a constant $C_{\theta}>0$, which is independent of the conformal weight $\varphi$ and of the pair 
$(\cE,\nabla^{\cE})$, such that
\begin{equation*}
  \left|  \lambda_{1}(D_{\varphi,\nabla^{\cE}}) \right|  = \left|  \lambda_{1}(D_{\omega,\tilde{\nabla}^{\cE}}) 
  \right|  \leq C_{\theta}\|k\|^{2}.
\end{equation*}The proof is complete. 
\end{proof}

Finally, we deal with the ordinary spectral triple $(\cA_{\theta},\cH_{\varphi},D_{\varphi})$. 

\begin{theorem}\label{thm:VWNCtorus2}
   Let $\theta \in \R$. Then, there is a constant $C_{\theta}>0$ such that, for any  conformal weight $\varphi$ with 
   the Weyl factor $k$ and any Hermitian finitely generated  projective module $\cE$ over $\cA_{\theta}$ equipped with a  
Hermitian connection $\nabla^{\cE}$,  we have  
         \begin{equation*}
             \lambda_{1}(D_{\varphi,\nabla^{\cE}})\leq C_{\theta} \left\|k^{-\frac{1}{2}}\right\| \left\|k^{\frac32}\right\|  \left(1+ 
             \left\|k^{\frac{1}{2}}\partial \left(k^{-\frac12}\right) \right\|\right),
         \end{equation*}where $\lambda_{1}(D_{\varphi,\nabla^{\cE}})$ is the first eigenvalue of the operator
         $D_{\varphi,\nabla^{\cE}}=\left(D_{\varphi}\right)_{\nabla^{\cE}}$ and $\partial$ is the holomorphic derivation~(\ref{eq:Derivation}). 
\end{theorem}
\begin{proof}
    Let $\cE$ be a Hermitian finitely generated  projective module over $\cA_{\theta}$ and $\nabla^{\cE}$ a  
Hermitian connection on $\cE$. The unitary operator $W$ given by~(\ref{eq:IsoHandH_phi}) also intertwines $(\cA_{\theta},\cH_{\varphi},D_{\varphi})$ with 
the pseudo-inner twisted spectral triple $(\cA_{\theta},\cH_{\varphi},\omega D \omega)$, where $\omega$ is given 
by~(\ref{eq:TST.Rk-inner-twisting}). By arguing as in the proof of Theorem~\ref{thm:VWNCtorus1} we can construct a Hermitian connection 
$\tilde{\nabla}^{\cE}$ on $\cE$ so that the operator $D_{\omega,\tilde{\nabla}^{\cE}}:=\left(\omega D 
\omega\right)_{\tilde{\nabla}^{\cE}}$ has the same eigenvalues as $D_{\varphi,\nabla^{\cE}}$. We are thus reduced to apply 
Theorem~\ref{eq:ConformalVersionInequality} to $(\cA_{\theta},\cH_{\varphi},\omega D \omega)$. In this case  $\omega^{+}=R_{k}=k^{\opp}$ and $\omega^{-}=1$, so that 
$k_{1}=1$, while $k_{2}^{+}=k^{-1}$ and $k_{2}^{-}=1$, and hence 
$h_{2}=(k_2^+)^{\frac12}(k_2^-)^{-\frac12}=k^{-\frac12}$. 

Having said this, an observation of the proof of Theorem~\ref{eq:ConformalVersionInequality} shows that, in the current setting, we can slightly improve the 
inequality~(\ref{eq:UpperBoundT_F}). 
Using the notation of the proof of Theorem~\ref{eq:ConformalVersionInequality}, we observe that~(\ref{eq:T_FMatrix}) gives 
 \begin{equation*}
   T_{\cF}= 
     \begin{pmatrix}        
      0   & k^{\opp}\tilde{T}_{\cF}^{-} \\
        \tilde{T}_{\cF}^{+}k^{\opp} & 0
     \end{pmatrix}.
 \end{equation*}
Set $h=h_{2}=k^{-\frac12}$. Using~(\ref{eq:T_hatF^+}) we get
 \begin{equation*}
     \tilde{T}_{\hat{\cF}}^{+} k^{\opp} =  h^{\opp}(1-2f^{\opp}) \left( 
     (h^{-1})^{\opp}[D^{+},h^{\opp}]f^{\opp}+[D^{+}, 
     f^{\opp}]-f^{\opp}(h^{-1})^{\opp}[D^{+},h^{\opp}]\right)(h^{-3})^{\opp}.  
  \end{equation*}
There is a similar formula for $k^{\opp}\tilde{T}_{\cF}^{-}$. We also note that $[D^{+},h^{\opp}]=[\partial, h^{\opp}]=\partial h$, 
since $\partial$ is a derivation. It then follows we can replace the estimate~(\ref{eq:UpperBoundT_F}) by
\begin{equation*}
    \|T_{\hat{\cF}}\|\leq C(f)  \|h\|\|h^{-3}\|\left( 1+\|h^{-1}\partial h\|\right),
\end{equation*}where $C(f)$ depends only on $f$. Therefore, by arguing as in the proof of Theorem~\ref{eq:ConformalVersionInequality}, we deduce there is a constant $C_{\theta}$, which is 
independent of $\varphi$ and of the pair $(\cE,\nabla^{\cE})$, such that 
 \begin{equation*}
       \left| \lambda_{1}(D_{\varphi,\nabla^{\cE}})\right| =  \left| \lambda_{1}(D_{\omega,\tilde{\nabla}^{\cE}})\right| \leq C_{\theta}\|h\|\|h^{-3}\|\left( 1+\|h^{-1}\partial h\|\right).
 \end{equation*}This proves the result. 
\end{proof}

\subsection{Duals of discrete cocompact subgroups of Lie groups} 
We end this section with a noncompact example related to duals of discrete cocompact subgroups of Lie groups. 

 Let $\Gamma$ be a torsion free discrete cocompact  subgroup of a connected semisimple Lie group $G$. Set $X=G/K$, 
 where $K$ is a maximal compact subgroup. As mentioned in Example~\ref{ex:BaumConnesST}, we have a pair of ordinary spectral triples 
 $\left(\cA_\Gamma, \cH, D_\tau\right)$ and 
 $\left(C^{\infty}(\Gamma\backslash X),\cH,D_\tau\right)$, where $\cA_{\Gamma}$ is  the holomorphic functional calculus closure of the group algebra $\C\Gamma$, the Hilbert space $\cH=L^2(X,\Lambda_\C^*T^*X)$ is defined by means of the canonical $G$-invariant Riemannian metric of $X$ and $D_{\tau}=d_{\tau}+d_{\tau}^{*}$ with $d_{\tau}=e^{-\tau \varphi}de^{\tau \varphi}$, $\tau\neq 0$, where $\varphi$ is the square of the geodesic distance to the base point 
 $o=K$. 
  
 Let $E$ be a $\Gamma$-equivariant vector bundle over $X$ equipped with a $\Gamma$-invariant Hermitian metric and a 
 $\Gamma$-equivariant Hermitian connection $\nabla^{E}$. 
The connection $\nabla^{E}$ uniquely extends to a covariant derivative,
\begin{equation*}
  \nabla^{E}:C^{\infty}(X,\Lambda^{*}_{\C}T^{*}X\otimes E)\rightarrow C^{\infty}(X,\Lambda^{*}_{\C}T^{*}X\otimes  E) ,
  \end{equation*}
  such that, for all $\zeta \in 
  C^{\infty}(X,\Lambda^{*}_{\C}T^{*}X)$ and $\xi \in C^{\infty}(X,E)$, 
    \begin{equation}
  \nabla^{E}(\zeta\otimes \xi)=d\zeta\otimes \xi+(-1)^{\partial \zeta} \zeta \wedge \nabla^{E}\xi. 
  \label{eq:VW-groups.covariant-derivative-nablaE}
\end{equation}
Following Connes~\cite[IV.9.{$\alpha$}]{Co:NCG} we define $\nabla^{E}_{\tau}:C^{\infty}(X,\Lambda^{*}_{\C}T^{*}X\otimes E)\rightarrow C^{\infty}(X,\Lambda^{*}_{\C}T^{*}X\otimes E)$ by
 \begin{equation*}
    \nabla^{E}_{\tau}:=e^{-\tau \varphi}\nabla^{E}e^{\tau \varphi}=\nabla^{E}_{\tau}+\tau \varepsilon(d\varphi) \otimes 
    1_{\cE},
 \end{equation*}where $\varepsilon(d\varphi)$ is the exterior multiplication by $d\varphi$. 
The operator $  \nabla^{E}_{\tau} + \left(  \nabla^{E}_{\tau}\right)^{*}$ is a Dirac-type operator on the noncompact 
manifold $X$ (see below). As it turns out, its closure in $L^{2}(X,\Lambda^{*}T^{*}_{\C}X)$ is selfadjoint and has 
compact resolvent (see~\cite[IV.9.{$\alpha$}]{Co:NCG}). Therefore, its spectrum consists of isolated real eigenvalues with 
finite multiplicity. The same properties hold for the conformal perturbation $ k \left( \nabla^{E}_{\tau} + 
\left(  \nabla^{E}_{\tau}\right)^{*}\right)k$, where $k=e^{h}$ and $h$ is a selfadjoint element of 
$\C\Gamma$. 

As a consequence of his Vafa-Witten inequality 
for ordinary spectral triples, Moscovici~\cite[Corollary~1]{Mo:EIPDNCG} obtained a version of 
Vafa-Witten's inequality for the operator $  \nabla^{E}_{\tau} + \left(  \nabla^{E}_{\tau}\right)^{*}$. The following is a ``$\C\Gamma$-conformal'' version of Moscovici's result.

 \begin{theorem}\label{prop:Vafa-Witten.duals-discrete-groups}
  Assume that  $\Gamma$ satisfies the Baum-Connes conjecture. Let $h$ be a selfadjoint element of $\C  
  \Gamma$ and set $k=e^{h}\in \cA_{\Gamma}$. Then, there exists a constant $C>0$ independent of $k$ such that, for any 
  $\Gamma$-equivariant Hermitian vector bundle $E$ equipped with a $\Gamma$-invariant  Hermitian connection $\nabla^{E}$, we have
  \begin{equation*}
      \left|\lambda_{1}\left(k\left(\nabla^{E}_{\tau}+\left(\nabla^{E}_{\tau}\right)^{*}\right)k\right)\right|\leq 
      C\left\|k\right\|^{2},
  \end{equation*}where $\lambda_{1}\left(k\left(\nabla^{E}_{\tau}+\left(\nabla^{E}_{\tau}\right)^{*}\right)k\right)$ is the first eigenvalue of the operator 
  $k\left(\nabla^{E}_{\tau}+\left(\nabla^{E}_{\tau}\right)^{*}\right)k$. 
 \end{theorem}
\begin{proof}
  Let $k=e^{h}$, where $h$ is a selfadjoint element of $\C\Gamma$. Set $\cB=C^{\infty}(\Gammab X)$. As $\cB$ is a commutative algebra, in the same way as in the proof of Theorem~\ref{thm:CoupledDiracOpWV},
   we may identify left and right modules over $\cB$, which allows us to work with left modules instead of right modules. 
   Moreover, we let $C^{\infty}(X)^{\Gamma}$ the space of smooth $\Gamma$-periodic functions on $X$ and denote by $\pi$ the canonical fibration of $X$ onto 
   $\Gammab X$.  
  For any function $b \in C^{\infty}(\Gammab X)$ we let $\tilde{b}=b\circ \pi$ be the unique lift of $b$ to a smooth $\Gamma$-periodic function on $X$.  This defines a representation of $\cB$ in $\cH$ as in~(\ref{eq:PD.duals-cocompact.action-cB}). We then form the pseudo-inner twisted (ordinary) spectral triple $\left( \cB, \cH,D_{\tau,k}\right)$, where $D_{\tau,k}:=kD_{\tau}k$.
  
  In what follows, given any 1-form $\omega\in C^{\infty}(X,T^{*}_{\C}X)$, we denote by $\varepsilon(\omega)$ (resp., 
  $\iota(\omega)$) the exterior (resp., interior) product on differential forms over $X$. We then set
  \begin{equation*}
      c_{k}(\omega)=k(\varepsilon(\omega)+\iota(\omega))k\in \cL(\cH).
  \end{equation*}
  In addition, we denote by $C^{\infty}(X,T^{*}_{\C}M)^{\Gamma}$ the space of $\Gamma$-invariant smooth 1-forms on $X$.
 Any 1-form $\omega\in C^{\infty}(\Gammab X,T^{*}(\Gammab X))$ lifts to the $\Gamma$-invariant form $\tilde{\omega}\in C^{\infty}(X,T^{*}_{\C}M)^{\Gamma}$ given by
     \begin{equation}
           \tilde{\omega}=(d\pi)^{t}\circ\omega\circ \pi,
          \label{eq:VW-groups.lifting-forms}
      \end{equation}where $(d\pi)^t:T^*(\Gammab X)\rightarrow T^*X$ is the transpose of the differential $d\pi:TX\rightarrow T(\Gammab X)$. 
      Conversely, as  the action of $\Gamma$ on $X$ is free, any $\Gamma$-invariant form $\tilde{\omega}\in C^{\infty}(X,T^{*}_{\C}X)^{\Gamma}$ descends to
      a unique 1-form $\omega\in C^{\infty}(\Gammab X,T^{*}(\Gammab X))$ obeying~(\ref{eq:VW-groups.lifting-forms}).
      In the special case $\omega=b_{1}db_{2}$ with $b_{j}\in C^{\infty}(\Gammab X)$ we find that 
      $\tilde{\omega}=\tilde{b}_{1}d\tilde{b}_{2}$. Therefore, we deduce that 
      \begin{equation}
          C^{\infty}(X,T^{*}_{\C}X)^{\Gamma}=\op{Span}\left\{ \tilde{b}_{1}d\tilde{b}_{2}; \ b_{j}\in 
          C^{\infty}(\Gammab X)\right\}.
          \label{eq:VW-groups.invariant-forms}
      \end{equation}
      
      Bearing this in mind, for  functions $a$ and $b$ in $C^{\infty}(\Gammab X)$, we have
      \begin{equation*}
          \tilde{a}[D_{\tau,k},\tilde{b}]= k\left(e^{-\tau \varphi}\tilde{a}[d,\tilde{b}]e^{\tau \varphi}+e^{\tau 
          \varphi}\tilde{a}[d^{*},\tilde{b}]e^{-\tau \varphi}\right)k. 
      \end{equation*}
      We observe that $[d,\tilde{b}]=\varepsilon(d\tilde{b})$ and $[d^{*},\tilde{b}]=-[d,\overline{\tilde{b}}]^{*}=-\varepsilon\left( 
          d\overline{\tilde{b}}\right)^{*}=\iota(d\tilde{b})$. Thus,
      \begin{equation*}
           \tilde{a}[D_{\tau,k},\tilde{b}]=k\left( \tilde{a}\varepsilon(d\tilde{b})+\tilde{a}\iota(d\tilde{b})\right)k= 
           c_{k}(\tilde{a}d\tilde{b}).
      \end{equation*}Combining this with~(\ref{eq:VW-groups.invariant-forms}) we then see that
      \begin{equation}
           \Omega_{D_{\tau,k}}^{1}(\cB)=\op{Span}\left\{ \tilde{b}_{1}d\tilde{b}_{2}; \ b_{j}\in 
          C^{\infty}(\Gammab X)\right\}= \left\{ c_{k}(\omega); \ \omega \in C^{\infty}(X,T^{*}_{\C}X)^{\Gamma}\right\}.
          \label{eq:VW-groups.OmegaDtauk}
      \end{equation}
  
   Let $E$ be a $\Gamma$-equivariant vector bundle equipped with a $\Gamma$-invariant Hermitian metric. 
   The action of $\Gamma$ on the space of smooth sections $C^{\infty}(X,E)$ is given by
   \begin{equation*}
       (\gamma \xi) (x):= \gamma\left( \xi(\gamma^{-1}x)\right), \qquad \gamma \in \Gamma, \ \xi \in C^{\infty}(X,E), \ x \in X,
   \end{equation*} 
   We let $\cE=C^{\infty}(X,E)^{\Gamma}$ be the space of smooth $\Gamma$-invariant sections of $E$. 
   This is a module over $\cB=C^{\infty}(\Gammab X)$. 
   As the action of $\Gamma$ on $E$ is free and equivariant, the vector bundle fibration  
   $\tilde{p}:E\rightarrow X$ descends to smooth vector bundle fibration $p:\Gammab E\rightarrow \Gammab X$ in such way 
   that the canonical projection $\pi^{E}:E\rightarrow \Gammab E$ is a smooth fibration obeying $\pi\circ 
   \tilde{p}=p\circ \pi^{E}$. It then follows that any section $\xi\in C^{\infty}(\Gammab X, \Gammab E)$ uniquely 
   lifts to a $\Gamma$-invariant section $\tilde{\xi}\in C^{\infty}(X,E)^{\Gamma}$ such that $\pi^{E}\circ 
   \tilde{\xi}=\xi\circ \pi$. Conversely, any $\Gamma$-invariant section $\tilde{\xi}\in 
   C^{\infty}(X,E)^{\Gamma}$ uniquely descends to a section $\xi\in C^{\infty}(\Gammab X, \Gammab E)$ obeying 
   $\xi\circ \pi=\pi^{E}\circ \tilde{\xi}$. This gives rise to a $\cB$-module isomorphism, 
   \begin{equation*}
       \cE=C^{\infty}(X,E)^{\Gamma} \simeq C^{\infty}(\Gammab X, \Gammab E). 
   \end{equation*}
   Incidentally, $\cE$ is a finitely generated projective module over $\cB$. We also note that the $\Gamma$-invariance 
   of the Hermitian metric of $E$ implies that it descends to a $\cB$-valued Hermitian metric on $\cE$ in the sense of 
   Definition~\ref{def:connection.Hermitian-metric}. 
    
    Let $\nabla^{E}$ be a $\Gamma$-equivariant Hermitian connection on $E$.  As $\nabla^{E}$ is $\Gamma$-equivariant it 
    maps $\cE$ to the $\Gamma$-invariant section space $C^{\infty}(X,T^{*}_{\C}X\otimes E)^{\Gamma}=C^{\infty}(X,T^{*}_{\C}X)^{\Gamma}\otimes_{\cB}\cE$.  
    Let $\nabla^{\cE}:\cE \rightarrow \Omega_{D_{\tau,k}}^{1}(\cB)\otimes_{\cB}\cE$ be the linear map given by
    \begin{equation*}
        \nabla^{\cE}:=(c_{k}\otimes 1_{\cE})\circ \nabla^{E}.
    \end{equation*}Note that $\nabla^{\cE}$ is well defined since~(\ref{eq:VW-groups.OmegaDtauk}) shows that $c_{k}$ maps
    $C^{\infty}(X,T^{*}_{\C}X)^{\Gamma}$ onto $\Omega_{D_{\tau,k}}^{1}(\cB)$. 
   Moreover, by arguing as in the proof of Theorem~\ref{thm:CoupledDiracOpWV}, it can be checked that $\nabla^{\cE}$ is a connection on 
    the finitely generated projective module $\cE$. We also observe that, for any 1-form $\omega \in C^{\infty}(X,T^{*}_{\C}X)$,
    \begin{equation*}
        c_{k}(\omega)^{*}=k(\varepsilon(\omega)^{*}+\iota(\omega)^{*})k=-k(\iota(\overline{\omega})+\varepsilon(\overline{\omega}))k=-c_{k}(\overline{\omega}).
    \end{equation*}Therefore, along similar lines as that of the proof of Theorem~\ref{thm:CoupledDiracOpWV}, it can be 
    shown that $\nabla^{\cE}$ is a Hermitian connection on $\cE$.

    As $\nabla^{\cE}$ is a Hermitian connection on $\cE$ we can form the operator 
    $D_{\tau,k,\nabla^{\cE}}=(kD_{\tau}k)_{\nabla^{\cE}}$, which we shall regard as an unbounded operator on 
    $L^{2}(X,\Lambda^{*}T^{*}_{\C}X)\otimes_{\cB}\cE\simeq \cH(\cE)$. We also observe that its definition~(\ref{eq:Index.Dnabla}) 
    continues to make sense on $C^{\infty}(X,\Lambda^{*}T^{*}_{\C}X)\otimes_{\cB}\cE$. Let  $\zeta\in \dom D_{\tau,k}$ and $\xi\in \cE$. 
    We write $\nabla^{E}\xi=\sum 
    \omega_{\alpha}\otimes \xi_{\alpha}$ with $\omega_{\alpha}\in C^{\infty}(X,T^{*}_{\C}X)^{\Gamma}$ and 
    $\xi_{\alpha}\in \cE$. Then 
    \begin{align}
        D_{\tau,k,\nabla^{\cE}}(\zeta\otimes \xi) & = kD_{\tau}k\zeta\otimes \xi+ \sum c_{k}(\omega_{\alpha})\zeta\otimes 
        \xi_{\alpha} \nonumber \\
      &  = ke^{-\tau \varphi}\left( d\zeta'\otimes \xi+\sum \varepsilon(\omega_{\alpha})\zeta'\otimes \xi_{\alpha} 
        \right) + ke^{\tau \varphi}\left( d^{*}\zeta''\otimes \xi+\sum \iota(\omega_{\alpha})\zeta''\otimes 
        \xi_{\alpha} 
        \right),
        \label{eq:VW-groups.DtauknablaE}
    \end{align}where we have set $\zeta'=e^{\tau\varphi}k\zeta$ and $\zeta''=e^{\tau\varphi}k\zeta$. 
    
    Bearing this in mind, the connection $\nabla^{E}$ uniquely extends to a covariant derivative $\nabla^{E}$ from 
    $C^{\infty}(X,\Lambda^{*}T^{*}_{\C}X\otimes E)$ to itself obeying~(\ref{eq:VW-groups.covariant-derivative-nablaE}). Thus $\nabla^{E}(\zeta \otimes \xi)$ 
    is equal to
    \begin{equation}
        d\zeta \otimes \xi+(-1)^{\partial \zeta}\wedge \nabla^{E}\xi = 
        d\zeta \otimes \xi+\sum (-1)^{\partial \zeta}\zeta \wedge \omega_{\alpha}\otimes \xi_{\alpha}  = d\zeta \otimes \xi+ \sum 
        \varepsilon(\omega_{\alpha})\zeta\otimes \xi_{\alpha}.
        \label{eq:VW-groups.nablaE}
    \end{equation}
    We also see that $\acou{(\nabla^{E})^{*}(\zeta\otimes \xi)}{\zeta\otimes \xi}=\acou{\zeta\otimes \xi}{\nabla^{E}(\zeta\otimes \xi)}$ is equal to 
    \begin{align*}
        \acou{\zeta\otimes \xi}{d\zeta \otimes \xi}+\sum \acou{\zeta\otimes 
        \xi}{\varepsilon(\omega_{\alpha})\zeta\otimes \xi_{\alpha}} & = \acou{\zeta}{\acoup{\xi}{\xi}d\zeta}+ 
        \acou{\zeta}{\varepsilon\left( \acoup{\xi}{\nabla^{E}\xi}\right)} \\
        & = \acou{\zeta}{d\left( \acoup{\xi}{\xi}\right)}+\acou{\zeta}{\varepsilon(\nu)\zeta},
    \end{align*}where we have set $\nu=-d\acoup{\xi}{\xi}+\acoup{\xi}{\nabla^{E}\xi}$. 
    
    As the connection $\nabla^{E}$ is Hermitian we have $\nu=-\acoup{\nabla^{E}\xi}{\xi}=-\sum 
    \overline{\omega_{\alpha}}\acoup{\xi_{\alpha}}{\xi}$. Thus, 
    \begin{equation*}
      \acou{\zeta}{\varepsilon(\nu)\zeta}=-\acou{\iota(\overline{\nu})\zeta}{\zeta}=\sum 
      \acou{\iota(\omega_{\alpha})\zeta\overline{\acou{\xi_{\alpha}}{\xi}}}{\zeta}=\sum 
      \acou{\iota(\omega_{\alpha})\zeta\otimes \xi_{\alpha}}{\zeta\otimes \xi}.
    \end{equation*}
    Note that $\acou{\zeta}{d\left( \acoup{\xi}{\xi}\right)}=\acou{d^{*}\zeta}{\acoup{\xi}{\xi}\zeta}=\acou{d^{*}\zeta\otimes 
         \xi}{\zeta\otimes \xi}$. Therefore, we get
    \begin{equation*}
        \acou{(\nabla^{E})^{*}(\zeta\otimes \xi)}{\zeta\otimes \xi}=\acou{d^{*}\zeta\otimes 
         \xi}{\zeta\otimes \xi} + \sum 
      \acou{\iota(\omega_{\alpha})\zeta\otimes \xi_{\alpha}}{\zeta\otimes \xi}.
    \end{equation*}This shows that
    \begin{equation}
       ( \nabla^{E})^{*}\zeta= d^{*}\zeta\otimes   \xi + \sum \iota(\omega_{\alpha})\zeta\otimes \xi_{\alpha}. 
               \label{eq:VW-groups.nablaE*}
    \end{equation}
    
    Combining~(\ref{eq:VW-groups.DtauknablaE}) with~(\ref{eq:VW-groups.nablaE}) and~(\ref{eq:VW-groups.nablaE*}) we deduce that 
\begin{equation*}
    D_{\tau,k,\nabla^{\cE}}(\zeta\otimes \xi)= ke^{-\tau \varphi}\nabla^{E}(\zeta'\otimes \xi)+ke^{\tau \varphi} 
    (\nabla^{E})^{*}(\zeta''\otimes \xi)= k\left( \nabla^{E}+(\nabla^{E})^{*}\right)k(\zeta\otimes \xi).
\end{equation*}
Therefore, the operator $ D_{\tau,k,\nabla^{\cE}}$ agrees with the restriction of $k\left( \nabla^{E}+(\nabla^{E})^{*}\right)k$ 
to $\dom D_{\tau,k,\nabla^{\cE}}= \dom D_{\tau,k}\otimes_{\cB}C^{\infty}(X,E)^{\Gamma}$. It then follows that any eigenvector of $ 
D_{\tau,k,\nabla^{\cE}}$ is an eigenvector of $k\left( \nabla^{E}+(\nabla^{E})^{*}\right)k$ for the same eigenvalue. 
Thus,
\begin{equation*}
     \left|\lambda_{1}\left(k\left(\nabla^{E}_{\tau}+\left(\nabla^{E}_{\tau}\right)^{*}\right)k\right)\right|\      
     \leq  \left| \lambda_{1}(D_{\tau,k,\nabla^{\cE}})\right|  .
\end{equation*}

We may apply Theorem~\ref{thm:ConformalPerST-WVIneq1} to the spectral triple $(\cB, \cH, D_{\tau})$ with Poincar\'e 
dual $(\cA_{\Gamma},\cH,D_{\tau})$ by taking $k_{1}=1$ and 
$k_{2}=k$.  We then obtain a constant $C>0$, 
independent of $k$ and of the pair $(E,\nabla^{E})$, such that 
\begin{equation*}
     \left|\lambda_{1}\left(k\left(\nabla^{E}_{\tau}+\left(\nabla^{E}_{\tau}\right)^{*}\right)k\right)\right|\      
     \leq  \left| \lambda_{1}(D_{\tau,k,\nabla^{\cE}})\right| \leq C \|k\|^{2}.
\end{equation*}
This proves the result.     
\end{proof}

\begin{remark}
A Vafa-Witten inequality for the operator $\nabla^{E}+(\nabla^{E})^{*}$ is given in~\cite{Mo:EIPDNCG}. It is mentioned as a corollary of the Vafa-Witten inequality for ordinary spectral triples established in that paper, but no details are given on the reduction to the latter. Such details are obtained by specializing the above proof to the case $k=1$. 
\end{remark}

\end{document}